\documentclass{article}

\usepackage{amsmath,amssymb,amsthm,amsxtra}
\usepackage[all]{xy}
\usepackage[hypertex]{hyperref}
\usepackage{bm}

\newtheorem{thm}{Theorem}[section]
\newtheorem{lem}[thm]{Lemma}
\newtheorem{prop}[thm]{Proposition}
\newtheorem{cor}[thm]{Corollary}
\newtheorem*{mthm}{Main Theorem}

\theoremstyle{definition}
\newtheorem{dfn}[thm]{Defintion}
\newtheorem{rem}[thm]{Remark}
\newtheorem{example}[thm]{Example}
\newtheorem*{notation}{Notation}

\newcommand{\wtil}[1]{\widetilde{#1}}

\newcommand{\qp}{\mathbb{Q}_p}
\newcommand{\zp}{\mathbb{Z}_p}
\newcommand{\cp}{\mathbb{C}_p}

\newcommand{\kpf}{K^{\mathrm{pf}}}
\newcommand{\kbar}{\overline{K}}

\newcommand{\ok}{\mathcal{O}_K}
\newcommand{\ocp}{\mathcal{O}_{\cp}}
\newcommand{\okk}{\mathcal{O}_{\mathcal{K}}}

\newcommand{\A}{\mathbb{A}}
\newcommand{\B}{\mathbb{B}}
\newcommand{\D}{\mathbb{D}}

\newcommand{\n}{\mathbb{N}}
\newcommand{\Q}{\mathbb{Q}}
\newcommand{\U}{\mathbb{U}}
\newcommand{\V}{\mathbb{V}}
\newcommand{\Z}{\mathbb{Z}}
\newcommand{\F}{\mathbb{F}}

\newcommand{\K}{\mathcal{K}}

\newcommand{\Brig}[1]{\widetilde{\mathbb{B}}_{\mathrm{rig},#1}^{\nabla+}}

\newcommand{\oo}{\mathcal{O}}

\newcommand{\cris}{\mathrm{cris}}
\newcommand{\st}{\mathrm{st}}
\newcommand{\dr}{\mathrm{dR}}
\newcommand{\HT}{\mathrm{HT}}
\newcommand{\rig}{\mathrm{rig}}
\newcommand{\rep}{\mathrm{Rep}}
\newcommand{\can}{\mathrm{can}}
\newcommand{\alg}{\mathrm{alg}}
\newcommand{\sep}{\mathrm{sep}}
\newcommand{\ur}{\mathrm{ur}}

\newcommand{\pf}{\mathrm{pf}}

\newcommand{\inc}{\mathrm{inc}}

\newcommand{\gk}{G_K}

\newcommand{\ndr}{\mathbb{N}_{\mathrm{dR}}^+}
\newcommand{\hndr}{\mathbb{N}_{\mathrm{dR}}^{\nabla+}}
\newcommand{\nrig}{\widetilde{\mathbb{N}}_{\mathrm{rig}}^{\nabla+}}

\newcommand{\om}{\hat{\Omega}}

\begin{document}

\title{The $p$-adic monodromy theorem in the imperfect residue field case}
\author{Shun Ohkubo}
\date{}
\maketitle

\begin{abstract}
Let $K$ be a complete discrete valuation field of mixed characteristic $(0,p)$ and $G_K$ the absolute Galois group of $K$. In this paper, we will prove the $p$-adic monodromy theorem for $p$-adic representations of $\gk$ without any assumption on the residue field of $K$, for example the finiteness of a $p$-basis of the residue field of $K$. The main point of the proof is a construction of $(\varphi,\gk)$-module $\nrig(V)$ for a de Rham representation $V$, which is a generalization of Pierre Colmez' $\widetilde{\n}_{\rig}^{+}(V)$. In particular, our proof is essentially different from Kazuma Morita's proof in the case when the residue field admits a finite $p$-basis.

We also give a few applications of the $p$-adic monodromy theorem, which are not mentioned in the literature. First, we prove a horizontal analogue of the $p$-adic monodromy theorem. Secondly, we prove an equivalence of categories between the category of horizontal de Rham representations of $\gk$ and the category of de Rham representations of an absolute Galois group of the canonical subfield of $K$. Finally, we compute $H^1$ of some $p$-adic representations of $G_K$, which is a generalization of Osamu Hyodo's results.
\end{abstract}

\tableofcontents

\section*{Introduction}
Let $p$ be a prime and $K$ a complete discrete valuation field of mixed characteristic $(0,p)$ with residue field $k_K$. Let $\gk$ be the absolute Galois group of $K$. When $k_K$ is perfect, Jean-Marc Fontaine defines the notions of cristalline, semi-stable, de Rham, Hodge-Tate representations for $p$-adic representations of $\gk$ (see \cite{Fon1}, \cite{Fon2} for example). The $p$-adic monodromy conjecture, which asserts that de Rham representations are potentially semi-stable, is first proved by Laurent Berger (\cite[Th\'eor\`eme~0.7]{Ber}) by using the theory of $p$-adic differential equations. Precisely speaking, Berger uses the $p$-adic local monodromy theorem for $p$-adic differential equations with Frobenius structure due to Yves Andr\'e, Zoghman Mebkhout, Kiran Kedlaya.

The notions of the above categories of representations are defined by Olivier Brinon when $k_K$ admits a finite $p$-basis (\cite{Bri}). In this case, the $p$-adic monodromy theorem is proved by Kazuma Morita (\cite[Corollary~1.2]{Mor2}). Roughly speaking, he proves the $p$-adic monodromy theorem by studying some differential equations, which are defined by Sen's theory of $\B_{\dr}$ due to Fabrizio Andreatta and Olivier Brinon \cite{AB}. In \cite{AB}, Tate-Sen formalism for a quotient $\Gamma_K$ of $\gk$ is applied to establish Sen's theory of $\B_{\dr}$, where $\Gamma_K$ is isomorphic to an open subgroup of $\zp^{\times}\ltimes \zp(1)^{J_K}$ with $J_K:=\dim_{k_K}\Omega^1_{k_K/\Z}<\infty$. To prove Tate-Sen formalism, we iteratively use analogues of the normalized trace map due to John Tate. Hence, we can not use Morita's approach in the case $J_K=\infty$.

Our main theorem in this paper is the $p$-adic monodromy theorem without any assumption on the residue field $k_K$. We also give the following applications of the $p$-adic monodromy theorem, which are not mentioned in the literature: First, we will prove a horizontal analogue of the $p$-adic monodromy theorem (Theorem~\ref{thm:horizontal}). Secondly, we will prove that the category of horizontal de Rham representations of $\gk$ is canonically equivalent to the category of de Rham representations of $G_{K_{\can}}$ (Theorem~\ref{thm:eqdR}), where $K_{\can}$ is the canonical subfield of $K$. Finally, we will calculate $H^1$ of horizontal de Rham representations under a certain condition on Hodge-Tate weights (Theorem~\ref{thm:gen}). This calculation is a generalization of calculations done by Hyodo for $\zp(n)$ with $n\in\Z$ (Theorem~\ref{thm:coh2}).

\subsection*{Statement of Main theorem}
Let $K$ and $\gk$ be as above. We do not put any assumption on the residue field $k_K$ of $K$, in particular, we may consider the case that $k_K$ is imperfect with $[k_K:k_K^p]=\infty$. In this setup, the notions of cristalline, semi-stable, de Rham, Hodge-Tate representations are also defined (see \ref{sec:hodge}). Then, our main theorem is the following:

\begin{mthm}[{The $p$-adic monodromy theorem}]\label{thm:main}
Let $V$ be a de Rham representation of $\gk$. Then, there exists a finite extension $L/K$ such that the restriction $V|_L$ is semi-stable.
\end{mthm}
Note that the converse can be easily proved by using Hilbert~90.

\subsection*{Strategy of proof}
As is mentioned above, Kazuma Morita's proof can not be generalized directly. When the residue field $k_K$ is perfect, an alternative proof of the $p$-adic monodromy theorem due to Pierre Colmez is available, which does not need the theory of $p$-adic differential equations. We will prove Main Theorem by generalizing Colmez' method. In the following, we will explain our strategy after recalling Colmez' proof in the case that $V$ is a $2$-dimensional de Rham representation. (We can prove the higher dimensional case in a similar way.) After replacing $K$ by the maximal unramified extension of $K$ and taking Tate twist of $V$, we may also assume that we have $\D_{\dr}(V)=(\B_{\dr}^+\otimes_{\qp}V)^{\gk}$ and $k_K$ is separably closed.

In this paragraph, assume that the residue field of $K$ is perfect, i.e., $k_K$ is algebraically closed. We first fix notation: Let $\widetilde{\B}_{\mathrm{rig}}^+:=\cap_{n\in\n}\varphi^n(\B_{\cris}^+)$. For $h\in\n_{>0}$ and $a\in\n$, denote $\mathbb{U}_{h,a}:=(\B_{\cris}^+)^{\varphi^h=p^a}$ and $\mathbb{U}'_{h,a}:=(\B_{\st}^+)^{\varphi^h=p^a}$. Note that we have $\mathbb{U}_{h,0}=\mathbb{U}'_{h,0}=\Q_{p^h}$, where $\Q_{p^h}$ denotes the unramified extension of $\qp$ with $[\Q_{p^h}:\qp]=h$. We will recall Colmez' proof: His proof has the following two key ingredients. One is Dieudonn\'e-Manin classification theorem over $\widetilde{\B}_{\rig}^+$. Then, he applies this theorem to construct a rank~$2$ free $\wtil{\B}_{\rig}^+$-submodule $\wtil{\n}_{\rig}^+(V)$ of $\wtil{\B}_{\rig}^+\otimes_{\qp}V$ with basis $e_1,e_2$. Moreover, $\wtil{\n}_{\rig}^+(V)$ is stable by $\varphi$ and $\gk$-actions and the following properties are satisfied:
\begin{enumerate}
\item[(i)] We have an isomorphism of $\B_{\dr}^+[\gk]$-modules
\[
\B_{\dr}^+\otimes_{\wtil{\B}_{\rig}^+}\wtil{\n}_{\rig}^+(V)\cong (\B_{\dr}^+)^2.
\]
\item[(ii)] There exist $h\in\n_{>0}$ and a $1$-cocycle
\[
C:\gk\to\begin{pmatrix}\Q_{p^h}^{\times}&\mathbb{U}_{h,a}\\0&\Q_{p^h}^{\times}\end{pmatrix};g\mapsto C_g:=\begin{pmatrix}\chi_1(g)&c_g\\0&\chi_2(g)\end{pmatrix}
\]
such that we have $g(e_1,e_2)=(e_1,e_2)C_g$ for all $g\in\gk$.
\end{enumerate}
The second key ingredient is an $H^1_g=H^1_{\st}$-type theorem for $\mathbb{U}'_{h,a}$ with $h,a\in\n_{>0}$: Let $L/K$ be a finite extension. If a $1$-cocycle $G_L\to\U'_{h,a}$ is a $1$-coboundary in $\B_{\dr}^+$, then it is a $1$-coboundary in $\U'_{h,a}$. By using these facts, Colmez prove Main Theorem as follows. When $h=0$, we may regard $C$ as a $p$-adic representation of $\gk$, which is Hodge-Tate of weights~$0$ by (i). By Sen's theorem on $\cp$-representations, $C$ has a finite image, which implies assertion. Therefore, we may assume $h>0$. By the cocycle condition of $C$, $\chi_i$ for $i=1,2$ is a character. By (i), $\chi_i$ for $i=1,2$ is Hodge-Tate with weights~$0$ as a $p$-adic representation. By Sen's theorem again, there exists a finite extension $L/K$ such that $\chi_i(G_L)=1$ for $i=1,2$. By the cocycle condition of $C$ again, $c:G_L\to\mathbb{U}_{h,a}$ is a $1$-cocycle, which is a $1$-coboundary in $\B_{\dr}^+$ by (i). By $H^1_g=H^1_{\st}$-type theorem, there exists $x\in\mathbb{U}'_{h,a}$ such that $c_g=(g-1)(x)$ for all $g\in G_L$. Therefore, $e_1,-xe_1+e_2\in\B_{\st}^+\otimes_{\wtil{\B}_{\rig}^+}\wtil{\n}^+_{\rig}(V)\subset\B_{\st}^+\otimes_{\qp}V$ forms a basis of $\D_{\st}(V|_L)$, which implies that $V|_L$ is semi-stable.

We will outline our proof of Main Theorem in the following: For simplicity, we omit some details. We first fix notation: In the imperfect residue field case, we can construct rings of $p$-adic periods $\B_{\cris}^+$, $\B_{\st}^+$ and $\B_{\dr}^+$, on which connections $\nabla$ act. Let $\B_{\cris}^{\nabla+}$ and $\B_{\st}^{\nabla+}$ be the rings of the horizontal sections of $\B_{\cris}^+$ and $\B_{\st}^+$ respectively. Let $\wtil{\B}_{\rig}^{\nabla+}:=\cap_{n\in\n}\varphi^n(\B_{\cris}^{\nabla+})$. For $h\in\n_{>0}$ and $a\in\n$, let $\mathbb{U}_{h,a}:=(\B_{\cris}^{\nabla+})^{\varphi^h=p^a}$ and $\mathbb{U}'_{h,a}:=(\B_{\st}^{\nabla+})^{\varphi^h=p^a}$. Even when $k_K$ may not be perfect, we can easily prove a generalization of Sen's theorem (Theorem~\ref{thm:Sen}) and an analogue of Colmez' Dieudonn\'e-Manin classification theorem in an appropriate setting (see \S~\ref{sec:const}). By using Dieudonn\'e-Manin theorem, we can also give a functorial construction $\nrig(V)$ for a de Rham representation $V$. Our object $\nrig(V)$ is a rank~$2$ free $\wtil{\B}_{\rig}^{\nabla+}$-submodule $\wtil{\n}_{\rig}^{\nabla+}(V)$ of $\wtil{\B}_{\rig}^{\nabla+}\otimes_{\qp}V$ with basis $e_1,e_2$. Moreover, $\nrig(V)$ is stable by $\varphi$ and $\gk$-actions and the following properties are satisfied:
\begin{enumerate}
\item[(i)] We have an isomorphism of $\B_{\dr}^+[\gk]$-modules
\[
\B_{\dr}^+\otimes_{\wtil{\B}_{\rig}^{\nabla+}}\wtil{\n}_{\rig}^{\nabla+}(V)\cong (\B_{\dr}^+)^2.
\]
\item[(ii)] There exist $h\in\n_{>0}$ and a $1$-cocycle
\[
C:\gk\to\begin{pmatrix}\Q_{p^h}^{\times}&\mathbb{U}_{h,a}\\0&\Q_{p^h}^{\times}\end{pmatrix};g\mapsto C_g:=\begin{pmatrix}\chi_1(g)&c_g\\0&\chi_2(g)\end{pmatrix}
\]
such that we have $g(e_1,e_2)=(e_1,e_2)C_g$ for all $g\in\gk$.
\end{enumerate}
Unfortunately, an analogue of the above $H^1_{\st}=H^1_g$-type theorem does not hold in the imperfect residue field case. Instead, we will use the following assertion: Let $L/K$ be a finite extension. Denote by $L^{\pf}$ a ``perfection'' of $L$, which is a complete discrete valuation field of mixed characteristic $(0,p)$ with residue field $k_L^{\pf}$. We may regard an absolute Galois group $G_{L^{\pf}}$ of $L^{\pf}$ as a closed subgroup of $\gk$. Then, our assertion is as follows (Lemma~\ref{lem:key}): If $c':G_L\to\mathbb{U}_{h,a}$ is a $1$-cocycle, which is a $1$-coboundary in $\B_{\dr}^+$, then there exists $x\in (\B_{\cris}^+)^{G_{L^{\pf}}}$ and $y\in\B_{\dr}^{\nabla+}$ such that $c_g=(g-1)(x+y)$ for all $g\in G_L$. By using these facts, we prove Main Theorem as follows. In the case $h=0$, the same proof as above is valid, hence we assume $h>0$. By the cocycle condition of $C$, $\chi_i$ for $i=1,2$ is a character, which is Hodge-Tate with weights~$0$ by (i). By a generalization of Sen's theorem, there exists a finite extension $L/K$ such that $\chi_i(G_L)=1$ for $i=1,2$. Then, by the cocycle condition of $C$, $c:G_L\to\mathbb{U}_{h,a}$ is a $1$-cocycle, which is a $1$-coboundary in $\B_{\dr}^+$. By the fact just stated above, there exists $x\in(\B_{\cris}^+)^{G_{K^{\pf}}}$ and $y\in\B_{\dr}^{\nabla+}$ such that $c_g=(g-1)(x+y)$ for $g\in G_L$. Since we have a canonical isomorphism $\nrig(V)|_{G_{K^{\pf}}}\cong \wtil{\n}_{\rig}^+(V|_{G_{K^{\pf}}})$ by functoriality, we can apply $H^1_g=H^1_{\st}$-type theorem to the $1$-cocycle $c|_{G_{L^{\pf}}}$. As a consequence, there exists $z\in\mathbb{U}'_{h,a}$ such that $c_g=(g-1)(z)$ for all $g\in G_{L^{\pf}}$. Since we have $c_g=(g-1)(y)$ for all $g\in G_{L^{\pf}}$, we have $z-y\in (\B_{\dr}^{\nabla+})^{G_{L^{\pf}}}$, which is included in $\B_{\cris}^{\nabla+}$ by a calculation. Hence, $e_1,-\{x+(y-z)+z\}e_1+e_2\in \B_{\st}^+\otimes_{\wtil{\B}_{\rig}^{\nabla+}}\nrig(V)\subset \B_{\st}^+\otimes V$ forms a basis of $\D_{\st}(V|_L)$, which implies that $V|_L$ is semi-stable.

\subsection*{Structure of paper}
In $\S~\ref{sec:pre}$, we will recall the preliminary facts used in the paper. In $\S~\ref{sec:Sen}$, we will generalize Sen's theorem on $\cp$-admissible representations, which is a special case of Main Theorem and will be used in the following. The following two sections are devoted to review rings of $p$-adic periods in the imperfect residue field case. Although the most of the results seem to be well-known, we will give proofs for the convenience of the reader. In $\S~\ref{sec:hodge}$, we will recall basic constructions and algebraic properties of rings of $p$-adic periods used in $p$-adic Hodge theory in the imperfect residue field case. In $\S~\ref{sec:hodge2}$, we will recall Galois-theoretic properties of rings of $p$-adic periods constructed in the previous section. In $\S~\ref{sec:const}$, we will construct the $(\varphi,\gk)$-modules $\nrig(V)$ for de Rham representations $V$ after Tate twist. In $\S~\ref{sec:main}$, we will prove Main Theorem combining the results proved in the previous sections. In $\S~\ref{sec:app}$, we will give applications of Main Theorem.

\subsection*{Acknowledgment}
The author is supported by the Research Fellowships of the Japan Society for the Promotion of Science for Young Scientists. The author thanks to his advisor Atsushi Shiho for reading earlier manuscripts carefully. The author also thanks to Professor Takeshi Tsuji for pointing out errors in an earlier manuscript and thanks to Professor Olivier Brinon and Kazuma Morita for helpful discussions.

\section*{Convention}
Throughout this paper, let $p$ be a prime and $K$ a complete discrete valuation field of mixed characteristic $(0,p)$. Denote the integer ring of $K$ by $\ok$ and a uniformizer of $\ok$ by $\pi_K$. Put $U^{(n)}_K:=1+\pi_K^n\ok$ for $n\in\n_{>0}$. Denote by $k_K$ the residue field of $K$. We denote by $K^{\ur}$ the $p$-adic completion of the maximal unramified extension of $K$. Denote by $e_K$ the absolute ramification index of $K$. For an extension $L/K$ of complete discrete valuation fields, we define the relative ramification index $e_{L/K}$ of $L/K$ by $e_{L/K}:=e_L/e_K$.

For a field $F$, fix an algebraic closure (resp. a separable closure) of $F$, denote it by $F^{\alg}$ or $\overline{F}$ (resp. $F^{\sep}$) and let $G_F$ be the absolute Galois group of $F$. For a field $k$ of characteristic $p$, let $k^{\mathrm{pf}}:=k^{p^{-\infty}}$ be the perfect closure in a fixed algebraic closure of $k$. Let $k^{p^{\infty}}:=\cap_{n\in\n}k^{p^n}$ be the maximal perfect subfield of $k$. Denote by $\cp$ and $\ocp$ the $p$-adic completion of $\kbar$ and its integer ring. Let $v_p$ be the $p$-adic valuation of $\cp$ normalized by $v_p(p)=1$.

We fix a system of $p$-power roots of unity $\{\zeta_{p^n}\}_{n\in\n_{>0}}$ in $\kbar$, i.e., $\zeta_p$ is a primitive $p$-the root of unity and $\zeta_{p^{n+1}}^p=\zeta_{p^n}$. Let $\chi:\gk\to\zp^{\times}$ be the cyclotomic character defined by $g(\zeta_{p^n})=\zeta_{p^n}^{\chi(g)}$ for $n\in\n_{>0}$.

For a set $S$, denote by $|S|$ the cardinality of $S$. Let $J_K$ be an index set such that we have an isomorphism $\Omega^1_{k_K/\Z}\cong k_K^{\oplus J_K}$ as $k_K$-vector spaces. In this paper, we do not assume $|J_K|<\infty$. Unless a particular mention is stated, we always fix a lift $\{t_j\}_{j\in J_K}$ of a $p$-basis of $k_K$ and its $p$-power roots $\{t_j^{p^{-n}}\}_{n\in\n,j\in J_K}$ in $\kbar$, i.e., we have $(t_j^{p^{-n-1}})^p=t_j^{p^{-n}}$ for $n\in\n_{>0}$.

For a ring $R$, denote the Witt ring with coefficients in $R$ by $W(R)$. If $R$ is characteristic $p$, then we denote the absolute Frobenius on $R$ by $\varphi:R\to R$ and also denote the ring homomorphism $W(\varphi):W(R)\to W(R)$ by $\varphi$. Denote by $[x]\in W(R)$ the Teichm\"uller lift of $x\in R$.

For a $p$-adically Hausdorff abelian group $M$ of which $p$ is a non-zero divisor, we define a $p$-adic semi-valuation of $M$ as the map $v:M\to\Z\cup\{\infty\}$ such that $v(0)=\infty$ and $v(x)=n$ if $x\in p^nM\setminus p^{n+1}M$. Note that we have the following properties
\[
v(px)=1+v(x),\ v(x+y)\ge \inf{(v(x),v(y))},\ v(x)=\infty\Leftrightarrow x=0,
\]
for $x,y\in M$. We can extend $v$ to $v:M[p^{-1}]\to\Z\cup\{\infty\}$, which we call the $p$-adic semi-valuation defined by the lattice $M$. We also call the topology induced by $v$ the $p$-adic topology defined by the lattice $M$.

Let $F$ be a non-trivial non-archimedean complete valuation field with valuation $v_F$. Assume that an $F$-vector space $V$ is endowed with a countable decreasing sequence of valuations $\{v^{(n)}:V\to\mathbb{R}\cup\{\infty\}\}_{n\in\n}$ over $F$, i.e, we have
\[
v^{(0)}(x)\ge v^{(1)}(x)\ge\dotsb,\ v^{(n)}(\lambda x)=v_F(\lambda)+v^{(n)}(x),\ v^{(n)}(x+y)\ge \inf{(v^{(n)}(x),v^{(n)}(y))}
\]
for $\lambda\in F$ and $x,y\in V$. We regard $V$ as a topological $F$-vector space, whose topology is generated by $V_r^{(n)}:=\{x\in V|v^{(n)}(x)\ge r\}$ for $n,r\in\n$. Then, we call $V$ a Fr\'echet space (over $F$) if $V$ is complete with respect to this topology (see \cite[\S~8]{Schn}). For Fr\'echet sapces $V$ and $W$, we define the completed tensor product $V\hat{\otimes}_FW$ as the inverse limit $\varprojlim_{n,r\in\n}V/V_r^{(n)}\otimes_F W/W_r^{(n)}$ (see \cite[\S~17]{Schn}).

For a multiset $\{a_i\}_{i\in I}$ of elements in $\mathbb{R}\cup\{\infty\}$, we denote $\{a_i\}_{i\in I}\to \infty$ if for all $n\in\n$, there exists a finite subset $I_n$ of $I$ such that $a_i\ge n$ for all $i\in I\setminus I_n$. Note that if $|I|<\infty$, then the condition $\{a_i\}_{i\in I}\to\infty$ is always satisfied.

In this paper, we refer to the continuous group cohomology as the group cohomology. For a profinite group $G$ and a topological $G$-module $M$, denote by $H^n(G,M)$ the $n$-th continuous group cohomology with coefficients in $M$. We also denote $H^0(G,M)$ by $M^G$. We also consider $H^q(G,M)$ for $q=0,1$ if $M$ is a (non-commutative) topological $G$-group $M$.

We denote by $\bm{e}_i\in\n^{\oplus I}$ the element whose $i$-th component is equal to $1$ and zero otherwise. We will use the following multi-index notation: Let $M$ be a monoid. For a subset $\{x_i\}_{i\in I}$ of $M$ and $\bm{n}=(n_i)_{i\in I}\in\n^{\oplus I}$, we define $\bm{x}^{\bm{n}}:=\Pi_{i\in I}x_i^{n_i}$ and $\bm{x}^{[\bm{n}]}:=\Pi_{i\in I}u_i^{n_i}/n_i!$ when it has a meaning. We denote by $|\bm{n}|$ the sum $\sum_{i\in I}n_i$ for $(n_i)_{i\in I}\in\n^{\oplus I}$. If no particular mention is stated, for an index set $I$, we denote by $\mathbf{u}_I$ or $\mathbf{v}_I$ formal variables $\{\mathrm{u}_i\}_{i\in I}$ or $\{\mathrm{v}_i\}_{i\in I}$ respectively.

For group homomorphisms $f,g:M\to N$ of abelian groups, we denote by $M^{f=g}$ the kernel of the map $f-g:M\to N$.

\section{Preliminaries}\label{sec:pre}
This preliminary section is a miscellany of basic definitions and facts, conventions, remarks used in the paper. Although we will give some proofs for convenience, the reader may skip the proofs by admitting the facts.

\subsection{Cohen ring}\label{subsec:cohen}
Let $k$ be a field of characteristic $p$. Let $C(k)$ be a Cohen ring of $k$, i.e., a complete discrete valuation ring with maximal ideal generated by $p$ and the residue field $k$. This is unique up to a canonical isomorphism if $k$ is perfect (in fact, $C(k)\cong W(k)$) and unique up to non-canonical isomorphisms in general. Denote $J_{C(k)[p^{-1}]}$ by $J$ for a while. For a lift $\{t_j\}_{j\in J}\subset C(k)$ of a $p$-basis of $k$, we regard $C(k)$ as a $\Z[T_j]_{j\in J}$-algebra by $T_j\mapsto t_j$. This morphism is formally \'etale for the $p$-adic topologies. In fact, we may replace $\Z[T_j]_{j\in J}$ by $R:=(\Z[T_j]_{j\in J})_{(p)}$. Since $C(k)/R$ is flat and $k/\F_p (T_j)_{j\in J}$ is formally \'etale for the discrete topologies, $C(k)/R$ is formally \'etale by \cite[0.19.7.1 and 0.20.7.5]{EGA}.

By the lifting property, we have an injective algebra homomorphism $C(k_K)\to\ok$, which is totally ramified of degree $e_K$. We will denote by $K_0$ the fractional field of the image of $C(k)$ in $K$. By the lifting property again, we have a lift $\varphi:\oo_{K_0}\to\oo_{K_0}$ of the absolute Frobenius of $k_K$: It is unique if $k_K$ is perfect and non-unique otherwise. An example of $\varphi$ is  $\varphi(t_j)=t_j^p$ for all $j\in J_{K_0}$. We also note that $\oo_{K_0}$ is unique if $k_K$ is perfect and non-unique otherwise. Moreover, in the case when $k_K$ is imperfect, the construction of $K_0$ cannot be functorial in the following sense: For a finite extension $L/K$, we cannot always choose $K_0\subset K$ and $L_0\subset L$ such that $K_0\subset L_0$.

Finally, note that for a given lift $\{t_j\}_{j\in J_K}\subset\ok$ of a $p$-basis of $k_K$, we can choose $\oo_{K_0}$ such that we have $\{t_j\}_{j\in J_K}\subset \oo_{K_0}$. In fact, we regard $\ok$ as a $\Z[T_j]_{j\in J_K}$-algebra by sending $T_j$ to $t_j$. Choosing a lift $\{t'_j\}_{j\in J_K}\subset C(k_K)$ of the $p$-basis $\{\bar{t}_j\}_{j\in J_K}\subset k_K$ and we regard $C(k_K)$ as a $\Z[T_j]_{j\in J_K}$-algebra by $T_j\mapsto t'_j$. Then, we lift the projection $C(k_K)\to k_K$ to a $\Z[T_j]_{j\in J_K}$-algebra homomorphism $C(k_K)\to \ok$ by the lifting property, whose image satisfies the condition. Thus, if we choose a lift $\{t_j\}_{j\in J_K}$ of a $p$-basis of $k_K$, we may always assume that we have $\{t_j\}_{j\in J_K}\subset K_0$.

\subsection{Canonical subfield}
We first recall the following two lemmas, which are proved in \cite[0.4]{Epp}. We give proofs for the reader.

\begin{lem}\label{lem:max}
Let $k$ be a field of characteristic $p$.
\begin{enumerate}
\item[(i)] The field $k^{p^{\infty}}$ is algebraically closed in $k$. In particular, the fields $(k^{p^{\infty}})^{\sep}$ and $k$ are linearly disjoint over $k^{p^{\infty}}$.
\item[(ii)] For a finite extension $k'/k^{p^{\infty}}$, we have $k'=(kk')^{p^{\infty}}$.
\end{enumerate}
\end{lem}
\begin{proof}
\begin{enumerate}
\item[(i)] The assertion follows from the fact that any algebraic extension over a perfect field is perfect.
\item[(ii)] As is mentioned in the above proof, $k'$ is perfect. By (i), we have $kk'=k\otimes_{k^{p^{\infty}}}k'$. Hence, we have $(kk')^{p^n}=k^{p^n}\otimes_{k^{p^{\infty}}}k'$ and 
\[
(kk')^{p^{\infty}}=\cap_n{(k^{p^n}\otimes_{k^{p^{\infty}}}k')}=k^{p^{\infty}}\otimes_{k^{p^{\infty}}}k'=k'.
\]
\end{enumerate}
\end{proof}

\begin{lem}\label{lem:relmax}
Let $l/k$ be an algebraic extension of fields of characteristic $p$.
\begin{enumerate}
\item[(i)] If $l/k$ is a (possibly infinite) Galois extension, then $l^{p^{\infty}}/k^{p^{\infty}}$ is also a (possible infinite) Galois extension. Moreover, the canonical map $G_{l/k}\to G_{l^{p^{\infty}}/k^{p^{\infty}}}$ is surjective.
\item[(ii)] If $l/k$ is finite, then $l^{p^{\infty}}/k^{p^{\infty}}$ is also a finite extension. Moreover, we have $[l^{p^{\infty}}:k^{p^{\infty}}]\le [l:k]$.
\end{enumerate}
\end{lem}
\begin{proof}
\begin{enumerate}
\item[(i)] We may easily reduce to the case that $l/k$ is finite Galois. Obviously any $k$-algebra endomorphism on $l$ induces a $k^{p^n}$-algebra endomorphism on $l^{p^n}$. In particular, $l^{p^n}$ and $l^{p^{\infty}}$ are $G_{l/k}$-stable. Since the Frobenius commutes with the action of $G_{l/k}$, we have $(l^{p^n})^{G_{l/k}}=(l^{G_{l/k}})^{p^n}=k^{p^n}$. By taking an intersection, we have $(l^{p^{\infty}})^{G_{l/k}}=k^{p^{\infty}}$. For $x\in l^{p^{\infty}}$, let $f(X)\in k[X]$ be the monic irreducible separable polynomial such that $f(x)=0$. Then all the solutions of $f$ belongs to $l^{p^{\infty}}$ and we have $f(X)\in (l^{p^{\infty}})^{G_{l/k}}[X]=k^{p^{\infty}}[X]$. This implies that $l^{p^{\infty}}/k^{p^{\infty}}$ is a Galois extension. The latter assertion follows from the equality $(l^{p^{\infty}})^{G_{l/k}}=k^{p^{\infty}}$.
\item[(ii)] We may assume that $l/k$ is purely inseparable or separable. If $l/k$ is purely inseparable, then $l$ is generated by finitely many elements of the form $x^{p^{-n}}$ with $n\in\n$ and $x\in k$ as a $k$-algebra. Hence we have $l^{p^n}\subset k$ for some $n$, i.e., $k^{p^{\infty}}=l^{p^{\infty}}$. Assume that $l/k$ is separable. The first assertion is reduced to the case that $l/k$ is a Galois extension, which follows from (i). Since the canonical $k$-algebra homomorphism $l^{p^{\infty}}\otimes_{k^{p^{\infty}}}k\to l$ is injective by Lemma~\ref{lem:max}~(i), we have $[l^{p^{\infty}}:k^{p^{\infty}}]\le [l:k]$.
\end{enumerate}
\end{proof}

\begin{dfn}\label{dfn:epp}
\begin{enumerate}
\item[(i)]{(cf. \cite[Theorem~2]{Hyo1})} We define the canonical subfield $K_{\can}$ of $K$ as the algebraic closure of $W(k_K^{p^{\infty}})[p^{-1}]$ in $K$.
\item[(ii)]{(cf. \cite[(0-5)]{Hyo1})} We define Condition~(H) as follows:
\[
K\ \text{contains a primitive}\ p^2\text{-th root of unity and we have}\ e_{K/K_{\can}}=1.
\]
\end{enumerate}
\end{dfn}

Note that $K_{\can}$ is a complete discrete valuation field of mixed characteristic $(0,p)$ with perfect residue field $k_K^{p^{\infty}}$. If $k_K$ is perfect, then we have $K_{\can}=K$. We also note that the restriction $\gk\to G_{K_{\can}}$ is surjective since $K_{\can}$ is algebraically closed in $K$. We will regard $G_{K_{\can}}$ as a quotient of $\gk$ in the rest of the paper.

\begin{rem}\label{rem:can}
\begin{enumerate}
\item[(i)] In \cite[Notation~2.29]{Bri}, $K_{\can}$ is denoted by $K^{\nabla}$ since $K_{\can}$ coincides with the kernel of the canonical derivation $d:K\to \hat{\Omega}^1_K$ (Proposition~\ref{prop:dlog} below).
\item[(ii)] The canonical morphism
\[
K_{\can}\otimes_{K_{\can,0}}K_0\to K
\]
is injective since we have $e_{K_0/K_{\can,0}}=1$ and $K_{\can}/K_{\can,0}$ is totally ramified. Note that we have $e_{K/K_{\can}}=1$ if and only if the above morphism is surjective.
\end{enumerate}
\end{rem}

The followings are the basic properties of the canonical subfields used in this paper.

\begin{lem}\label{lem:can}
Let $L/K$ be a finite extension.
\begin{enumerate}
\item[(i)] The fields $(K_{\can})^{\alg}$ and $K$ are linearly disjoint over $K_{\can}$.
\item[(ii)] If $L/K$ is Galois, then $L_{\can}/K_{\can}$ is also a finite Galois extension. Moreover, the canonical map $G_{L/K}\to G_{L_{\can}/K_{\can}}$ is surjective.
\item[(iii)] The field extension $L_{\can}/K_{\can}$ is finite with $[L_{\can}:K_{\can}]\le [L:K]$.
\item[(iv)] If $K'/K_{\can}$ is a finite extension, then we have $(KK')_{\can}=K'$.
\end{enumerate}
\end{lem}
\begin{proof}
\begin{enumerate}
\item[(i)] Since $K_{\can}$ is algebraically closed in $K$, we have $(K_{\can})^{\alg}\cap K=K_{\can}$, which implies the assertion.
\item[(ii)] Since $k_L^{p^{\infty}}/k_K^{p^{\infty}}$ is finite by Lemma~\ref{lem:relmax}~(ii), we have $L_{\can}=L\cap (K_{\can})^{\alg}$. Hence we have $L_{\can}\cap K=K_{\can}$. Since $L_{\can}/K_{\can}$ is algebraic, $L_{\can}$ and $K$ are linearly disjoint over $K_{\can}$ by (i). Let $x\in L_{\can}$ and $f(X)\in K_{\can}[X]$ be the monic irreducible polynomial such that $f(x)=0$. By the linearly disjointness, $f(X)$ is irreducible in $K[X]$. Since $L/K$ is Galois, all the solutions of $f(X)=0$ belong to $L\cap (K_{\can})^{\alg}=L_{\can}$. This implies that $L_{\can}/K_{\can}$ is Galois. Since we have $(L_{\can})^{G_{L/K}}=L_{\can}\cap K=K_{\can}$, we have the rest of the assertion.
\item[(iii)] The finiteness of $L_{\can}/K_{\can}$ is reduced to the case that $L/K$ is Galois, which follows from (ii). Since the canonical $K$-algebra homomorphism $L_{\can}\otimes_{K_{\can}}K\to L$ is injective by (i), we have $[L_{\can}:K_{\can}]\le [L:K]$.
\item[(iv)] The assertion follows from the inequalities
\[
[K':K_{\can}]\le [(KK')_{\can}:K_{\can}]\le [KK':K]=[K':K_{\can}],
\]
where the second inequality follows from (iii) and the last equality follows from the linear disjointness of $K$ and $K'$ over $K_{\can}$ by (i).
\end{enumerate}
\end{proof}

\begin{thm}[The complete case of Epp's theorem \cite{Epp}]\label{thm:Epp}
There exists a finite Galois extension of $K'/K_{\can}$ such that $KK'$ satisfies Condition~(H).
\end{thm}
\begin{proof}
By the original Epp's theorem, we have a finite extension $K'/K_{\can}$ such that we have $e_{KK'/K'}=1$. We have only to prove that we have $e_{KK''/K''}=1$ for any finite extension $K''/K'$. In fact, if we choose $K''$ as the Galois closure of $K'(\mu_{p^2})$ over $K$, then $K''$ satisfies the condition by Lemma~\ref{lem:can}~(iv). Since we have $KK''=(KK')\otimes_{K'}K''$ by Lemma~\ref{lem:can}~(i) and (iv), we have $e_{KK''/KK'}\le e_{K''/K'}$. By multiplying $e_{KK'}=e_{K'}$, we have $e_{KK''}\le e_{K''}$, which implies the assertion.
\end{proof}

\begin{example}[Higher dimensional local fields case]\label{ex:high}
We say that $K$ has a structure of a higher dimensional local field if $K$ is isomorphic to a finite extension over the fractional field of a Cohen ring of the field
\[
\mathbb{F}_q(\!(X_1)\!)(\!(X_2)\!)\dotsb (\!(X_d)\!)
\]
with $q=p^f$ (see \cite{Zhu} about higher dimensional local fields). In this case, $K_{\can}$ coincides with the algebraic closure of $\qp$ in $K$. In fact, we have only to prove that $k_K^{p^{\infty}}$ is a finite field. By Lemma~\ref{lem:relmax}~(ii), we may reduce to the case $k_K=\mathbb{F}_q(\!(X_1)\!)\dotsb (\!(X_d)\!)$. Then, the assertion follows from an iterative use of the following fact: If $k$ is a field of characteristic $p$, then we have $k(\!(X)\!)^{p^{\infty}}=k^{p^{\infty}}$. We will prove the fact. Obviously, the RHS is contained in the LHS. Let $f=\sum_{n\gg -\infty}a_nX^n\in k(\!(X)\!)^{p^{\infty}}$ with $a_n\in k$. Since we have $f\in k(\!(X)\!)^p$, we have $a_n=0$ if $p\nmid n$ and $a_n\in k^p$ otherwise. By repeating this argument, we have $a_n=0$ for $n\neq 0$ and $f=a_0\in k^{p^{\infty}}$.
\end{example}

\subsection{Canonical derivation}
\begin{dfn}[{cf. \cite[\S~4]{Hyo1}}]
Let $q\in\n$. For a complete discrete valuation ring $R$ with mixed characteristic $(0,p)$, let $\hat{\Omega}^q_R:=\varprojlim_n\Omega^q_{R/\Z}/p^n\Omega^q_{R/\Z}$ and $d:R\to \hat{\Omega}^1_R$ the canonical derivation. Let $\hat{\Omega}^q_{R[p^{-1}]}:=\hat{\Omega}^q_{R}[p^{-1}]$ for $q\in\Z$, $d:R[p^{-1}]\to\hat{\Omega}^1_{R[p^{-1}]}$ the canonical derivation and $d_q:\om^q_{R[p^{-1}]}\to\om^{q+1}_{R[p^{-1}]}$ the morphism induced by the exterior derivation, which satisfies a usual formula $d_q(\lambda\omega)=\lambda d_q\omega+(-1)^q\omega\wedge d\lambda$ for $\lambda\in K$ and $\omega\in\om^q_K$. We endow $\hat{\Omega}^q_{R[p^{-1}]}$ with the $p$-adic topology defined by the lattice $\mathrm{Im}(\hat{\Omega}^q_R\stackrel{\can.}{\to}\hat{\Omega}^q_{R[p^{-1}]})$. Obviously, the above derivation $d_q$ is continuous.

For $q\in\Z_{<0}$, we put $\om_{R[p^{-1}]}^q:=0$ as a matter of convention.
\end{dfn}

The followings are the basic properties of the canonical derivations used in the following of the paper.

\begin{lem}[{\cite{Hyo1}}]\label{lem:dif}
Let $q\in\n$.
\begin{enumerate}
\item[(i)] We have the $\oo_{K_0}$-linear isomorphism
\[
\om^q_{\oo_{K_0}}\cong \varprojlim_n((\oo_{K_0}/p^n\oo_{K_0})\otimes_{\Z}\wedge^q_{\Z}(\Z^{\oplus J_K}));dt_{j_1}\wedge\dotsb\wedge dt_{j_q}\mapsto \bm{e}_{j_1}\wedge\dotsb\wedge \bm{e}_{j_q}.
\]
In particular, $\om^q_{\oo_{K_0}}/(p^n)$ is a free $\oo_{K_0}/(p^n)$-module.
\item[(ii)] We have a canonical isomorphism
\[
(\wedge^q_K\om^1_K)\sphat\to\om^q_K.
\]
\item[(iii)] Let $L$ be a finite extension over the completion of an unramified extension of $K$. Then, we have a canonical isomorphism
\[
L\otimes_K\hat{\Omega}^q_K\to\hat{\Omega}^q_L.
\]
\end{enumerate}
\end{lem}
\begin{proof}
The assertions~(i) and (ii) follow from \cite[Lemma~(4.4), Remark~3]{Hyo1} respectively. We will prove the assertion~(iii). We first note the following fact: Let $\alpha:M\to M'$ be a morphism of $\oo_{L}$-modules whose kernel and cokernel are killed by $p^c$ for $c\in\n$. Then, for any $\oo_{L}$-module $M''$, the kernel and cokernel of the morphism $\mathrm{id}\otimes \alpha:M''\otimes_{\oo_{L}}M\to M''\otimes_{\oo_{L}}M'$ are killed by $p^{2c}$. In fact, if $\alpha$ is injective or surjective, then the cokernel and kernel are killed by $p^c$ by the calculation of $\mathrm{Tor}^{\bullet}_{\oo_{L}}$. The general case follows easily from these cases by writing $\alpha$ as a composition of an injection and a surjection. In particular, the kernel and cokernel of $\alpha^{\otimes q}: M^{\otimes q}\to M'^{\otimes q}$ are killed by $p^{2qc}$. In fact, it follows from the following decomposition and induction on $q$:
\[\xymatrix{
M^{\otimes (q+1)}=M\otimes_{\oo_{L}}M^{\otimes q}\ar[r]^(.63){\mathrm{id}\otimes \alpha^{\otimes q}}&M\otimes_{\oo_{L}}M'^{\otimes q}\ar[r]^(.37){\alpha\otimes\mathrm{id}}&M'\otimes_{\oo_{L}}M'^{\otimes q}=M'^{\otimes (q+1)}.
}\]

The canonical exact sequence (\cite[\S~3.4, footnote]{Sch})
\[\xymatrix{
0\ar[r]&\oo_{L}\otimes_{\ok}\Omega^1_{\ok/\Z}\ar[r]&\Omega^1_{\oo_{L}/\Z}\ar[r]&\Omega^1_{\oo_{L}/\ok}\ar[r]&0
}\]
induces the exact sequence
\[\xymatrix{
\Omega^1_{\oo_{L}/\ok}[p^n]\ar[r]&\oo_{L}\otimes_{\ok}\Omega^1_{\ok/\Z}/(p^n)\ar[r]^(.6){\alpha_n}&\Omega^1_{\oo_{L}/\Z}/(p^n)\ar[r]&\Omega^1_{\oo_{L}/\ok}/(p^n)\ar[r]&0,
}\]
where $\Omega^1_{\oo_{L}/\ok}[p^n]$ denotes the kernel of the multiplication by $p^n$ on $\Omega^1_{\oo_{L}/\ok}$. Fix $c\in\n$ such that $p^c\Omega^1_{\oo_{L}/\ok}=0$. Then, the kernel and cokernel of $\alpha_n$ are killed by $p^c$. Denote by $\mathcal{Q}_n$ and $Q_n$ the kernel of the canonical maps
\begin{align*}
\otimes^q_{\oo_{L}}(\oo_{L}\otimes_{\ok}\Omega^1_{\ok/\Z}/(p^n))&\to\wedge_{\oo_{L}}^q(\oo_{L}\otimes_{\ok}\Omega^1_{\ok/\Z}/(p^n)),\\
\otimes^q_{\oo_{L}}\Omega^1_{\oo_{L}/\Z}/(p^n)&\to\Omega^q_{\oo_{L}/\Z}/(p^n).
\end{align*}
We consider the commutative diagram
\[\xymatrix{
\oo_{L}\otimes_{\ok}(\otimes^q_{\ok}\Omega^1_{\ok/\Z}/(p^n))\ar[r]^{\can.}_{\cong}\ar@{->>}[d]^{\can.}&\otimes^q_{\oo_{L}}(\oo_{L}\otimes_{\ok}\Omega^1_{\ok/\Z}/(p^n))\ar@{->>}[d]^{\can.}\ar[r]^(.6){\alpha_n^{\otimes q}}&\otimes^q_{\oo_{L}}\Omega^1_{\oo_{L}/\Z}/(p^n)\ar@{->>}[d]^{\can.}\\
\oo_{L}\otimes_{\ok}\Omega^q_{\ok/\Z}/(p^n)\ar[r]^(.45){\can.}_(.45){\cong}&\bigwedge_{\oo_{L}}^q(\oo_{L}\otimes_{\ok}\Omega^1_{\ok/\Z}/(p^n))\ar[r]^(.63){\wedge^q \alpha_n}&\Omega^q_{\oo_{L}/\Z}/(p^n).
}\]
We have only to prove that the kernel and cokernel of $\wedge^q \alpha_n$ are killed by $p^{3qc}$. Indeed, if this is true, then we decompose the canonical map $\alpha_n^q:\oo_{L}\otimes_{\ok}\Omega^q_{\ok/\Z}/(p^n)\to\Omega^q_{\oo_{L}/\Z}/(p^n)$ into the following exact sequences:
\begin{gather*}
\xymatrix{
0\ar[r]&\ker{\alpha^q_n}\ar[r]^(.3){\inc.}&\oo_{L}\otimes_{\ok}\Omega^q_{\ok/\Z}/(p^n)\ar[r]^(.7){\alpha_n^q}&\mathrm{Im}{\,\alpha_n^q}\ar[r]&0,}\\
\xymatrix{
0\ar[r]&\mathrm{Im}{\,\alpha_n^q}\ar[r]^(.4){\inc.}&\Omega^q_{\oo_{L}/\Z}/(p^n)\ar[r]^(.57){\mathrm{pr.}}&\mathrm{cok}{\,\alpha_n^q}\ar[r]&0.
}
\end{gather*}
By passing to limits, we obtain the following exact sequences
\begin{gather*}
\xymatrix{
0\ar[r]&\varprojlim_n\ker{\alpha^q_n}\ar[r]^(.45){\inc.}&\oo_{L}\otimes_{\ok}\om^q_{\ok}\ar[r]^(.53){\can.}&\varprojlim_n\mathrm{Im}{\,\alpha_n^q}\ar[r]^(.47){\delta}&\varprojlim^1_n\ker{\alpha^q_n},}\\
\xymatrix{
0\ar[r]&\varprojlim_n\mathrm{Im}{\,\alpha_n^q}\ar[r]^(.57){\inc.}&\om^q_{\oo_{L}}\ar[r]^(.38){\mathrm{pr.}}&\varprojlim_n\mathrm{cok}{\,\alpha_n^q}.
}
\end{gather*}
Since $\ker{\alpha^q_n}$ and $\mathrm{cok}{\,\alpha_n^q}$ are killed by $p^{3qc}$, $\varprojlim_n\ker{\alpha_n^q}$ and $\varprojlim^1_n\ker{\alpha_n^q}$, $\varprojlim{\mathrm{cok}{\,\alpha_n^q}}$ are also killed by $p^{3qc}$ (\cite[Proposition~2.7.4]{NSW}). Hence, the kernel and cokernel of the canonical map $\oo_{L}\otimes_{\ok}\om^q_{\ok}\to\om^q_{\oo_{L}}$ are killed by $p^{3qc}$ and $p^{6qc}$ respectively. By inverting $p$, we obtain the assertion.

Note that the kernel and cokernel of $\alpha^{\otimes q}_n$ are killed by $p^{2qc}$ by the assertion that we have shown in the first paragraph of the proof. By snake lemma, it suffices to prove that the cokernel of the map $\alpha^{\otimes q}_n:\mathcal{Q}_n\to Q_n$ is killed by $p^{qc}$. The $\oo_{L}$-module $Q_n$ is generated by the elements of the form $x:=x_1\otimes\dotsb \otimes x_q$ with $x_i\in \Omega^1_{\oo_{L}/\Z}/(p^n)$ such that $x_i=x_j$ for some $i\neq j$. Since the cokernel of $\alpha_n$ is killed by $p^c$, there exist $y_1,\dotsc,y_q\in \oo_{L}\otimes_{\ok}\Omega^1_{\oo_{L}/\Z}/(p^n)$ such that $p^cx_i=\alpha_n(y_i)$ and $y_i=y_j$. Hence we have $p^{qc}x=(p^cx_1)\otimes\dotsb\otimes (p^cx_q)=\alpha_n^{\otimes q}(y_1\otimes\dotsb \otimes y_q)$ and $y_1\otimes\dotsb \otimes y_q\in\mathcal{Q}_n$, which implies the assertion.
\end{proof}

\begin{rem}\label{rem:dif}
If we have $[k_K:k_K^p]=p^d<\infty$, then we have $\dim_K\om^q_K=\binom{d}{q}<\infty$ for $q\in\n$ by Lemma~\ref{lem:dif}. In particular, the canonical derivation $d$ is $K_{\can}$-linear since the restriction $d|_{K_{\can}}$ factors through $\om^1_{K_{\can}}=0$ by functoriality.
\end{rem}

\begin{dfn}
Fix a lift $\{t_j\}_{j\in J_K}\subset \oo_{K_0}$ of a $p$-basis of $k_K$. By using Lemma~\ref{lem:dif}~(i), $dx$ for $x\in\oo_{K_0}$ is uniquely written of the form $\sum_{j\in J_K}dt_j\otimes \partial_j(x)$, where we have $\{\partial_j(x)\}_{j\in J_K}\subset \oo_{K_0}$ such that $\{v_p(\partial_j(x))\}_{j\in J_K}\to\infty$. Note that $\{\partial_j\}_{j\in J_K}$ are mutually commutative derivations of $\oo_{K_0}$ by the formula $d_1\circ d=0$. We also note that $\partial_j$ is continuous since we have the inequality $v_p(\partial_j(x))\ge v_p(x)$ for $x\in\oo_{K_0}$, which we can check by taking modulo~$p$.
\end{dfn}

The following is another characterization of the canonical subfields.
\begin{prop}[{\cite[Proposition~2.28]{Bri}}]\label{prop:dlog}
We have the exact sequence
\[\xymatrix{
0\ar[r]&K_{\can}\ar[r]^(.55){\inc.}&K\ar[r]^(.45)d&\hat{\Omega}^1_K.
}\]
\end{prop}
\begin{proof}
We first reduce to the case $K=K_0$. In the case that $K$ satisfies Condition~(H), we obtain the exact sequence by applying $K_{\can}\otimes_{K_{\can,0}}$ to the exact sequence for $K=K_0$ by Remark~\ref{rem:can}~(ii) and Lemma~\ref{lem:dif}~(iii). In general case, we choose a finite Galois extension $K'/K_{\can}$ such that $KK'$ satisfies Condition~(H) by Epp's theorem~\ref{thm:Epp}. Since we have $K'\otimes_{K_{\can,0}}K=KK'$ by Lemma~\ref{lem:can}~(i) and $(\om^1_{KK'})^{G_{K'/K_{\can}}}=\om^1_K$ by Lemma~\ref{lem:dif}~(iii), the assertion follows from Galois descent.

We will prove the assertion in the case $K=K_0$. We may replace $K_{\can},K$ and $\om^1_K$ by $\oo_{K_{\can}},\ok$ and $\om^1_{\oo_K}$ respectively. Notation is as above. Let $\varphi$ be a Frobenius on $\ok$ uniquely determined by $\varphi(t_j)=t_j^p$ for $j\in J_K$. Let $\varphi_*:\om^1_K\to\om^1_K$ be a Frobenius induced by $\varphi$. Since we have $d\circ \varphi=\varphi_*\circ d$, by a simple calculation, we have $\partial_j\circ \varphi=pt_j^{p-1}\varphi\circ \partial_j$, i.e., $(t_j\partial_j)\circ\varphi =p\varphi\circ (t_j\partial_j)$ for $j\in J_K$.

The ring $\varphi(\ok)$ is a complete discrete valuation ring of mixed characteristic $(0,p)$ and we may regard its residue field as $k_K^p$. Let $\Lambda:=\{0,\dotsc,p-1\}^{\oplus J_K}$. Since the image of $\{\bm{t}^{\bm{n}}\}_{\bm{n}\in\Lambda}$ in $k_K$ forms a $k_K^p$-basis of $k_K$, by approximation, every element $x\in\ok$ is uniquely written in the form $x=\sum_{\bm{n}\in \Lambda}\varphi(a_{\bm{n}})\bm{t}^{\bm{n}}$, where $a_{\bm{n}}\in\ok$ such that $\{v_p(a_{\bm{n}})\}_{\bm{n}\in\Lambda}\to \infty$. We will claim that for $\varphi^n(x)\in\ker{d}$ with $n\in\n$ and $x\in\ok$, we have $x\in\varphi(\ok)$. Since the Frobenius $\varphi_*$ on $\om^1_{\ok}$ is injective by Lemma~\ref{lem:dif}~(i) and the commutativity $d\circ \varphi=\varphi_*\circ d$, we may assume $n=0$. By definition, we have $\partial_j(x)=0$ for all $j\in J_K$. By a simple calculation, we have
\[
t_j\partial_j(x)=\sum_{\bm{n}\in \Lambda}(t_j\partial_j)\circ\varphi(a_{\bm{n}})\bm{t}^{\bm{n}}+\sum_{\bm{n}\in \Lambda}\varphi(a_{\bm{n}})t_j\partial_j(\bm{t}^{\bm{n}})=\sum_{\bm{n}\in \Lambda}\varphi(pt_j\partial_j(a_{\bm{n}})+n_ja_{\bm{n}})\bm{t}^{\bm{n}}.
\]
Hence, we have $a_{\bm{n}}=-n_j^{-1}pt_j\partial_j(a_{\bm{n}})$ if $n_j\neq 0$. Therefore, for $\bm{n}\in\Lambda\setminus\{\mathbf{0}\}$, we have $v_p(a_{\bm{n}})\ge v_p(a_{\bm{n}})+1$ by taking $v_p$, i.e., $a_{\bm{n}}=0$, which implies the claim. By using the claim, if we have $x\in\ker{d}$, then we have $x\in \cap_{n\in\n}\varphi^n(\ok)$. Since the complete discrete valuation ring $\cap_{n\in\n}\varphi^n(\ok)$ is absolutely unramified with residue field $k_K^{p^{\infty}}$, the inclusion $\oo_{K_{\can}}\subset\cap_{n\in\n}\varphi^n(\ok)$ is an equality by approximation, which implies the assertion.
\end{proof}

\subsection{A spectral sequence of continuous group cohomology}\label{subsec:sp}
The following lemma is a basic fact when we calculate a continuous Galois cohomology, whose coefficient is an inverse limit of $p$-adic Banach spaces with surjective transition maps. For example, we need it later when we calculate cohomology of $\B_{\dr}^+$-modules.

\begin{lem}[{cf. \cite[Theorem~2.7.5]{NSW}}]\label{lem:sp}
Let $G$ be a profinite group and $\{M_n\}_{n\in\n}$ be an inverse system of continuous $G$-modules (each $M_n$ may not be discrete) such that the transition map $M_{n+1}\to M_n$ admits a continuous section as a topological space for all $n\in\n$. Let $M_{\infty}:=\varprojlim M_n$ be a continuous $G$-module with the inverse limit topology. Then, we have a canonical exact sequence
\[\xymatrix{
0\ar[r]&\varprojlim^1_nH^{q-1}(G,M_n)\ar[r]&H^q(G,M_{\infty})\ar[r]&\varprojlim_n H^q(G,M_n)\ar[r]&0
}\]
for all $q\in\n$, where $\varprojlim^{\bullet}$ is the derived functor of $\varprojlim$ in the category of inverse systems of abelian groups with directed set $\n$.
\end{lem}
\begin{proof}

Let $\mathcal{C}^{\bullet}_{\infty}:=\mathcal{C}_{\mathrm{cont}.}^{\bullet}(G,M_{\infty})$ (resp. $\mathcal{C}^{\bullet}_{n}:=\mathcal{C}^{\bullet}_{\mathrm{cont}.}(G,M_n)$) be the continuous cochain complex of $G$ coefficients in $M_{\infty}$ (resp. $M_n$). Then, $\{\mathcal{C}^{\bullet}_{n}\}_{n\in\n}$ forms an inverse system of cochain complexes and we have $\mathcal{C}_{\infty}^{\bullet}=\varprojlim_n\mathcal{C}_n^{\bullet}$. Moreover, the transition maps of the inverse system $\{\mathcal{C}^{\bullet}_n\}_{n\in\n}$ are surjective by the existence of continuous sections, in particular, $\{\mathcal{C}^{\bullet}_n\}_{n\in \n}$ satisfies the Mittag-Leffler condition. Then, the assertion follows from \cite[Variant in pp.84]{Wei}.
\end{proof}

\subsection{Hyodo's calculations of Galois cohomology}\label{subsec:Hyo}
We will recall Hyodo's calculations of Galois cohomology. For $n\in\Z$, denote by $\zp(n)$ the $n$-th Tate twist of $\zp$. For a $\zp[\gk]$-module $V$, let $V(n):=V\otimes_{\zp}\zp(n)$.

\begin{thm}[{\cite[Theorem~1]{Hyo1}}]\label{thm:coh}
For $n\in\n$ and $q\in\Z$, we have canonical isomorphisms
\[
H^n(\gk,\cp(q))\cong\begin{cases}0& q\neq n,n-1\\\hat{\Omega}^q_K&otherwise.\end{cases}
\]
\end{thm}

We will generalize the following theorem as an application of Main Theorem in \S~\ref{sec:app}.

\begin{thm}\label{thm:coh2}
\begin{enumerate}
\item[(i)] (\cite[Theorem~2]{Hyo1}) We have the exact sequence
\[\xymatrix{
0\ar[r]&H^1(G_{K_{\can}},\zp(1))\ar[r]^{\mathrm{Inf}}&H^1(\gk,\zp(1))\ar[r]^{\mathrm{can}.}&H^1(\gk,\cp(1)).
}\]
\item[(ii)] (\cite[Theorem~(0-2)]{Hyo2}) If $k_K$ is separably closed, then we have the canonical isomorphism
\[
\mathrm{Inf}:H^1(G_{K_{\can}},\zp(n))\cong H^1(\gk,\zp(n))
\]
for $n\in\Z_{\neq 1}$.
\end{enumerate}
\end{thm}

\subsection{Closed subgroups of $\gk$}
Let $L$ be an algebraic extension of $K$ in $\cp$. Let $\hat{L}^{\alg}$ be the algebraic closure of $\hat{L}$ in $\cp$. Let $M$ be a finite extension of $\hat{L}$ and choose a polynomial $f(X)\in\hat{L}[X]$ such that $M\cong \hat{L}[X]/(f(X))$. Let $f_0(X)\in L[X]$ be a polynomial such that the $p$-adic valuation of the coefficients of $f-f_0$ is large enough. Then, we have $M\cong\hat{L}[X]/(f_0(X))$ by Krasner's lemma. In particular, the algebraic extension $(M\cap L^{\alg})/L$ is dense in $M$. Hence, we have a canonical morphism of profinite groups $G_L\to G_{\hat{L}}$, which is an isomorphism with inverse $G_{\hat{L}}\to G_L;g\mapsto g|_{L^{\alg}}$. In the following of the paper, we will identify $G_L$ with $G_{\hat{L}}$ and we also regard $G_{\hat{L}}$ as a closed subgroup of $\gk$.

\subsection{Perfection}\label{subsec:pf}
For a subset $J$ of $J_K$, we denote the $p$-adic completion of the field $\cup_{n\in\n}K(\{t_j^{p^{-n}}\}_{j\in J})$ by $K_J$. Then, $K_J$ is a complete discrete valuation field of mixed characteristic $(0,p)$ with $e_{K_J/K}=1$ and its residue field is isomorphic to $\cup_{n\in\n}k_K(\{\bar{t}_j^{p^{-n}}\}_{j\in J})$. We also denote $K_{J_K}$ by $K^{\pf}$, which is refereed as a perfection of $K$ since the residue field $k_{K^{\pf}}\cong k_K^{\pf}$ of $K^{\pf}$ is perfect. Since we may assume that $\{t_j\}_{j\in J_K}$ is contained in $K_0$ (\S\S~\ref{subsec:cohen}), we may regard $(K_0)_J$ as $(K_J)_0$, which is denoted by $K_{J,0}$ for simplicity.

Let $\mathcal{P}(J_K)$ be the subsets of $J_K$ consisting of subsets $J\in J_K$ such that $J_K\setminus J$ is finite.  Note that we have $[k_{K_J}:k_{K_J}^p]=p^{|J_K\setminus J|}<\infty$ for $J\in \mathcal{P}(J_K)$ since $\{\bar{t}_j\}_{j\in J_K\setminus J}$ forms of a $p$-basis of $k_{K_J}$. We regard $\mathcal{P}(J_K)$ as an inverse system with respect to the reverse inclusion. Then, we have
\[
K\cong\varprojlim_{J\in \mathcal{P}(J_K)}K_J=\bigcap_{J\in \mathcal{P}(J_K)}K_J,
\]
i.e., $K$ is represented by an inverse limit of complete discrete valuation fields, whose residue fields admit a finite $p$-basis. In fact, if we put $J_K$ an well-order $\precsim$ by the axiom of choice, then for $J\in\mathcal{P}(J_K)$, the subset
\[
\mathcal{E}_J:=\{1\}\cup\left\{t_{j_1}^{a_1p^{-n_1}}\dotsb t_{j_m}^{a_mp^{-n_m}}\Bigl|
\begin{aligned}
&j_1\precnsim \dotsb\precnsim j_m\in J,0<a_{j_i}<p^{n_{j_i}}\in\n_{>0}\\
&(p,a_{j_i})=1\ \text{for}\ 1\le i\le m\in\n_{>0}
\end{aligned}
\right\}
\]
of $K_J$ forms a basis of $K_J$ as a $K$-Banach space. If $J_1\subset J_2$ are in $\mathcal{P}(J_K)$, then we have $\mathcal{E}_{J_1}\subset \mathcal{E}_{J_2}$ and the assertion follows from the fact $\{1\}=\cap_{J\in\mathcal{P}(J_K)}\mathcal{E}_J$.

\subsection{$G$-regular ring}\label{subsec:reg}
We will recall basic facts about $G$-regular rings. See \cite[\S~1]{Fon2} for detail.

Let $E$ be a topological field and $G$ a topological group. A finite dimensional $E$-vector space $V$ is an $E$-representation of $G$ if $V$ has a continuous $E$-linear action of $G$. We denote the category of $E$-representations of $G$ by $\rep_E{G}$. We call $B$ an $(E,G)$-ring if $B$ is a commutative $E$-algebra and $G$ acts on $B$ as $E$-algebra automorphisms. Let $B$ be an $(E,G)$-ring. For $V\in\rep_E{G}$, let $D_B(V):=(B\otimes_E{V})^{G}$ and we will call the following canonical homomorphism the comparison map:
\[
\alpha_B(V):B\otimes_{B^G}D_B(V)\to B\otimes_E{V}.
\]

We say that an $(E,G)$-ring $B$ is $G$-regular if the followings are satisfied:
\begin{enumerate}
\item[] ($G\cdot R_1$) The ring $B$ is reduced;
\item[] ($G\cdot R_2$) For all $V \in \rep_{E}{G}$, $\alpha_B(V)$  is injective;
\item[] ($G\cdot R_3$) Every $G$-stable $E$-line in $B$ is generated by an invertible element of $B$.
\end{enumerate}
Here, a $G$-stable $E$-line in $B$ means one-dimensional $G$-stable $E$-vector space in $B$. The condition~($G\cdot R_3$) implies that $B^G$ is a field. We say that $V\in\rep_E{G}$ is $B$-admissible if $\alpha_B(V)$ is an isomorphism. We denote the category of $B$-admissible $E$-representations of $G$ by $\rep_{B/E}{G}$, which is a Tannakian full subcategory of $\rep_{E}{G}$ (\cite[Proposition~1.5.2]{Fon2}).

\begin{notation}
We will call an object of $\rep_{\qp}\gk$ a $p$-adic representation of $\gk$. For a $(\qp,\gk)$-ring $B$, we denote $\rep_{B/\qp}\gk$ by $\rep_B\gk$ if no confusion arises.
\end{notation}

We recall the basic facts about $G$-regular rings used in the following of the paper.

\begin{lem}\label{lem:inj}
Let $B$ be a field and $G$ a group, which acts on $B$ as ring automorphisms. Let $M$ be a finite dimensional $B$-vector space with semi-linear $G$-action. Then, the canonical map
\[
B\otimes_{B^G}M^G\to M
\]
is injective. In particular, we have $\dim_{B^G}M^G\le\dim_{B}M$.
\end{lem}
\begin{proof}
Suppose that the assertion does not hold. Let $n\in\n$ be the minimum integer such that there exists $n$ elements $v_1,\dotsc,v_n\in M^G$, which is linearly independent over $B^G$ but not over $B$. Let $\sum_{1\le i\le n}\lambda_iv_i=0$ be a non-trivial relation with $\lambda_i\in B$. Since $B$ is a field, we may assume that $\lambda_1=1$. Then, we have
\[
0=(g-1)\left(\sum_{1\le i\le n}\lambda_iv_i\right)=\sum_{1<i\le n}(g(\lambda_i)-\lambda_i)v_i.
\]
Hence, we have $\lambda_i\in B^G$ by assumption, which is a contradiction.
\end{proof}

\begin{example}[{\cite[Proposition~1.6.1]{Fon2}}]\label{ex:reg}
All $(E,G)$-rings, which are fields, are $G$-regular. In fact, we have only to verify $(G\cdot R_2)$, which follows by applying the above lemma to $M:=B\otimes_{E}V$.
\end{example}

\begin{lem}[{\cite[Proposition~1.4.2]{Fon2}}]\label{lem:reg}
Let $B$ be a $G$-regular $(E,G)$-ring and $V$ an $E$-representation of $G$. Then, we have $\dim_{B^G}D_B(V)\le\dim_E{V}$. Moreover, the equality holds if and only if $V$ is $B$-admissible.
\end{lem}

\begin{lem}[{\cite[Proposition~1.6.5]{Fon2}}]\label{lem:rel}
Let $B$ be a $G$-regular $(E,G)$-ring and $B'$ an $E$-subalgebra of $B$ stable by $G$. Assume that $B'$ satisfies $(G\cdot R_3)$ and the canonical map $B^G\otimes_{{B'}^G}B'\to B$ is injective. Then, $B'$ is a $G$-regular $(E,G)$-ring. Moreover, if $V\in \rep_{E}G$ is $B'$-admissible, then $V$ is $B$-admissible and the canonical map
\[
B^G\otimes_{{B'}^G}D_{B'}(V)\to D_B(V)
\]
is an isomorphism.
\end{lem}

\begin{lem}[{\cite[Corollaire~1.6.6]{Fon2}}]\label{lem:frac}
Let $B'$ be an integral domain, which is an $(E,G)$-ring. If the fractional field $B$ of $B'$ satisfies $(G\cdot R_3)$ and $B'^G=B^G$, then $B'$ is $\gk$-regular.
\end{lem}

\begin{rem}[Restriction]\label{rem:res}
Let $B$ be a $G$-regular $(E,G)$-ring and $H$ a subgroup of $G$ such that $B$ is $H$-regular as an $(E,H)$-ring. If $V\in \rep_E{G}$ is $B$-admissible, then $V|_H$ is also $B$-admissible in $\rep_{E}H$. Moreover, we have a canonical isomorphism $B^H\otimes_{B^G}D_B(V)\cong D_B(V|_H)$. In fact, the admissibility of $V$ implies that we have the comparison isomorphism $B\otimes_{B^G}D_B(V)\cong B\otimes_{E}V$ as $B[\gk]$-modules. By taking $H$-invariant, we have $B^H\otimes_{B^G}D_B(V)\cong D_B(V|_H)$. In particular, we have $\dim_{B^H}D_B(V|_H)=\dim_{B^G}D_B(V)=\dim_{E}V$, which implies the $B$-admissibility of $V|_H$ by Lemma~\ref{lem:reg}.
\end{rem}

\section{A generalization of Sen's theorem}\label{sec:Sen}
The aim of this section is to prove the following generalization of Sen's theorem on $\cp$-representations.

\begin{thm}[{cf. \cite[Corollary in (3.2)]{Sen}}]\label{thm:Sen}
Let $V\in \rep_{\qp}\gk$. The followings are equivalent:
\begin{enumerate}
\item[(i)] There exists a finite extension $L$ over the maximal unramified extension of $K$ such that $G_L$ acts on $V$ trivially;
\item[(ii)] $V$ is $\cp$-admissible;
\item[(iii)] $V|_{K^{\pf}}$ is $\cp$-admissible as an object of $\rep_{\qp}G_{K^{\pf}}$.
\end{enumerate}
\end{thm}

\begin{lem}\label{lem:fin}
Let $E$ be a field of characteristic $0$ and $\rho:U^{(n)}_{\qp}\ltimes\Pi_{i\in I}p^{n_i}\zp\to GL_r(E)$ a group homomorphism with $n,r\in \n_{>0}$ and $(n_i)_{i\in I}\in \n^I$, where the action of $U^{(n)}_{\qp}$ on $\Pi_{i\in I}p^{n_i}\zp$ is given by scalar multiplications. If $\mathit{ker}{\rho}$ contains an open subgroup of $U^{(n)}_{\qp}$, then the image of $\rho$ is finite.
\end{lem}
\begin{proof}
By shrinking $U^{(n)}_{\qp}$, we may assume that $\ker{\rho}$ contains $U^{(n)}_{\qp}$. Also, we may assume that $E$ is algebraically closed. Let $x_0:=1+p^n\in U^{(n)}_{\qp},\bm{x}\in \Pi_{i\in I}p^{n_i}\zp$. By the fact that $\ker{\rho}$ is a normal subgroup of $U^{(n)}_{\qp}\ltimes\Pi_{i\in I}p^{n_i}\zp$ and a simple calculation, we have $(1,\bm{x})^{-1}(x_0,\mathbf{0})(1,\bm{x})(x_0^{-1},\mathbf{0})=(1,(x_0-1)\bm{x})=(1,p^n\bm{x})\in \ker{\rho}$. In particular, $\ker{\rho}$ contains $U^{(n)}_{\qp}\ltimes\Pi_{i\in I}p^{n+n_i}\zp$ as a normal subgroup. By taking the quotient of $U$ by this subgroup, $\rho$ factors through a group homomorphism $\bar{\rho}:(\Z/p^n\Z)^I\to GL_r(E)$.

To prove the assertion, it suffices to prove that for any finite subset $S$ of $\mathrm{Im}{\bar{\rho}}$, we have $|S|\le p^{nr}$. For $g\in \mathrm{Im}{\bar{\rho}}$, $g$ is conjugate to a diagonal matrix, whose diagonal entries are in $\mu_{p^n}(E)$ since the order of $g$ divides $p^n$. Since the elements of $S$ are mutually commutative, $S$ is simultaneously diagonalizable. Hence, up to conjugate, $S$ is contained in the set $\{\mathrm{diag}(a_1,\dotsc,a_r)|a_i\in \mu_{p^n}(E)\}$, whose order is $p^{nr}$.
\end{proof}

\begin{proof}[Proof of Theorem~\ref{thm:Sen}]
The implication $(i)\Rightarrow(ii)$ follows from Hilbert~90 and $(ii)\Rightarrow (iii)$ follows from Remark~\ref{rem:res}. We will prove $(iii)\Rightarrow(i)$. Note that if $k_K$ is perfect, then the assertion is Sen's theorem (\cite[Corollary in (3.2)]{Sen}).

By replacing $K$ by a finite extension of $K^{\ur}$, we may assume that $k_K$ is separably closed and $K$ satisfies Condition~(H). In this case, the assertion to prove is that $G_K$ acts on $V$ via a finite quotient. Since the residue field $k_K^{\pf}$ of $K^{\pf}$ is algebraically closed, $G_{K^{\pf}}=G_{K^{\mathrm{geo}}}$ acts on $V$ via a finite quotient by Sen's theorem, where $K^{\mathrm{geo}}:=\cup_{n\in\n}K(\{t_j^{p^{-n}}\}_{j\in J_K})$. Hence, there exists a finite extension $L/K$ such that $G_{LK^{\mathrm{geo}}}$ acts on $V$ trivially. In particular, if we put $K_{\infty}:=K^{\mathrm{geo}}(\mu_{p^{\infty}})$, then $G_{LK_{\infty}}$ acts on $V$ trivially. In the following, we regard $V$ as a $p$-adic representation of $G_{LK_{\infty}/L}$. Take a basis of $V$ and let $\rho':G_{LK_{\infty}/L}\to GL_r(\qp)$ be the corresponding matrix presentation of $V$ with $r:=\dim_{\qp}V$. We have only to prove that the image of $\rho'$ is finite.

Since $K$ satisfies Condition~(H), we have an isomorphism $G_{K_{\infty}/K}\cong U_{\qp}^{(n_0)}\ltimes \zp^{J_K}$, where $n_0\in\n_{>1}$ satisfies $G_{K(\mu_{p^{\infty}})/K}\cong U^{(n_0)}_{\qp}$ via the cyclotomic character and $U_{\qp}^{(n_0)}$ acts on $\zp^{J_K}$ as the scalar multiplications (see \cite[\S~1]{Hyo1} for detail). We have $G_{LK_{\infty}/LK^{\mathrm{geo}}}\le \ker{\rho'}\trianglelefteq_c G_{LK_{\infty}/L}$. By using the restriction map $\mathrm{Res}^{LK_{\infty}}_{K_{\infty}}$ and the above isomorphism, we may regard these groups as subgroups of $U^{(n_0)}_{\qp}\ltimes\zp^{J_K}$. Since $G_{LK_{\infty}/L}$ is an open subgroup of $G_{K_{\infty}/K}$, there exists $n\in \n$ and $(n_j)_{j\in J_K}\in \n^{J_K}$ such that $G_{LK_{\infty}/L}$ contains $U:=U^{(n)}_{\qp}\ltimes\Pi_{j\in J_K}p^{n_j}\zp$ as an open subgroup. Since $G_{LK_{\infty}/LK^{\mathrm{geo}}}$ is an open subgroup of $G_{K_{\infty}/K^{\mathrm{geo}}}\cong G_{K(\mu_{p^{\infty}})/K}\cong U^{(n_0)}_{\qp}\cong \zp$, $\ker{\rho'}$ contains an open subgroup of $U^{(n)}_{\qp}$. Therefore, the group homomorphism $\rho:=\rho'|_U:U\to GL_r(\qp)$ satisfies the assumption of Lemma~\ref{lem:fin}, hence, the image of $\rho$ is finite. Since $U$ is open in $G_{LK_{\infty}/L}$, we obtain the assertion.
\end{proof}

\section{Basic construction of rings of $p$-adic periods}\label{sec:hodge}
Throughout this section, let $\K$ be a closed subfield of $\cp$, whose value group $v_p(\K^{\times})$ is discrete. We will recall the construction of rings of $p$-adic periods
\[
\A_{\inf,\cp/\K},\ \B_{\cris,\cp/\K},\ \B_{\st,\cp/\K},\ \B_{\dr,\cp/\K},\ \B_{\HT,\cp/\K}
\]
due to Fontaine \cite{Fon1}, which is functorial with respect to $\cp$ and $\K$. We also recall abstract algebraic properties of these rings as in \cite{Bri}. Although we do not assume $\K=K$, standard techniques of proofs in the case $\K=K$, which are developed in \cite{Fon1,Bri}, can be applied to our situation.

\subsection{Universal pro-infinitesimal thickenings}\label{subsec:proinf}
\begin{dfn}[{\cite[\S~1]{Fon1}}]
A $p$-adically formal pro-infinitesimal $\okk$-thickening of $\oo_{\cp}$ is a pair $(D,\theta_D)$ where
\begin{enumerate}
\item[$\bullet$] $D$ is an $\okk$-algebra,
\item[$\bullet$] $\theta_D:D\to\oo_{\cp}$ is a surjective $\okk$-algebra homomorphism such that $D$ is $(p,\ker{\theta})$-adic Hausdorff complete.
\end{enumerate}
\end{dfn}
Obviously, $p$-adically formal $\okk$-thickenings of $\oo_{\cp}$ form a category.

\begin{thm}[{\cite[Th\'eor\`eme~1.2.1]{Fon1}}]
The category of $p$-adically formal pro-infinitesimal $\okk$-thickenings of $\oo_{\cp}$ admits a universal object, i.e., an initial object.
\end{thm}

Such an object is unique up to a canonical isomorphism and we denote it by $(\A_{\inf,\cp/\K},\theta_{\cp/\K})$. Note that $\A_{\inf,\cp/\K}$ is functorial with respect to $\cp$ and $\K$. We recall the construction. Let $R_{\cp}:=\varprojlim_{x\mapsto x^p}\ocp/p\ocp$ be the perfection of the ring $\ocp/p\ocp$. We have the canonical isomorphism
\[
\varprojlim_{x\mapsto x^p}\ocp\to R_{\cp};(x^{(n)})_{n\in\n}\mapsto (x^{(n)}\mathrm{mod}\ p\ocp)_{n\in\n},
\]
where the addition and the multiplication of the LHS are given by
\begin{gather*}
((x^{(n)})+(y^{(n)}))_n=\lim_m(x^{(n+m)}+y^{(n+m)})^{p^m},\\
(x^{(n)})\cdot (y^{(n)})=(x^{(n)}y^{(n)}).
\end{gather*}
Let $\theta_{\cp/\qp}:W(R_{\cp})\to\oo_{\cp};\sum_{n\in\n}p^n[x_n]\mapsto \sum_{n\in\n}p^nx_n^{(0)}$ be a surjective algebra homomorphism. Let $\theta_{\cp/\K}:\okk\otimes_{\Z}W(R_{\cp})\to \oo_{\cp}$ be the linear extension of $\theta_{\cp/\qp}$. Then, $\A_{\inf,\cp/\K}$ is the Hausdorff completion of $\okk\otimes_{\Z}W(R_{\cp})$ with respect to the $(p,\ker{\theta_{\cp/\K}})$-adic topology. We will give an explicit description of $\A_{\inf,\cp/\K}$ in the following: Note that the description, together with an isomorphism $W(R_{\cp})\cong \A_{\inf,\cp/\qp}$ (Remark~\ref{rem:bc}), immediately implies that $\A_{\inf,\cp/\K}$ is an integral domain (at least) when we have $\K=\K_0$.

We define $\wtil{t}_j:=(t_j,t_j^{p^{-1}},\dotsc)\in R_{\cp}$ and $u_j:=t_j-[\wtil{t}_j]\in \ker{\theta_{\cp/\K_0}}$. Let $v_{\inf,\cp/\K}$ be the $p$-adic semi-valuation of $\A_{\inf,\cp/\K}$. We define an $\A_{\inf,\cp/\qp}$-algebra
\[
\A_{\inf,\cp/\qp}\{\mathbf{u}_{J_{\K}}\}:=\left\{\sum_{\bm{n}\in\n^{\oplus J_{\K}}}a_{\bm{n}}\mathbf{u}^{\bm{n}}\ \Bigl|\ a_{\bm{n}}\in\A_{\inf,\cp/\qp},\{v_{\inf,\cp/\qp}(a_{\bm{n}})\}_{|\bm{n}|=n}\to \infty\ \mathrm{for}\ \mathrm{all}\ n\in\n\right\}
\]
and we extend $\theta_{\cp/\qp}$ to a surjective $\A_{\inf,\cp/\qp}$-algebra homomorphism $\vartheta_{\cp/\K}:\A_{\inf,\cp/\qp}\{\mathbf{u}_{J_{\K}}\}\to\oo_{\cp}$ by $\vartheta_{\cp/\K}(\mathrm{u}_j)=0$. Then, $(\A_{\inf,\cp/\qp}\{\mathbf{u}_{J_{\K}}\},\vartheta_{\cp/\K})$ is a $p$-adically formal $\zp$-pro-infinitesimal thickening of $\oo_{\cp}$. We have a canonical $\A_{\inf,\cp/\qp}$-algebra homomorphism $\iota_{\inf,\cp/\K}:\A_{\inf,\cp/\qp}\{\mathbf{u}_{J_{\K}}\}\to\A_{\inf,\cp/\K};\mathbf{u}^{\bm{n}}\mapsto \bm{u}^{\bm{n}}$.

\begin{lem}\label{lem:inf}
If we assume $\K=\K_0$, then $\iota_{\inf,\cp/\K}$ is an isomorphism. In particular, we have
\[
v_{\inf,\cp/\K}(x)=\inf_{\bm{n}\in\n^{\oplus J_{\K}}}v_{\inf,\cp/\qp}(a_{\bm{n}})
\]
for $x=\sum_{\bm{n}\in\n^{\oplus J_{\K}}}a_{\bm{n}}\bm{u}^{\bm{n}}$ with $a_{\bm{n}}\in\A_{\inf,\cp/\qp}$.
\end{lem}

\begin{proof}
Denote $\mathcal{A}=\A_{\inf,\cp/\qp}\{\mathbf{u}_{J_{\K}}\}$ and $\vartheta=\vartheta_{\cp/\K}$. We regard $\oo_{\K}$ as a $\Z[T_j]_{j\in J_{\K}}$-algebra as in \S\S~\ref{subsec:cohen}. We also regard $\mathcal{A}$ as a $\Z[T_j]_{j\in J_{\K}}$-algebra by $T_j\mapsto [\wtil{t}_j]+{\mathrm{u}_j}$. Then, by the lifting property, we can lift the canonical $\okk$-algebra structure on $\mathcal{A}/(p,\ker{\vartheta})\cong\ocp/(p)$ to an $\okk$-algebra structure on $\mathcal{A}\cong\varprojlim_n\mathcal{A}/(p,\ker{\vartheta})^n$:
\[\xymatrix{
\oo_{\K}\ar[r]^{\can.}\ar@{-->}[rd]^{\exists !}&\ocp\\
\Z[T_j]_{j\in J_{\K}}\ar[r]^{\mathrm{str}.}\ar[u]^{\mathrm{str}.}&\mathcal{A}\ar[u]_{\vartheta}.
}\]
By this structure map, we may regard $\mathcal{A}$ as a pro-infinitesimal $\okk$-thickening of $\ocp$. By universality, we have only to prove that $\iota_{\inf,\cp/\K}$ is an $\okk$-algebra homomorphism. Let $\alpha:\okk\to\A_{\inf,\cp/\K}$ be the composition of the structure map $\okk\to\mathcal{A}$ and $\iota_{\inf,\cp/\K}$. Since $\iota_{\inf,\cp/\K}$ commutes with the projections $\vartheta$ and $\theta_{\cp/\K}$, we have the commutative diagram
\[\xymatrix{
\oo_{\K}\ar[r]^{\can.}\ar[rd]^{\alpha}&\ocp\\
\Z[T_j]_{j\in J_{\K}}\ar[r]^{\mathrm{str}.}\ar[u]^{\mathrm{str}.}&\A_{\inf,\cp/\K}\ar[u]_{\theta_{\cp/\K}},
}\]
where the horizontal structure map is given by $T_j\mapsto t_j$. By this diagram and the lifting property, $\alpha$ coincides with the structure map $\okk\to\A_{\inf,\cp/\K}$ modulo $(p,\ker{\theta_{\cp/\K}})^n$ for all $n\in\n$. Since $\A_{\inf,\cp/\K}$ is $(p,\ker{\theta_{\cp/\K}})$-adically Hausdorff complete, $\alpha$ coincides with the structure map $\okk\to\A_{\inf,\cp/\K}$, which implies the assertion.
\end{proof}

For general $\K$, we have the following:
\begin{lem}\label{lem:bc}
\begin{enumerate}
\item[(i)] The canonical map
\[
\A_{\inf,\cp/\K}\to\A_{\inf,\cp/\K^{\ur}}
\]
is an isomorphism.
\item[(ii)] If $\mathcal{L}/\K$ is a finite extension with $[k_{\mathcal{L}}:k_{\K}]_{\sep}=1$, then the canonical map
\[
\oo_{\mathcal{L}}\otimes_{\oo_{\K}}\A_{\inf,\cp/\K}\to\A_{\inf,\cp/\mathcal{L}}
\]
is an isomorphism.
\item[(iii)] Let $\mathcal{L}$ be a finite extension of the $p$-adic completion of an unramified extension of $\K$. Then, the canonical map
\[
\A_{\inf,\cp/\K}[p^{-1}]/(\ker{\theta_{\cp/\K}})^n\to\A_{\inf,\cp/\mathcal{L}}[p^{-1}]/(\ker{\theta_{\cp/\mathcal{L}}})^n
\]
is an isomorphism for all $n\in\n$.
\end{enumerate}
\end{lem}
\begin{proof}
\begin{enumerate}
\item[(i)] The assertion is equivalent to say that the category of $p$-adically formal $\oo_{\K}$-pro-infinitesimal thickening of $\oo_{\cp}$ is equivalent to the category of $p$-adically formal $\oo_{\K^{\ur}}$-pro-infinitesimal thickening of $\oo_{\cp}$. Let $(D,\theta_D)$ be a $p$-adically formal $\oo_{\K}$-pro-infinitesimal thickening of $\oo_{\cp}$. Then, we have only to prove that there exists a unique $\oo_{\K}$-algebra homomorphism $\oo_{\K^{\ur}}\to D$ such that $\theta_D$ is an $\oo_{\K^{\ur}}$-algebra homomorphism. By d\'evissage, we may replace $D$ by $D/(p,\ker{\theta_D})^n$ with $n\in\n$. Since $\theta_D$ induces an isomorphism $D/(p,\ker{\theta_D})\cong\ocp/(p)$ and $\oo_{\K^{\ur}}/\oo_{\K}$ is $p$-adically formally \'etale, the assertion follows from the commutative diagram
\[\xymatrix{
\oo_{\K^{\ur}}\ar[r]^{\can.}\ar@{-->}[rd]^{\exists !}&\oo_{\cp}/(p)\\
\oo_{\K}\ar[u]^{\can.}\ar[r]^(.3){\mathrm{str}.}&D/(p,\ker{\theta}_D)^n.\ar[u]_{(\theta_D)_*}
}\]
\item[(ii)] By assumption, the canonical map $\oo_{\mathcal{L}}\otimes_{\ok}\oo_{\K^{\ur}}\to\oo_{\mathcal{L}^{\ur}}$ is an isomorphism. By using this fact and (i), we may replace $\K$, $\mathcal{L}$ by $\K^{\ur}$, $\mathcal{L}^{\ur}$ respectively. In particular, we may consider the case that $k_{\K}$ is separably closed, where the condition $[k_{\mathcal{L}}:k_{\K}]_{\sep}=1$ is always satisfied. By faithfully flat descent, the assertion is reduced to the case that $\mathcal{L}/\K$ is Galois. Since $\mathcal{L}/\K$ is a solvable extension by the solvability of the inertia subgroup \cite[Exercise~2, \S~2, Chapter~II]{FV}, we may assume that $\mathcal{L}/\K$ has a prime degree.

By universality, we have only to prove that the LHS is a $p$-adically formal $\oo_{\mathcal{L}}$-pro-infinitesimal thickening of $\oo_{\cp}$. Hence, it suffices to verify that $\oo_{\mathcal{L}}\otimes_{\oo_{\K}}\A_{\inf,\cp/\K}$ is $(p,I)$-adically Hausdorff complete, where $I$ denotes the kernel of the canonical map $1\otimes\theta_{\cp/\K}:\oo_{\mathcal{L}}\otimes_{\oo_{\K}}\A_{\inf,\cp/\K}\to\ocp$. Since we have an isomorphism of $\A_{\inf,\cp/\K}$-modules $\oo_{\mathcal{L}}\otimes_{\oo_{\K}}\A_{\inf,\cp/\K}\cong (\A_{\inf,\cp/\K})^{[\mathcal{L}:\K]}$, we have only to prove that the topologies of $\oo_{\mathcal{L}}\otimes_{\oo_{\K}}\A_{\inf,\cp/\K}$ defined by the ideals $(p,I)$ and $(p,I')$ are equivalent, where $I'$ denotes the ideal of $\oo_{\mathcal{L}}\otimes_{\oo_{\K}}\A_{\inf,\cp/\K}$ generated by $\ker{(\theta_{\cp/\K}:\A_{\inf,\cp/\K}\to\ocp)}$. By definition, we have $(p,I')\subset (p,I)$. We have only to prove that we have $I^n\subset (\pi_{\K}\otimes 1,I')$ for some $n\in\n$ since $\pi_{\K}^{e_{\K}}$ is divided by $p$.

In the following, for $x\in\ocp$, we denote by $\wtil{x}$ any element $\wtil{x}\in R_{\cp}$ such that $\wtil{x}^{(0)}=x$. Since we have $\pi_{\K}\otimes 1-1\otimes [\wtil{\pi}_{\K}]\in I'$, we have $(\pi_{\K}\otimes 1,1\otimes [\wtil{\pi}_{\K}])\subset (\pi_{\K}\otimes 1,I')$. Note that if $x\in\oo_{\mathcal{L}}$ is a cyclic base, i.e., $1,x,\dotsc,x^{[\mathcal{L}:\K]-1}$ is an $\oo_{\K}$-basis of $\oo_{\mathcal{L}}$, then we have $I\subset (x\otimes 1-1\otimes [\wtil{x}],I')$. Hence, we have only to prove an existence of a cyclic base $x\in\oo_{\mathcal{L}}$ satisfying $(x\otimes 1-1\otimes [\wtil{x}])^n\in (\pi_{\K}\otimes 1,I')$ for some $n\in\n$. In the case $[\mathcal{L}:\K]=e_{\mathcal{L}/\K}$, $\pi_{\mathcal{L}}$ is a cyclic base of $\oo_{\mathcal{L}}$ and we have $(\pi_{\mathcal{L}}\otimes 1-1\otimes [\wtil{\pi}_{\mathcal{L}}])^{2e_{\mathcal{L}/\K}}\in (\pi_{\K}\otimes 1,1\otimes [\wtil{\pi}_{\K}])$. Otherwise, we have $[\mathcal{L}:\K]=[k_{\mathcal{L}}:k_{\K}]_{\mathrm{insep}}=p$. If we choose $x\in\oo_{\mathcal{L}}$ such that $x^p\equiv a\mod{\pi_{\K}\oo_{\mathcal{L}}}$ with $a\in\oo_{\K}$, then $x$ is a cyclic base of $\oo_{\mathcal{L}}$ by Nakayama's lemma. Moreover, we have $(x\otimes 1-1\otimes [\wtil{x}])^p\equiv a\otimes 1-1\otimes [\wtil{a}]\mod{(\pi_{\K}\otimes 1,1\otimes [\wtil{\pi}_{\K}])}$ and $a\otimes 1-1\otimes[\wtil{a}]\in I'$, which implies the assertion.
\item[(iii)] We denote the map by $i$ and we will construct the inverse. By replacing $\K$ and $\mathcal{L}$ by $\K^{\ur}$ and $\mathcal{L}^{\ur}$, we may assume $[k_{\mathcal{L}}:k_{\K}]_{\sep}=1$. By (ii), we identify $\A_{\inf,\cp/\mathcal{L}}$ with $\oo_{\mathcal{L}}\otimes_{\oo_{\K}}\A_{\inf,\cp/\K}$. By a similar argument as in the proof of (i), we have a unique $\K$-algebra homomorphism $j:\mathcal{L}\to \A_{\inf,\cp/\K}[p^{-1}]/(\ker{\theta_{\cp/\K}})^n$ such that $\theta_{\cp/\K}:\A_{\inf,\cp/\K}[p^{-1}]/(\ker{\theta_{\cp/\K}})^n\to\cp$ is an $\mathcal{L}$-algebra homomorphism since $\mathcal{L}/\K$ is \'etale. Hence, we have the $\A_{\inf,\cp/\K}$-algebra homomorphism $j\otimes \mathrm{id}:\A_{\inf,\cp/\mathcal{L}}[p^{-1}]/(\ker{\theta_{\cp/\mathcal{L}}})^n\to\A_{\inf,\cp/\K}/(\ker{\theta_{\cp/\K}})^n$. By construction, we have $(j\otimes \mathrm{id})\circ i=\mathrm{id}$. To prove $i\circ (j\otimes \mathrm{id})=\mathrm{id}$, we have only to prove that $i\circ (j\otimes \mathrm{id})$ is an $\mathcal{L}$-algebra homomorphism, which follows from the uniqueness of $j$.
\end{enumerate}
\end{proof}

\begin{rem}\label{rem:bc}
We may identify $\A_{\inf,\cp/\qp}$ with $W(R_{\cp})$ (\cite[1.2.4~(e)]{Fon1}) and the kernel of $\theta_{\cp/\qp}$ is principal by \cite[2.3.3]{Fon1}. Moreover, if we have $\K=\K_0$ and $k_{\K}$ is perfect, then the canonical map $\A_{\inf,\cp/\qp}\to\A_{\inf,\cp/\K}$ is an isomorphism (\cite[1.2.4~(e)]{Fon1}). Note that we have no canonical choice of an embedding $W(k_K^{\alg})[p^{-1}]\to \cp$ when $k_K$ is imperfect, since different perfections of $K$ induce different embeddings. Thus, we can not endow $\A_{\inf,\cp/\qp}$ with a canonical $W(k_K^{\alg})$-algebra structure induced by that of $\A_{\inf,\cp/W(k_K^{\alg})[p^{-1}]}$ via the above isomorphism as in the perfect residue field case.
\end{rem}

\subsection{$\B_{\dr}$ and $\B_{\HT}$}\label{subsec:dr}
We define $\B_{\dr,\cp/\K}^+:=\varprojlim_n\A_{\inf,\cp/\K}[p^{-1}]/(\ker{\theta_{\cp/\K}})^n$ and
\[
t:=\log{([\varepsilon])}=\sum_{n\in\n_{>0}}(-1)^{n-1}\frac{([\varepsilon]-1)^n}{n}\in\B_{\dr,\cp/\qp}^+
\]
with $\varepsilon:=(1,\zeta_p,\zeta_{p^2},\dotsc)\in R_{\cp}$. We also define $\B_{\dr,\cp/\K}:=\B_{\dr,\cp/\K}^+[t^{-1}]$. We denote the projection $\B_{\dr,\cp/\K}^+\to\cp$ by $\theta_{\cp/\K}$ again. Then, $\B_{\dr,\cp/\K}^+$ is a Hausdorff complete local ring with maximal ideal $\ker{\theta_{\cp/\K}}$. Moreover, $\B_{\dr,\cp/\K}$ is an integral domain. In fact, by the following explicit description of $\B_{\dr,\cp/\K}$, it reduces to the fact that $\B_{\dr,\cp/\qp}$ is a field (Remark~\ref{rem:dr}~(ii) below).

We define the canonical topology on $\B_{\dr,\cp/\K}^+$ as follows. We regard $\A_{\inf,\cp/\K}[p^{-1}]/(\ker{\theta_{\cp/\K}})^n$ as a $p$-adic Banach space, whose lattice is given by the image of $\A_{\inf,\cp/\K}$. Then, we endow $\B_{\dr,\cp/\K}^+$ with the inverse limit topology, which is a Fr\'echet complete $\K$-algebra. We also endow $\B_{\dr,\cp/\K}$ with a limit of Fr\'echet (LF) topology by regarding $\B_{\dr,\cp/\K}$ as the direct limit of $\B_{\dr,\cp/\K}^+$ with respect to the multiplication by $t^{-1}$. Let $v_{\dr,\cp/\K}^{(n)}$ be the semi-valuation of $\B_{\dr,\cp/\K}^+$ induced by the $p$-adic semi-valuation of $\B_{\dr,\cp/\K}^+/(\ker{\theta_{\cp/\K}})^n$ defined by the lattice $\mathrm{Im}(\A_{\inf,\cp/\K}\stackrel{\can.}{\to}\B_{\dr,\cp/\K}^+/(\ker{\theta_{\cp/\K}})^n)$. Obviously, the semi-valuations $\{v_{\dr,\cp/\K}^{(n)}\}_{n\in\n}$ are decreasing.

We will give an explicit description of $\B_{\dr,\cp/\K}^+$. Let
\[
\B_{\dr,\cp/\qp}^+\{\mathbf{u}_{J_{\K}}\}:=\left\{\sum_{\bm{n}\in\n^{\oplus J_K}}a_{\bm{n}}\mathbf{u}^{\bm{n}}\Bigl|a_{\bm{n}}\in\B_{\dr,\cp/\qp}^+,\{v_{\dr,\cp/\qp}^{(r)}(a_{\bm{n}})\}_{|\bm{n}|=n}\to\infty\ \text{for all}\ n,r\in\n\right\}
\]
be a $\B_{\dr,\cp/\qp}^+$-algebra. Then, the canonical $\B^+_{\dr,\cp/\qp}$-algebra homomorphism
\[
\iota_{\dr,\cp/\K}:\B_{\dr,\cp/\qp}^+\{\mathbf{u}_{J_{\K}}\}\to\B_{\dr,\cp/\K}^+;\mathbf{u}^{\bm{n}}\mapsto \bm{u}^{\bm{n}}
\]
is an isomorphism. In fact, by Remark~\ref{rem:dr}~(ii) below, we may reduce to the case $\K=\K_0$. In this case, the assertion follows from the explicit description of $\A_{\inf,\cp/\K}$.

Let $\mathrm{Fil}^n\B_{\dr,\cp/\K}^+$ be the closed ideal of $\B_{\dr,\cp/\K}^+$ generated by the ideal $(\ker{\theta_{\cp/\K}})^n$ for $n\in\n$. We endow $\B_{\dr,\cp/\K}$ with the decreasing filtration defined by $\mathrm{Fil}^n\B_{\dr,\cp/\K}:=\sum_{i+j=n}t^i\mathrm{Fil}^j\B_{\dr,\cp/\K}^+$. Denote the graded $\cp$-algebra associated to the filtration by $\B_{\HT,\cp/\K}$. We also denote by $v_j$ the image of $u_j/t$ in $\B_{\HT,\cp/K,0}$ for $j\in J_K$. Since the filtration is compatible with the multiplication by $t$, i.e., $t^m\mathrm{Fil}^n\B_{\dr,\cp/\K}=\mathrm{Fil}^{n+m}\B_{\dr,\cp/\K}$, we have an isomorphism $\B_{\HT,\cp/\K}\cong\oplus_{n\in\Z}\B_{\HT,\cp/\K,0}t^n$.

For $n\in\n$, let
\[
\cp\{\mathbf{v}_{J_{\K}}\}_n:=\left\{\sum_{\bm{n}\in\n^{\oplus J_{\K}}:|\bm{n}|=n}a_{\bm{n}}\mathbf{v}^{\bm{n}}\Bigl|a_{\bm{n}}\in\cp,\{v_p(a_{\bm{n}})\}_{\bm{n}}\to\infty\right\}
\]
and $\cp\{\mathbf{v}_{J_{\K}}\}:=\oplus_{n\in\n}\cp\{\mathbf{v}_{J_{\K}}\}_n$ be a $\cp$-algebra. We have the $\cp$-algebra homomorphisms
\[
\iota_{\HT,\cp/\K,0}:\cp\{\mathbf{v}_{J_{\K}}\}\to\B_{\HT,\cp/\K,0};\mathbf{v}^{\bm{n}}\mapsto \bm{v}^{\bm{n}},
\]
which is an isomorphism. In fact, the assertion reduces to the case $\K=\K_0$ by Remark~\ref{rem:dr}~(ii) below. Then, the assertion follows from the above explicit description of $\B_{\dr,\cp/\K}^+$ and the formula of the semi-valuation $v^{(n)}_{\dr,\cp/\K}$ (Remark~\ref{rem:dr}~(iii) below). By this description, $\B_{\HT,\cp/\K}$ is an integral domain.

\begin{rem}\label{rem:dr}
\begin{enumerate}
\item[(i)] (The perfect residue field case) Assume that $k_{\K}$ is perfect. Then, we have a canonical isomorphism $\B_{\clubsuit,\cp/\qp}\to\B_{\clubsuit,\cp/\K}$ for $\clubsuit\in\{\dr,\HT\}$. Moreover, $\B_{\dr,\cp/\qp}$ is a complete discrete valuation field of equal characteristic $0$ with valuation ring $\B_{\dr,\cp/\qp}^+$, $t$ is a uniformizer and the residue field is $\cp$. We also have an isomorphism $\B_{\HT,\cp/\qp}\cong\oplus_{n\in\Z}\cp t^n$. In fact, the first assertion follows from Remark~\ref{rem:bc} and the latter assertion reduces to the perfect residue field case by regarding $\cp$ as the $p$-adic completion of $(K^{\pf})^{\alg}$ (\cite[1.5.1]{Fon1}).
\item[(ii)] (Invariance) The above structures on $\B_{\dr,\cp/\K}^+$ (ring structure, filtration, topology) are invariant under a finite extension and an unramified extension. As a consequence, we may regard $\B_{\dr,\cp/\K}^+$ as a $\K^{\alg}$-algebra and a similar invariance for $\B_{\HT,\cp/\K}$ as a graded $\cp$-algebra also holds. As for a filtered ring, the invariance follows from Lemma~\ref{lem:bc}~(iii). To prove the rest of the assertion, we have only to prove that for an unramified extension or a finite extension $\mathcal{L}/\K$, the $p$-adic semi-valuations $v^{(n)}_{\dr,\cp/\K}$ and $v^{(n)}_{\dr,\cp/\mathcal{L}}$ are equivalent for all $n\in\n$. The unramified case follows from Lemma~\ref{lem:bc}~(i). In the other case, let $\Lambda^{(n)}_{\K}$ (resp. $\Lambda_{\mathcal{L}}^{(n)}$) be the image of $\A_{\inf,\cp/\K}$ (resp. $\A_{\inf,\cp/\mathcal{L}}$) in $\B_{\dr,\cp/\K}^+/(\ker{\theta_{\cp/\K}})^n$. By taking the maximal unramified extension of $\K$ in $\mathcal{L}$, we may assume that $\mathcal{L}/\K$ satisfies the assumption in Lemma~\ref{lem:bc}~(ii). Since $\A_{\inf,\cp/\mathcal{L}}$ is a finite $\A_{\inf,\cp/\K}$-module by Lemma~\ref{lem:bc}~(ii), there exists $m\in\n$ such that $p^m\Lambda^{(n)}_{\mathcal{L}}\subset \Lambda_{\K}^{(n)}$ by Lemma~\ref{lem:bc}~(iii). Since we have $\Lambda^{(n)}_{\K}\subset \Lambda_{\mathcal{L}}^{(n)}$ by definition, the two $p$-adic topologies induced by the lattices $\Lambda^{(n)}_{\K}$ and $\Lambda_{\mathcal{L}}^{(n)}$ respectively are equivalent, which implies the assertion.
\item[(iii)] Assume $\K=\K_0$. Then, we have the formula
\[
v_{\dr,\cp/\K}^{(n)}(x)=\inf_{|\bm{n}|<n}v_{\dr,\cp/\qp}^{(n)}(a_{\bm{n}}),
\]
where we have $x=\sum_{\bm{n}\in\n^{\oplus J_{\K}}}a_{\bm{n}}\bm{u}^{\bm{n}}\in\B_{\dr,\cp/\K}^+$ with $a_{\bm{n}}\in\B_{\dr,\cp/\qp}^+$. In fact, the assertion follows from the explicit description of $\A_{\inf,\cp/\K}$.
\end{enumerate}
\end{rem}

\subsection{Connections on $\B_{\dr}$ and $\B_{\HT}$}\label{subsec:connection}
We denote by $\om^q_{\K}\hat{\otimes}_{\K}\B_{\dr,\cp/\K}$ the direct limit $\varinjlim\om^q_{\K}\hat{\otimes}_{\K}\B_{\dr,\cp/\K}^+$, where the transition maps are the multiplication by $1\otimes t^{-1}$. Then, the canonical derivation $d:\K\to\om^1_{\K}$ uniquely extends to a $\B_{\dr,\cp/\qp}$-linear continuous derivation
\[
\nabla:\B_{\dr,\cp/\K}\to\om^1_{\K}\hat{\otimes}_{\K}\B_{\dr,\cp/\K}
\]
by the explicit description of $\B_{\dr,\cp/\K}$. More precisely, if we denote by $\{\partial_j\}_{j\in J_{\K}}$ the derivations on $\B_{\dr,\cp/\K}$ given by $\nabla(x)=\sum_{j\in J_{\K}}dt_j\otimes\partial_j(x)$, then $\{\partial_j\}_{j\in J_{\K}}$ are mutually commutative continuous $\B_{\dr,\cp/\qp}$-derivations and we have $\partial_j=\partial/\partial u_j$. More generally, the exterior derivation $d_q:\om^q_{\K}\to\om^{q+1}_{\K}$ for $q\in\n_{>0}$ uniquely extends to a $\B_{\dr,\cp/\qp}$-linear continuous homomorphism $\nabla_q:\om^q_{\K}\hat{\otimes}_{\K}\B_{\dr,\cp/\K}\to\om^{q+1}_{\K}\hat{\otimes}_{\K}\B_{\dr,\cp/\K}$ such that we have $\nabla_q(\omega\otimes x)=\nabla_q(\omega)\otimes x +(-1)^q\omega\wedge \nabla(x)$ for $x\in\B_{\dr,\cp/\K}$ and $\omega\in\om^q_{\K}$. Obviously, the connection $\nabla$ satisfies Griffith transversality $\nabla(\mathrm{Fil}^n\B_{\dr,\cp/\K})\subset \om^1_{\K}\hat{\otimes}_{\K}\mathrm{Fil}^{n-1}\B_{\dr,\cp/\K}$ for $n\in\Z$. These connections are invariant under a finite extension and an unramified extension by Lemma~\ref{lem:dif}~(iii) and Remark~\ref{rem:dr}~(ii).

\begin{notation}
We will use the following notation:
\begin{gather*}
\B_{\dr,\cp/\K}^{\nabla+}:=(\B_{\dr,\cp/\K}^+)^{\nabla=0},\B_{\dr,\cp/\K}^{\nabla}:=(\B_{\dr,\cp/\K})^{\nabla=0}\\
\B_{\HT,\cp/\K}^{\nabla}:=\mathrm{Im}(\B_{\HT,\cp/\qp}\stackrel{\can.}{\to}\B_{\HT,\cp/\K}).
\end{gather*}
\end{notation}

We endow the first two rings with induced filtrations and the last one with an induced graded structure. Note that these rings are invariant under a finite extension and an unramified extension of $\K$ and that $\B_{\dr,\cp/\K}^{\nabla+}$ and $\B_{\dr,\cp/\K}^{\nabla}$ (resp. $\B_{\HT,\cp/\K}^{\nabla}$) have a canonical $(\K_{\can})^{\alg}$-algebra (resp. $\cp$-algebra) structure. By the above description of the connection and the explicit descriptions of $\B_{\dr,\cp/\K}$ and $\B_{\HT,\cp/\K}$, we have
\begin{lem}\label{lem:isodr}
The canonical maps
\[
\B_{\dr,\cp/\qp}^+\to \B_{\dr,\cp/\K}^{\nabla+},\ \B_{\dr,\cp/\qp}\to \B_{\dr,\cp/\K}^{\nabla},\ \B_{\HT,\cp/\qp}\to \B_{\HT,\cp/\K}^{\nabla}
\]
are isomorphisms. Moreover, these maps are compatible with filtrations and gradings.
\end{lem}

\begin{rem}\label{rem:HT}
Assume that $[k_{\K}:k_{\K}^p]<\infty$. Since $\om^1_{\K}$ is a finite dimensional $\K$-vector space (Remark~\ref{rem:dif}), the connection $\nabla:\B_{\dr,\cp/\K}\to\om^1_{\K}\otimes_{\K}\B_{\dr,\cp/\K}$ induces a $\B_{\HT,\cp/\qp}$-linear derivation
\[
\nabla:\B_{\HT,\cp/\K}\to\om^1_{\K}\otimes_{\K}\B_{\HT,\cp/\K}.
\]
More precisely, if we denote by $\{\partial_j\}_{j\in J_{\K}}$ derivations on $\B_{\HT,\cp/\K}$ defined as a similar way as above, then, by the explicit description of $\B_{\HT,\cp/\K}$, $\{\partial_j\}_{j\in J_{\K}}$ are mutually commutative $\B_{\HT,\cp/\qp}$-linear derivation and we have $\partial_j=t\partial/\partial v_j$. In particular, $\B_{\HT,\cp/\K}^{\nabla}$ coincides with $(\B_{\HT,\cp/\K})^{\nabla=0}$. In general case, we must handle complicated topologies to define such a connection. To avoid it, we define $\B_{\HT,\cp/\K}^{\nabla}$ in an ad-hoc way as above.
\end{rem}

We also have an analogue of Poincar\'e lemma.

\begin{lem}\label{lem:poincare}
The complex
\[\xymatrix{
0\ar[r]&\B_{\dr,\cp/\K}^{\nabla+}\ar[r]^{\inc.}&\B_{\dr,\cp/\K}^+\ar[r]^(.4){\nabla}&\om^1_{\K}\hat{\otimes}_{\K}\B_{\dr,\cp/\K}^+\ar[r]^{\nabla_1}&\om^2_{\K}\hat{\otimes}_{\K}\B_{\dr,\cp/\K}^+
}\]
is exact.
\end{lem}
\begin{proof}
By the invariance of the above complex under a finite extension, we may assume $\K=\K_0$. Recall the explicit description of $\B_{\dr,\cp/\K}^+$ in $\S\S~\ref{subsec:dr}$. Since we have $v_p(\bm{n}!)\le |\bm{n}|$ for $\bm{n}\in\n^{\oplus J_{\K}}$, $x\in\B_{\dr,\cp/\K}^+$ is written uniquely in the form $x=\sum_{\bm{n}\in\n^{\oplus J_{\K}}}a_{\bm{n}}\bm{u}^{[\bm{n}]}$ with $a_{\bm{n}}\in\B_{\dr,\cp/\qp}^+$ such that $\{v_{\dr,\cp/\qp}^{(r)}(a_{\bm{n}})\}_{|\bm{n}|=n}\to\infty$ for all $r,n\in\n$. Moreover, we have an inequality
\begin{equation}\label{eq:val}
\inf_{|\bm{n}|<r}v^{(r)}_{\dr,\cp/\qp}(a_{\bm{n}})+r> \inf_{|\bm{n}|<r}v^{(r)}_{\dr,\cp/\qp}(\bm{n}!a_{\bm{n}})=v^{(r)}_{\dr,\cp/\K}(x)
\end{equation}
by Remark~\ref{rem:dr}~(iii). We have only to prove that for $\omega\in\ker{\nabla_1}$, there exists $x\in\B_{\dr,\cp/\K}^+$ such that $\nabla(x)=\omega$. Write $\omega=\sum_{j\in J_{\K}}dt_j\otimes \lambda_j$ with $\lambda_j\in\B_{\dr,\cp/\K}^+$ such that $\{v_{\dr,\cp/\K}^{(r)}(\lambda_j)\}_{j\in J_{\K}}\to \infty$ for all $r\in\n$. The assumption $\omega\in\ker{\nabla_1}$ implies that we have $\partial_{j'}(\lambda_j)=\partial_j(\lambda_{j'})$ for $j,j'\in J_{\K}$. By the same way as above, we can define an expansion of the form $\lambda_j=\sum_{\bm{n}\in\n^{\oplus J_{\K}}}\lambda_{j,\bm{n}}\bm{u}^{[\bm{n}]}$, where $\lambda_{j,\bm{n}}\in\B_{\dr,\cp/\qp}^+$ satisfies a similar condition as above. By a simple calculation, we have the relation $\lambda_{j,\bm{n}+\mathbf{e}_{j'}}=\lambda_{j',\bm{n}+\mathbf{e}_j}$ for $\bm{n}\in\n^{\oplus J_{\K}}$ and $j,j'\in J_{\K}$. We will define a sequence $\{a_{\bm{n}}\}_{\bm{n}\in\n^{\oplus J_{\K}}}$ in $\B_{\dr,\cp/\qp}^+$ as follows: Let $a_{\mathbf{0}}=0$. For $\bm{n}\neq \bm{0}$, choose any $j\in J_{\K}$ such that $n_j\neq 0$ and define $a_{\bm{n}}:=\lambda_{j,\bm{n}-\mathbf{e}_j}$. By the above relation, this is independent of the choice of $j$. To prove the assertion, it suffices to prove that we have $\{v^{(r)}_{\dr,\cp/\qp}(a_{\bm{n}})\}_{|\bm{n}|=n}\to\infty$ for all $r,n\in\n$. In fact, if this is proved, we see that the element $x:=\sum_{\bm{n}\in\n^{\oplus J_{\K}}}a_{\bm{n}}\mathbf{u}^{[\bm{n}]}$ belongs to $\B_{\dr,\cp/\K}^+$ and we have $\nabla(x)=\omega$ by a simple calculation. We have only to prove that, for fixed $r,n,N\in\n$, we have $v^{(r)}_{\dr,\cp/\qp}(a_{\bm{n}})\ge N$ for all but finitely many $\bm{n}\in\n^{\oplus J_{\K}}$ such that $|\bm{n}|=n$. We may assume $r\ge n$. Choose a finite subset $J$ of $J_{\K}$ such that $v^{(r)}_{\dr,\cp/\K}(\lambda_j)\ge r+N$ for $j\in J_{\K}\setminus J$. Let $\bm{n}\in\n^{\oplus J_K}$ such that $|\bm{n}|=n$. If there exists $j\in J_{\K}\setminus J$ such that $n_j\neq 0$, then we have
\[
v^{(r)}_{\dr,\cp/\qp}(a_{\bm{n}})=v_{\dr,\cp/\qp}^{(r)}(\lambda_{j,\bm{n}-\mathbf{e}_j})> v_{\dr,\cp/\K}^{(r)}(\lambda_j)-r\ge r+N-r=N,
\]
where the first inequality follows from the inequality~(\ref{eq:val}). This implies the assertion since our exceptional set $\{\bm{n}\in\n^J||\bm{n}|=n\}$ is finite.
\end{proof}

\subsection{Universal PD-thickenings}\label{subsec:pd}
\begin{dfn}
A $p$-adically formal $\okk$-PD-thickening of $\oo_{\cp}$ is a triple $(D,\theta_D,\gamma_D)$, where
\begin{enumerate}
\item[$\bullet$] $D$ is a $p$-adically Hausdorff complete $\okk$-algebra,
\item[$\bullet$] $\theta_D:D\to\ocp$ is a surjective $\okk$-algebra homomorphism,
\item[$\bullet$] $\gamma_D$ is a PD-structure on $\ker{\theta_D}$, compatible with the canonical PD-structure on the ideal $(p)$.
\end{enumerate}
\end{dfn}

Obviously, $p$-adically formal $\okk$-thickenings of $\oo_{\cp}$ form a category.

\begin{thm}[{\cite[Th\'eor\`eme~2.2.1]{Fon2}}]
The category of $p$-adically formal $\okk$-thickenings of $\oo_{\cp}$ admits a universal object, i.e., an initial object.
\end{thm}
Such an object is unique up to a canonical isomorphism and we denote it by $(\A_{\cris,\cp/\K},\theta_{\cp/\K},\gamma)$. We will recall the construction. Let $(\okk\otimes_{\Z}W(R_{\cp}))^{\mathrm{PD}}$ be the PD-envelope of $\okk\otimes_{\Z}W(R_{\cp})$ with respect to the ideal $\ker{(\theta_{\cp/\K}:\okk\otimes_{\Z}W(R_{\cp})\to\ocp)}$, compatible with the canonical PD-structure on the ideal $(p)$. Then, $\A_{\cris,\cp/\K}$ is the $p$-adic Hausdorff completion of $(\okk\otimes_{\Z}W(R_{\cp}))^{\mathrm{PD}}$.

\begin{rem}\label{rem:cris}
\begin{enumerate}
\item[(i)] By \cite[R\'emarques~2.2.3]{Fon1}, if we have $\K=\K_0$ and $k_{\K}$ is perfect, then the canonical map $\A_{\cris,\cp/\qp}\to\A_{\cris,\cp/\K}$ is an isomorphism.
\item[(ii)] By a similar proof as Lemma~\ref{lem:bc}~(i), the canonical map
\[
\A_{\cris,\cp/\K}\to\A_{\cris,\cp/\K^{\ur}}
\]
is an isomorphism. In general, we have no invariance for $\A_{\cris,\cp/\K}$ as Remark~\ref{rem:dr}~(ii) even after inverting $p$.
\end{enumerate}
\end{rem}

If $\K=\K_0$ and $k_{\K}$ is perfect, then we have an explicit description of $\A_{\cris,\cp/\K}$:
\[
\A_{\cris,\cp/\K}=\left\{\sum_{n\in\n}a_n\frac{\omega^n}{n!}\Bigl|\ a_n\in\A_{\inf,\cp/\K},\ \{v_{\inf,\cp/\K}(a_n)\}_{n\in\n}\to \infty\right\},
\]
where $\omega$ denotes a generator of $\ker{(\theta_{\cp/\K}:\A_{\inf,\cp/\K}\to\oo_{\cp})}$. Note that the sequence $\{a_n\}_{n\in\n}$ can not be uniquely determined. Moreover, we have $t\in\A_{\cris,\cp/\K}$ and $\A_{\cris,\cp/\K}$ is an integral domain of characteristic~$0$, whose PD-structure is given by $\gamma_n(x)=x^{[n]}=x^n/n!$ for $x\in\ker{\theta_{\cp/\K}}$. In fact, the assertions are reduced to the case $\K=K_0^{\pf}$ by Remark~\ref{rem:bc} and Remark~\ref{rem:cris}~(i), and the assertion in this case follows from \cite[2.3.3]{Fon1}.

We define $\B_{\cris,\cp/\K}^+:=\A_{\cris,\cp/\K}[p^{-1}]$ and $\B_{\cris,\cp/\K}:=\B_{\cris,\cp/\K}^+[t^{-1}]$. We also define $\A_{\st,\cp/\K}:=\A_{\cris,\cp/\K}[\mathrm{x}]$, where $\mathrm{x}$ is a formal variable, $\B_{\st,\cp/\K}^+:=\A_{\st,\cp/\K}[p^{-1}]$ and $\B_{\st,\cp/\K}:=\B_{\st,\cp/\K}^+[t^{-1}]$. We define a monodromy operator $N$ on $\B_{\st,\cp/\K}$ as the $\B_{\cris,\cp/\K}$-derivation $N:=-d/d\mathrm{x}$. We denote by $v_{\cris,\cp/\K}$ the $p$-adic semi-valuation on $\B_{\cris,\cp/\K}^+$ (or $\A_{\cris,\cp/\K}$) defined by the lattice $\A_{\cris,\cp/\K}$.

In the following, we will give an explicit description of $\A_{\cris,\cp/\K}$. Let $\A_{\cris,\cp/\qp}\langle \mathbf{u}_{J_{\K}}\rangle$ be the $p$-adic Hausdorff completion of the PD-polynomial $\A_{\cris,\cp/\qp}$-algebra on the indeterminates $\{\mathrm{u}_j\}_{j\in J_{\K}}$. Note that the PD-structure is given by $\gamma_n(\mathrm{u}_j)=\mathrm{u}_j^n/n!=\mathrm{u}_j^{[n]}$ for $n\in\n$ and $j\in J_{\K}$. We also have
\[
\A_{\cris,\cp/\qp}\langle \mathbf{u}_{J_{\K}} \rangle=\left\{\sum_{\bm{n}\in\n^{\oplus J_{\K}}}a_{\bm{n}}\mathbf{u}^{[\bm{n}]}\Bigl|a_{\bm{n}}\in\A_{\cris,\cp/\qp},\ \{v_{\cris,\cp/\qp}(a_{\bm{n}})\}_{\bm{n}\in\n^{\oplus J_{\K}}}\to\infty\right\}.
\]
We regard $\A_{\cris,\cp/\K}$ as an $\A_{\cris,\cp/\qp}$-algebra by functoriality. Then, by the universal property of PD-polynomial algebras, we have the $\A_{\cris,\cp/\qp}$-algebra homomorphism
\[
\iota_{\cris,\cp/\K}:\A_{\cris,\cp/\qp}\langle \mathbf{u}_{J_{\K}}\rangle\to\A_{\cris,\cp/\K};\mathbf{u}^{[\bm{n}]}\mapsto \bm{u}^{[\bm{n}]}.
\]

\begin{lem}\label{lem:cris}
If we have $\K=\K_0$, then $\iota_{\cris,\cp/\K}$ is an isomorphism. Moreover, we have
\[
v_{\cris,\cp/\K}(x)=\inf_{\bm{n}\in\n^{\oplus J_{\K}}}v_{\cris,\cp/\qp}(a_{\bm{n}})
\]
for $x=\sum_{\bm{n}\in\n^{\oplus J_{\K}}}a_{\bm{n}}\bm{u}^{[\bm{n}]}\in\B_{\cris,\cp/\K}^+$ with $a_{\bm{n}}\in\B_{\cris,\cp/\qp}^+$.
\end{lem}

We use the following lemma in the proof:

\begin{lem}\label{lem:PD}
In the following of the lemma, $\mathrm{Hom}_A$ for a ring $A$ denotes the Hom set of $A$-algebras. We also assume that $\K=\K_0$ and we use the notation in \S~\ref{subsec:cohen}.
\begin{enumerate}
\item[(i)] If $R$ is a $p$-adically Hausdorff complete $\Z[T_j]_{j\in J_{\K}}$-algebra, then the canonical map
\[
\mathrm{Hom}_{\Z[T_j]_{j\in J_{\K}}}(\oo_{\K},R)\to\mathrm{Hom}_{\F_p[T_j]_{j\in J_{\K}}}(k_{\K},R/(p))
\]
is bijective, where the $\F_p[T_j]_{j\in J_{\K}}$-algebra structure on $k_{\K}$ (resp. $R/(p)$) is given by $T_j\mapsto \bar{t}_j$ (resp. is induced by $\Z[T_j]_{j\in J_{\K}}\to R$). Moreover, the restriction map
\[
|_{k_{\K}^p}:\mathrm{Hom}_{\F_p[T_j]_{j\in J_{\K}}}(k_K,R/(p))\to\mathrm{Hom}_{\F_p[T_j^p]_{j\in J_{\K}}}(k_{\K}^p,R/(p))
\]
is bijective, where the $\F_p[T_j^p]_{j\in J_{\K}}$-algebra structure on $k_{\K}^p$ (resp. $R/(p)$) is given by $T_j^p\mapsto t_j^p$ (resp. the composition of the inclusion $\F_p[T_j^p]_{j\in J_{\K}}\to \F_p[T_j]_{j\in J_{\K}}$ and the above structure map $\F_p[T_j]_{j\in J_{\K}}\to R/(p)$).
\item[(ii)] Let $\vartheta:S\to R$ be a surjective homomorphism of $p$-adically Hausdorff complete $\Z[T_j]_{j\in J_{\K}}$-algebras, whose kernel admits a PD-structure, compatible with the canonical PD-structure on the ideal $(p)$. Then, the canonical map
\[
\vartheta_*:\mathrm{Hom}_{\Z[T_j]_{j\in J_{\K}}}(\oo_{\K},S)\to\mathrm{Hom}_{\Z[T_j]_{j\in J_{\K}}}(\oo_{\K},R);f\mapsto \vartheta\circ f
\]
is bijective.
\end{enumerate}
\end{lem}
\begin{proof}
\begin{enumerate}
\item[(i)] The first assertion follows from the $p$-adic formal \'etaleness of $\oo_{\K}/\Z[T_j]_{j\in J_{\K}}$. The latter assertion follows by using the isomorphism of $k_{\K}^p$-algebras $k_{\K}^p[T_j]_{j\in J_{\K}}/(\{T_j^p-\bar{t}_j^p\}_{j\in J_{\K}})\cong k_{\K};\bar{T}_j\mapsto \bar{t}_j$.
\item[(ii)] We denote by $\vartheta_1:S/(p)\to R/(p)$ the ring homomorphism induced by $\vartheta$. By the first assertion of (i), we have only to prove that the canonical map
\[
\mathrm{Hom}_{\F_p[T_j]_{j\in J_{\K}}}(k_{\K},S/(p))\to\mathrm{Hom}_{\F_p[T_j]_{j\in J_{\K}}}(k_{\K},R/(p));f\mapsto \vartheta_{1}\circ f,
\]
which is denoted by $\vartheta_*$ again, is bijective.

We first note the following: We regard $R/(p)$ as a quotient of $S/(p)$ by $\vartheta_1$. Let $x\in R/(p)$ and let $\hat{x}_1$, $\hat{x}_2\in S/(p)$ be lifts of $x$. Then, we have $\hat{x}_1-\hat{x}_2\in \ker{\vartheta_1}$. Since we have $a^p=p!\gamma_p(a)\in pS$ for $a\in\ker{\vartheta}$, where $\gamma$ denotes a PD-structure on $\ker{\vartheta}$, we have $\hat{x}_1^p=\hat{x}_2^p$. In particular, if we denote by $\hat{x}\in S/(p)$ a lift of $x\in R/(p)$, then $\hat{x}^p$ depends only on $x$.

We prove the injectivity. Let $\bar{f}:k_{\K}\to R/(p)$ be an $\F_p[T_j]_{j\in J_{\K}}$-algebra homomorphism and $f$, $f':k_{\K}\to S/(p)$ lifts of $\bar{f}$, i.e., $\vartheta_*(f)=\vartheta_*(f')=\bar{f}$. For $\bar{x}\in k_K$, $f(\bar{x})$ and $f'(\bar{x})\in S/(p)$ are lifts of $\bar{f}(\bar{x})\in R/(p)$, hence we have $f(\bar{x}^p)=f(\bar{x})^p=f'(\bar{x})^p=f'(\bar{x}^p)$ by the above remark. Hence, we have $f|_{k_{\K}^p}=f'|_{k_{\K}^p}$, i.e., $f=f'$ by the latter assertion of (i).

We prove the surjectivity. Let $\bar{f}:k_{\K}\to R/(p)$ be an $\F_p[T_j]_{j\in J_{\K}}$-algebra homomorphism. We have only to construct an $\F_p[T_j^p]_{j\in J_{\K}}$-algebra homomorphism $f:k_{\K}^p\to S/(p)$ such that $\vartheta_*(f)|_{k_{\K}^p}$ coincides with $\bar{f}|_{k_{\K}^p}$, where we endow $k_{\K}^p$ and $S/(p)$ with $\F_p[T_j^p]_{j\in J_{\K}}$-algebra structures by a similar way as in the statement of (i). In fact, we can uniquely extend $f$ to a $\Z[T_j]_{j\in J_{\K}}$-algebra homomorphism $f:k_{\K}\to S/(p)$ by the latter assertion of (i). Moreover, $(\vartheta_*(f))|_{k_{\K}^p}=\vartheta_*(f|_{k_{\K}^p})$ coincides with $\bar{f}|_{k_{\K}^p}$, which implies $\vartheta_*(f)=\bar{f}$ by the latter assertion of (i) again. The set-theoretic map $f:k_{\K}^p\to S/(p);\bar{y}\mapsto \hat{x}^p$, where $\hat{x}\in S/(p)$ is any lift of $\bar{f}(\bar{y}^{p^{-1}})\in R/(p)$, is well-defined by the above remark. Moreover, $f$ is a $\Z[T_j]_{j\in J_{\K}}$-algebra homomorphism by a simple calculation and $\vartheta_*(f)|_{k_{\K}^p}$ coincides with $\bar{f}|_{k_{\K}^p}$ by construction, which implies the assertion.
\end{enumerate}
\end{proof}

\begin{proof}[{Proof of Lemma~\ref{lem:cris}}]
Obviously, we have only to prove the first assertion. Denote $\mathcal{A}=\A_{\cris,\cp/\qp}\langle \mathbf{u}_{J_{\K}}\rangle$. Extend $\theta_{\cp/\qp}:\A_{\cris,\cp/\qp}\to\oo_{\cp}$ to a surjective $\A_{\cris,\cp/\qp}$-algebra homomorphism $\vartheta:\mathcal{A}\to\oo_{\cp}$ by $\vartheta(\mathrm{u}_j^{[n]})=0$. We first prove that $\mathcal{A}$ has an $\okk$-algebra structure such that $\vartheta$ is an $\okk$-algebra homomorphism.

Denote by $\omega$ a generator of the kernel of $\theta_{\cp/\qp}:\A_{\cris,\cp/\qp}\to\ocp$. Then, the PD-structure on the ideal $\ker{\theta_{\cp/\qp}}$ of $\A_{\cris,\cp/\qp}$ canonically extends to a PD-structure $\delta_1$ on the ideal $(\omega)$ of $\mathcal{A}$, compatible with the canonical PD-structure on the ideal $(p)$. By construction, the kernel of the map $\xi:\mathcal{A}\to\A_{\cris,\cp/\qp};\mathbf{u}^{[\bm{n}]}\mapsto 0$ is endowed with a PD-structure $\delta_2$, compatible with the canonical PD-structure on the ideal $(p)$. Since $\mathcal{A}$ is an integral domain of characteristic~$0$, $\delta_1$ and $\delta_2$ induce the same PD-structure on $(\omega)\cap \ker{\xi}$. Hence, by \cite[Proposition~3.12]{BO}, the ideal $\ker{\vartheta}=(\omega)+\ker{\xi}$ admits a PD-structure, compatible with the canonical PD-structure on the ideal $(p)$. Then, the assertion follows by applying Lemma~\ref{lem:PD}~(ii) to $\vartheta$:\[\xymatrix{
\okk\ar@{-->}[rd]^{\exists !}\ar[r]^{\can.}&\ocp\\
\Z[T_j]_{j\in J_{\K}}\ar[r]^{\mathrm{str}.}\ar[u]^{\mathrm{str}.}&\mathcal{A}\ar[u]_{\vartheta},
}\]
where the horizontal structure map is given by $T_j\mapsto \mathrm{u}_j+[\wtil{t}_j]\in\mathcal{A}$.

By the above $\okk$-structure, we may regard $\mathcal{A}$ as a $p$-adically formal $\okk$-PD-thickening of $\ocp$. By universality, we have only to prove that $\iota_{\cris,\cp/\K}$ is an $\okk$-algebra homomorphism. Let $\alpha:\okk\to\A_{\cris,\cp/\K}$ be the composition of the structure map $\okk\to\mathcal{A}$ and $\iota_{\cris,\cp/\K}$. Since $\iota_{\cris,\cp/\K}$ commutes with the projections $\vartheta$ and $\theta_{\cp/\K}$, we have the commutative diagram
\[\xymatrix{
\okk\ar[rd]^{\alpha}\ar[r]^{\can.}&\ocp\\
\Z[T_j]_{j\in J_{\K}}\ar[r]^{\mathrm{str}.}\ar[u]^{\mathrm{str}.}&\A_{\cris,\cp/\K}\ar[u]_{\theta_{\cp/\K}},
}\]
where the horizontal structure map is given by $T_j\mapsto t_j$. By Lemma~\ref{lem:PD}~(ii), $\alpha$ coincides with the structure map $\okk\to\A_{\cris,\cp/\K}$, which implies the assertion.
\end{proof}

Finally, we remark that if we have $\K=\K_0$, then $\B_{\cris,\cp/\K}$ and $\B_{\st,\cp/\K}$ are integral domains by the above explicit description of $\A_{\cris,\cp/\K}$.

\subsection{Connections and Frobenius on $\B_{\cris}$ and $\B_{\st}$}
In this section, assume $\K=\K_0$. We endow $\B_{\cris,\cp/\K}^+$ with the $p$-adic topology defined by the lattice $\A_{\cris,\cp/\K}$. We regard $\B_{\cris,\cp/\K}$ as the direct limit of $\B_{\cris,\cp/\K}^+$ under the multiplication by $t^{-1}$ and denote $\om^q_{\K}\hat{\otimes}_{\K}\B_{\cris,\cp/\K}=\varinjlim\om^q_{\K}\hat{\otimes}_{\K}\B_{\cris,\cp/\K}^+$. Then, the canonical derivation $d:\K\to\om^1_{\K}$ uniquely extends to a $\B_{\cris,\cp/\K}$-linear continuous derivation $\nabla:\B_{\cris,\cp/\K}\to\om^1_{\K}\hat{\otimes}_{\K}\B_{\cris,\cp/\K}$ by the explicit description of $\B_{\cris,\cp/\K}$. Note that we have $\nabla(x^{[n]})=\nabla(x)\cdot x^{[n-1]}$ for $x\in\ker{\theta_{\cp/\K}}$. As in \S\S~\ref{subsec:connection}, if we denote by $\{\partial_j\}_{j\in J_{\K}}$ the derivations on $\B_{\cris,\cp/\K}$ given by $\nabla(x)=\sum_{j\in J_{\K}}dt_j\otimes\partial_j(x)$, then $\{\partial_j\}_{j\in J_{\K}}$ are mutually commutative continuous $\B_{\cris,\cp/\qp}$-derivations and we have $\partial_j=\partial/\partial u_j$. We also have a canonical extension $\nabla_q$ of exterior derivations $d_q$. Also, we can uniquely extend $\nabla_q$ to the map $\nabla_q:\om^q_{\K}\hat{\otimes}_{\K}\B_{\st,\cp/\K}\to\om^{q+1}_{\K}\hat{\otimes}_{\K}\B_{\st,\cp/\K}$ by defining $\nabla(\mathrm{x})=0$, where we define $\om^q_{\K}\hat{\otimes}_{\K}\B_{\st,\cp/\K}:=(\om^q_{\K}\hat{\otimes}_{\K}\B_{\cris,\cp/\K})[\mathrm{x}]$.

Let $\varphi:\oo_{\K}\to\oo_{\K}$ be a lift of the absolute Frobenius on $k_{\K}$. Then, the ring homomorphism
\[
\varphi\otimes\varphi:\oo_{\K}\otimes W(R_{\cp})\to\oo_{\K}\otimes W(R_{\cp})
\]
induces a ring homomorphism on $\A_{\cris,\cp/\K}$. Although the resulting map depends on the choice of a Frobenius lift of $\oo_{\K}$ in general, we denote it by $\varphi$ again. By defining $\varphi(\mathrm{x}):=p\mathrm{x}$, we also have a Frobenius on $\B_{\st,\cp/\K}$. By construction, the connection and the Frobenius on $\B_{\cris,\cp/\K}$ commute and we have the relation $N\circ \varphi=p\varphi\circ N$ by a simple calculation.

\begin{notation}
We define $\B_{\spadesuit,\cp/\K}^{\nabla}:=(\B_{\spadesuit,\cp/\K})^{\nabla=0}$ for $\spadesuit\in\{\cris,\st\}$.
\end{notation}

By the commutativity of $\nabla$ and $\varphi$, these rings are endowed with $\varphi$-actions. Obviously, $\B_{\st,\cp/\K}^{\nabla}$ is endowed with the monodromy operator $N$. By the explicit description of $\B_{\cris,\cp/\K}$, we have

\begin{lem}\label{lem:isocr}
For $\spadesuit\in\{\cris,\st\}$, the canonical map
\[
\B_{\spadesuit,\cp/\qp}\to\B_{\spadesuit,\cp/\K}^{\nabla}
\]
is an isomorphism. Since this map is compatible with Frobenius, Frobenius on $\B_{\spadesuit,\cp/\K}^{\nabla}$ is independent of the choice of a Frobenius lift of $\oo_{\K}$. In particular, the Frobenius on $\B_{\spadesuit,\cp/\K}^{\nabla}$ is injective.
\end{lem}

\subsection{Compatibility with limit}\label{subsec:lim}
When a $p$-basis of $k_{\K}$ is not finite, some technical difficulties occurs. In this case, we will reduce a problem to the finite $p$-basis case by using the results in \S\S~\ref{subsec:pf} and the following inverse limits.

Let the notation be as in \S\S~\ref{subsec:pf}. By functoriality, we have canonical maps
\[
\B_{\spadesuit,\cp/\K_0}\to\varprojlim_{J\in\mathcal{P}(J_K)}\B_{\spadesuit,\cp/\K_{J,0}},\ \B_{\clubsuit,\cp/\K}\to\varprojlim_{J\in\mathcal{P}(J_K)}\B_{\clubsuit,\cp/\K_J},
\]
where $\spadesuit\in\{\cris,\st\}$, $\clubsuit\in\{\dr,\HT\}$. Since these morphisms are compatible with the above explicit descriptions of these rings, it is easy to see that these maps are injective.

\subsection{Embeddings of $\B_{\cris}$ and $\B_{\st}$ into $\B_{\dr}$}
Let $\mathbb{J}_{\cp/\K}:=\ker{(\theta_{\cp/\K}:\A_{\inf,\cp/\K}[p^{-1}]\to\cp)}$. We endow the ideal $\mathbb{J}_{\cp/\K}/\mathbb{J}_{\cp/\K}^n$ of the $\Q$-algebra $\A_{\inf,\cp/\K}[p^{-1}]/\mathbb{J}_{\cp/\K}^n$ with the unique PD-structure. This is compatible with the canonical PD-structure of the ideal $(p)$. Hence, the canonical map $\oo_{\K}\otimes_{\Z}W(R_{\cp})\to\A_{\inf,\cp/\K}[p^{-1}]/\mathbb{J}_{\cp/\K}^n$ factors through $(\oo_{\K}\otimes_{\Z}W(R_{\cp}))^{\mathrm{PD}}\to\A_{\inf,\cp/\K}[p^{-1}]/\mathbb{J}_{\cp/\K}^n$. If we endow the LHS and the RHS with the $p$-adic topology and the $p$-adic Banach space topology respectively (see \S\S~\ref{subsec:dr}), then the above morphism is continuous. In fact, the canonical map times $n!$ factors through the image of $\A_{\inf,\cp/\K}$. By passing to limit, the map extends to $\A_{\cris,\cp/\K}\to\B_{\dr,\cp/\K}^+$. Thus, we have a canonical $\K$-algebra homomorphism $\B_{\cris,\cp/\K}^+\to\B_{\dr,\cp/\K}^+$. Fixing $\wtil{p}\in R_{\cp}$ such that $\wtil{p}^{(0)}=p$, we extends this map to $\B_{\st,\cp/\K}^+\to\B_{\dr,\cp/\K}^+$ by sending $\mathrm{x}$ to $\log{([\wtil{p}]/p)}:=\sum_{n\in\n_{>0}}{(-1)^{n-1}([\wtil{p}]/p-1)^n/n}$. Note that these morphisms are compatible with connections.

\begin{prop}\label{prop:inj}
Assume that the algebraic closure of $\K$ in $\cp$ is dense in $\cp$. Then, the canonical maps
\begin{gather*}
\K_{\can}\otimes_{\K_{\can,0}}\B_{\cris,\cp/\K_0}^{\nabla}\to \B_{\dr,\cp/\K}^{\nabla},\ \K_{\can}\otimes_{\K_{\can,0}}\B_{\st,\cp/\K_0}^{\nabla}\to \B_{\dr,\cp/\K}^{\nabla},\\
\K\otimes_{\K_0}\B_{\cris,\cp/\K_0}\to \B_{\dr,\cp/\K},\ \K\otimes_{\K_0}\B_{\st,\cp/\K_0}\to \B_{\dr,\cp/\K}
\end{gather*}
are injective.
\end{prop}
\begin{proof}
By identifying $\cp$ with the $p$-adic completion of $\K^{\alg}$, we may assume  $\K=K$. Note that if $k_K$ is perfect, then this is due to \cite[4.2.4]{Fon1}. We consider the general case. We first prove that the first two cases. We have only to prove the semi-stable case. The canonical map $K_{\can}\otimes_{K_{\can,0}}K_0^{\pf}\to K^{\pf}$ is injective since $K_{\can}/K_{\can,0}$ is totally ramified and $K_0^{\pf}$ is absolutely unramified. Hence, we have the commutative diagram
\[\xymatrix{
K_{\can}\otimes_{K_{\can,0}}\B_{\st,\cp/K_0}^{\nabla}\ar[d]^{\cong}\ar[rr]^{\can.}&&\B_{\dr,\cp/K}^{\nabla}\ar[d]^{\cong}\\
K_{\can}\otimes_{K_{\can,0}}\B_{\st,\cp/K^{\pf}_0}\ar@{^(->}[r]^(.52){\can.}&K^{\pf}\otimes_{K_0^{\pf}}\B_{\st,\cp/K_0^{\pf}}\ar@{^(->}[r]^(.58){\can.}&\B_{\dr,\cp/K^{\pf}},
}\]
where the vertical arrows are induced by base changes and the injectivity of the bottom second arrow follows from the perfect residue field case. Then, the assertion follows from the above diagram. We consider the latter two cases. By passing to limit (\S\S~\ref{subsec:lim}), we may assume $[k_K:k_K^p]<\infty$. Then, the cristalline case follows from \cite[Proposition~2.47]{Bri}, where $\B_{\cris,\cp/K_0}$ is denoted by $\mathrm{B}_{\cris}$. We will prove the semi-stable case. By regarding $K\otimes_{K_0}\B_{\cris,\cp/K_0}$ as a subring of $\mathrm{Frac}(\B_{\dr,\cp/K})$, the assertion is equivalent to say that $\mathrm{x}$ is transcendental over $\mathrm{Frac}(\B_{\cris,\cp/K_0})$. Suppose that it is not the case. To deduce a contradiction, we have only to construct a non-zero polynomial in $\B_{\cris,\cp/K^{\pf}_0}[X]$ which has $\mathrm{x}$ as a zero. By assumption, we have a non-zero polynomial $f(X)=\sum_ia_iX^i\in\B_{\cris,\cp/K_0}^+[X]$ such that $f(\mathrm{x})=0$. For $\bm{m}\in\n^{\oplus J_K}$, we denote by $\partial^{\bm{m}}$ the product $\Pi_{j\in J_K}\partial_j^{m_j}$, where $\{\partial_j\}_{j\in J_K}$ is the derivations defined in \S\S~\ref{subsec:connection}. Denote by $\tilde{f}^{(\bm{m})}(X)\in\B_{\cris,\cp/K_0^{\pf}}^+[X]$ the image of the polynomial $f^{(\bm{m})}(X):=\sum_i\partial^{\bm{m}}(a_i)X^i$ under the canonical homomorphism $\B_{\dr,\cp/K}^+\to\B_{\dr,\cp/K^{\pf}}^+$. Then, $\tilde{f}^{(\bm{m})}(X)$ has $\mathrm{x}$ as a zero since we have $\mathrm{x}\in\B_{\dr,\cp/K}^{\nabla+}$. Write $a_i=\sum_{\bm{n}\in\n^{\oplus J_K}}a_{i,\bm{n}}\bm{u}^{[\bm{n}]}$ with $a_{i,\bm{n}}\in\B_{\cris,\cp/\qp}^+$ by using the explicit description of $\B_{\cris,\cp/K_0}^+$ given in \S\S~\ref{subsec:pd}. Then, we have $\partial^{\bm{m}}(a_i)=\sum_{\bm{n}\in\n^{\oplus \K}}a_{i,\bm{n}+\bm{m}}\bm{u}^{[\bm{n}]}$ and $\tilde{f}^{(\bm{m})}(X)=\sum_ia_{i,\bm{m}}X^i$ by simple calculations. Hence, we obtain a desired polynomial $\tilde{f}^{(\bm{m})}(X)$ by choosing $\bm{m}\in\n^{\oplus J_K}$ such that we have $a_{i,\bm{m}}\neq 0$ for some $i$.
\end{proof}

\section{Basic properties of rings of $p$-adic periods}\label{sec:hodge2}
We will apply the construction of the previous section to the cases $\K=\qp,K,K^{\pf},\dotsc$. Then, the resulting rings of $p$-adic periods have an appropriate Galois action by the functoriality of the construction: For example, $\gk$ acts on $\B_{\dr,\cp/\qp}$ and $\B_{\dr,\cp/K}$, $G_{K^{\pf}}$ acts on $\B_{\dr,\cp/K^{\pf}}$. In this section, we will review Galois theoretic properties of such rings. The proofs of the properties are somewhat technical and the reader may skip this section by admitting the results including the $\gk$-regularities just below. We keep the notation in the previous section.

\subsection{Calculations of $H^0$ and verification of $\gk$-regularities}
In this subsection, we will prove the $\gk$-regularities of the $(\qp,\gk)$-rings
\begin{gather*}
\B_{\cris,\cp/K_0},\ \B_{\st,\cp/K_0},\ \B_{\dr,\cp/K},\ \B_{\HT,\cp/K},\\
\B_{\cris,\cp/K_0}^{\nabla},\ \B_{\st,\cp/K_0}^{\nabla},\ \B_{\dr,\cp/K}^{\nabla},\ \B_{\HT,\cp/K}^{\nabla},
\end{gather*}
which are used in the following of the paper, and calculate its $H^0$. Note that these rings are integral domains by the explicit descriptions of these rings.

\begin{lem}\label{lem:dr}
Let $\clubsuit\in\{\dr,\HT\}$.
\begin{enumerate}
\item[(i)] We have
\[
H^0(\gk,\mathrm{Frac}(\B_{\clubsuit,\cp/K}))=K.
\]
\item[(ii)] The $(\qp,\gk)$-ring $\B_{\clubsuit,\cp/K}$ satisfies $(G\cdot R_3)$ (\S\S~\ref{subsec:reg}).
\item[(iii)] The $(\qp,\gk)$-ring $\B_{\clubsuit,\cp/K}$ is $\gk$-regular.
\end{enumerate}
\end{lem}
\begin{proof}
The assertion~(iii) follows from (i), (ii) and Lemma~\ref{lem:frac}. We will prove (i) and (ii) separately in the Hodge-Tate case and the de Rham case.
\begin{enumerate}
\item[(a)] The Hodge-Tate case: We first verify (i). By Theorem~\ref{thm:coh}, we have only to prove that if we have non-zero $x,y\in\B_{\HT,\cp/K}$ such that $g(x)y=xg(y)$ for all $g\in \gk$, then we have $x/y\in\cp$. Recall the notation in \S\S~\ref{subsec:pf}. In the following, let $J\in\mathcal{P}(J_K)$ and denote by $x_J,y_J$ the image of $x,y$ in $\B_{\HT,\cp/K_J}$. If $x_J$ and $y_J$ are non-zero, then we have $g(x_J)=c_gx_J$ and $g(y_J)=c_gy_J$ for $c_g\in(\B_{\HT,\cp/K_J})^{\times}\cong\cup_{n\in\Z}\cp^{\times} t^n$ by the fact that $\B_{\HT,\cp/K_J}\cong\cp[t,t^{-1},\{v_j\}_{j\in J_K\setminus J}]$ is a uniquely factorization domain. By the explicit description of $\B_{\HT,\cp/K_J}$, we can choose $\bm{n}\in\n^{J_K\setminus J}$ such that $\partial^{\bm{n}}(x_J)\in \B_{\HT,\cp/K_J}^{\nabla}\setminus\{0\}\cong \cp[t,t^{-1}]\setminus \{0\}$, where $\partial_j=t\partial/\partial v_j$ and $\partial^{\bm{n}}:=\Pi_j\partial_j^{n_j}$ (Remark~\ref{rem:HT}). Write $\partial^{\bm{n}}(x)=\sum_{n\in\Z}a_nt^n$ with $a_n\in\cp$. Then, we have $g(\partial^{\bm{n}}(x_J))=c_g\partial^{\bm{n}}(x_J)$ by the commutativity of $\partial_j$ and the $\gk$-action. Since $c_g$ is homogeneous with respect to $t$, we have $c_g\in \cp$ by comparing the degree. By comparing the leading terms, we have $c_g=g(a_n)/a_n\chi^n(g)$ for all $g\in G_{K_J}$, where $n$ is the degree of $\partial^{\bm{n}}(x)$ with respect to $t$. Hence, we have $x_J/a_nt^n\in (\B_{\HT,\cp/K_J})^{G_{K_J}}$. Note that we have $(\B_{\HT,\cp/K_J})^{G_{K_J}}=K_J$. In fact, it follows from the facts that we have $\B_{\HT,\cp/K_J}=\cup_{r\in\n}t^{-r}\cp[t,\{tv_j\}_{j\in J_K\setminus J}]$ and $H^0(G_{K_J},t^{-r}\cp[t,\{tv_j\}_{j\in J_K\setminus J}])=K_J$ by \cite[Lemme~2.15]{Bri}, where $\cp[t,\{tv_j\}_{j\in J_K\setminus J}]$ is denoted by $\oplus_{r\in\n}\mathrm{gr}^r(\B_{\dr}^+)$ in the reference. Thus, we have $x_J\in \cp^{\times}t^n$. By the same argument, we have $y_J\in\cp^{\times}t^m$ for some $m\in\Z$. We will prove $n=m$. Write $x_J=at^n$, $y_J=bt^m$ with $a,b\in\cp^{\times}$. Then, we have $g(a/b)=\chi^{m-n}(g)(a/b)$ for $g\in G_{K_J}$. Since $H^0(G_{K_J},\cp(n-m))$ is non-zero if and only if $n=m$ by Theorem~\ref{thm:coh}, we must have $n=m$. In particular, we have $x_J/y_J=a/b\in \cp$. Since the set $S_{x,y}:=\{J\in\mathcal{P}(J_K)|x_J\neq 0\text{ and }y_J\neq 0\}$ is a cofinal subset of $\mathcal{P}(J_K)$ by the explicit description of $\B_{\HT,\cp/K}$, the assertion follows from the injection in \S\S~\ref{subsec:lim}.

We will verify (ii). Let $x\in\B_{\HT,\cp/K}$ be a generator of a $\gk$-stable $\qp$-line in $\B_{\HT,\cp/K}$. Write $g(x)=c_gx$ with $c_g\in\qp^{\times}$. We use the same notation as above. By a similar argument as above, if $x_J\neq 0$, then we have $x_J=a_Jt^{n_J}$ for $a_J\in\cp^{\times}$ and $n_J\in\n$. Moreover, $a_J$ and $n_J$ are unique. In particular, $\{a_J\}$ and $\{n_J\}$ are constant on the cofinal subset $S_{x,x}$ of $\mathcal{P}(J_K)$ and we have $x\in\cp^{\times} t^n\subset (\B_{\HT,\cp/K})^{\times}$ by the injection in \S\S~\ref{subsec:pf}.
\item[(b)] The de Rham case: To prove the assertion~(i), we have only to prove that if we have non-zero $x,y\in\B_{\dr,\cp/K}$ such that $g(x)y=xg(y)$ for all $g\in\gk$, then we have $x/y\in K$. Let $J\in\mathcal{P}(J_K)$ and denote by $x_J,y_J\in\B_{\dr,\cp/K_J}$ the image of $x,y$. If $x_J\neq 0$ and $y_J\neq 0$, then we have $x_J/y_J\in H^0(G_{K_J},\mathrm{Frac}(\B_{\dr,\cp/K_J}))=K_J$ by \cite[Proposition~2.18]{Bri}, where $\mathrm{Frac}(\B_{\dr,\cp/K_J})$ is denoted by $\mathrm{C}_{\dr}$. Since the set $\{J\in\mathcal{P}(J_K)|x_J\neq 0\text{ and }y_J\neq 0\}$ is a cofinal subset of $\mathcal{P}(J_K)$ by the explicit description of $\B_{\dr,\cp/K}^+$, we have $x/y\in \cap_{J\in\mathcal{P}(J_K)}K_J=K$ by the injection in \S\S~\ref{subsec:pf}. We will verify (ii). By Remark~\ref{rem:bc}~(i), we may assume $K=K^{\ur}$. Let $V$ be a $\gk$-stable $\qp$-line in $\B_{\dr,\cp/K}$ generated by $x$. By Lemma~\ref{lem:line} below and Theorem~\ref{thm:Sen}, there exist $n\in\Z$ and a finite extension $L/K$ such that $Vt^n\subset (\B_{\dr,\cp/K})^{G_L}=(\B_{\dr,\cp/L})^{G_L}=L$, in particular, we have $x\in(\B_{\dr,\cp/K})^{\times}$.
\end{enumerate}
\end{proof}

\begin{lem}\label{lem:line}
Let $V$ be a $\gk$-stable $\qp$-line in $\B_{\dr,\cp/K}$. Then, up to Tate twist, $V$ is $\cp$-admissible as a $p$-adic representation.
\end{lem}
\begin{proof}
We may assume $K=K^{\ur}$ by Hilbert~90 and Remark~\ref{rem:dr}~(ii). Let $x\in\B_{\dr,\cp/K}$ be a generator of $V$. By multiplying a power of $t$, we may assume $x\in\B_{\dr,\cp/K}^+$. Let $\rho:\gk\to\qp^{\times}$ be the character defined by $\rho(g)=g(x)/x$. By the explicit description of $\B_{\dr,\cp/K}^+$ (\S\S~\ref{subsec:dr}), we have $x=\sum_{\bm{n}\in\n^{\oplus J_K}}a_{\bm{n}}\bm{u}^{\bm{n}}$ with $a_{\bm{n}}\in\B_{\dr,\cp/\qp}^+$. Choose $\bm{n}\in\n^{\oplus J_K}$ such that $a_{\bm{n}}\neq 0$ and write $a_{\bm{n}}=t^n\lambda$ with $n\in\n$ and $\lambda\in (\B_{\dr,\cp/\qp}^+)^{\times}$. Since we have $g(a_{\bm{n}})=\rho(g)a_{\bm{n}}$ for $g\in G_{K^{\pf}}$, we have $(\rho\chi^{-n})(g)=g(\lambda)/\lambda$ for $g\in G_{K^{\pf}}$. By taking the $\qp$-linear map $\theta_{\cp/\qp}$, we have $(\rho\chi^{-n})(g)=g(\theta_{\cp/\qp}(\lambda))/\theta_{\cp/\qp}(\lambda)$ for $g\in G_{K^{\pf}}$, i.e., $\rho\chi^{-n}|_{K^{\pf}}$ is $\cp$-admissible. Hence, $\rho \chi^{-n}$ is $\cp$-admissible by Theorem~\ref{thm:Sen}.
\end{proof}

\begin{cor}\label{cor:inv}
We have
\begin{gather*}
(\B_{\cris,\cp/K_0}^{\nabla})^{\gk}=(\B_{\st,\cp/K_0}^{\nabla})^{\gk}=K_{\can,0},(\B_{\cris,\cp/K_0})^{\gk}=(\B_{\st,\cp/K_0})^{\gk}=K_0,\\
(\B_{\dr,\cp/K}^{\nabla+})^{\gk}=(\B_{\dr,\cp/K}^{\nabla})^{\gk}=K_{\can},(\B_{\dr,\cp/K}^+)^{\gk}=(\B_{\dr,\cp/K})^{\gk}=K,\\
(\B_{\HT,\cp/K}^{\nabla})^{\gk}=(\B_{\HT,\cp/K})^{\gk}=K.
\end{gather*}
\end{cor}
\begin{proof}
Since we have trivial inclusions (such as $K_0\subset (\B_{\cris,\cp/K_0})^{\gk}$), we have only to the converse inclusions. By passing to limit (\S\S~\ref{subsec:pf} and \ref{subsec:lim}), we may assume $[k_K:k_K^p]<\infty$. We prove Hodge-Tate case first. Since we have $\B_{\HT,\cp/K}^{\nabla}\cong\oplus_{n\in\Z}\cp(n)$ (\S\S~\ref{subsec:dr}), the assertion for $\B_{\HT,\cp/K}^{\nabla}$ follows from Theorem~\ref{thm:coh}. The assertion for $\B_{\HT,\cp/K}$ follows from \cite[Lemme~2.15]{Bri}.

We will prove the rest of the assertion. Since we have $K_{\can,0}=(K_0)_{\can}$ by comparing the residue fields, the assertions in the horizontal cases follow from those in without $\nabla$ case by taking horizontal sections. The de Rham case follows from Lemma~\ref{lem:dr}~(i) and the cristalline and semi-stable cases follow from de Rham case and Proposition~\ref{prop:inj}.
\end{proof}

\begin{lem}
The $(\qp,\gk)$-ring $\B_{\spadesuit,\cp/K_0}$ satisfies $(G\cdot R_3)$ for $\spadesuit\in\{\cris,\st\}$. In particular, $\B_{\spadesuit,\cp/K_0}$ is $\gk$-regular.
\end{lem}
\begin{proof}
Note that the last assertion is obtained by applying Lemma~\ref{lem:rel}, whose assumptions are satisfied by  Proposition~\ref{prop:inj}, Lemma~\ref{lem:dr}~(iii) and Corollary~\ref{cor:inv}. By Remark~\ref{rem:cris}~(iii), we may assume $K=K^{\ur}$. Let $V$ be a $\gk$-stable $\qp$-line in $\B_{\spadesuit,\cp/K_0}$ with generator $x$. By Lemma~\ref{lem:line}, there exists $n\in\Z$ such that $Vt^n$ is $\cp$-admissible as a $p$-adic representation of $\gk$. By Theorem~\ref{thm:Sen}, the image of the map $\rho:\gk\to\qp^{\times};g\mapsto g(xt^n)/(xt^n)$ is included in $(\qp^{\times})_{\mathrm{tors}}$, which is killed by $2(p-1)$. Therefore, we have $(xt^n)^{2(p-1)}\in (\B_{\spadesuit,\cp/K_0})^{\gk}=K_0$, which implies $x\in\B_{\spadesuit,\cp/K_0}^{\times}$.
\end{proof}

\begin{lem}
The $(\qp,\gk)$-rings
\[
\B_{\cris,\cp/K_0}^{\nabla},\B_{\st,\cp/K_0}^{\nabla},\B_{\dr,\cp/K}^{\nabla},\B_{\HT,\cp/K}^{\nabla}
\]
are $\gk$-regular.
\end{lem}
\begin{proof}
The $\gk$-regularity of the field $\B_{\dr,\cp/K}^{\nabla}$ follows from Example~\ref{ex:reg}. Since we have a $G_{K^{\pf}}$-equivariant canonical isomorphism $\B_{\spadesuit,\cp/K_0}^{\nabla}\cong\B_{\spadesuit,\cp/K_0^{\pf}}$ for $\spadesuit\in\{\cris,\st\}$, the verification of $(G\cdot R_3)$ for $\B_{\spadesuit,\cp/K_0}^{\nabla}$ is reduced to that for $\B_{\spadesuit,\cp/K_0^{\pf}}$, which follows from \cite[Proposition~5.1.2~(ii)]{Fon2}. By a similar reason, $(G\cdot R_3)$ for $\B_{\dr,\cp/K}^{\nabla}$ is reduced to \cite[Proposition~3.6]{Fon2}. The $(\qp,\gk)$-ring $\cp(\!(t)\!)$ is a field containing the fractional field of $\B_{\HT,\cp/K}^{\nabla}\cong\cp[t,t^{-1}]$. By Theorem~\ref{thm:coh} and d\'evissage, we have $\cp(\!(t)\!)^{\gk}=K=(\B_{\HT,\cp/K})^{\gk}$, where the last equality follows from Corollary~\ref{cor:inv}. By applying Lemma~\ref{lem:frac}, $\B_{\HT,\cp/K}^{\nabla}$ is $\gk$-regular. By Corollary~\ref{cor:inv}, the $\gk$-regularities for $\B_{\cris,\cp/K_0}^{\nabla}$ and $\B_{\st,\cp/K_0}^{\nabla}$ follows from Lemma~\ref{lem:rel} and Proposition~\ref{prop:inj}.
\end{proof}

\begin{rem}
Note that the $(\qp,\gk)$-rings $\B_{\bullet,\cp/\qp}^{\nabla}$ and $\B_{\bullet,\cp/\qp}$ for $\bullet\in\{\cris,\st,\dr,\HT\}$ are $\gk$-regular. We also have $(\B_{\bullet,\cp/\qp}^{\nabla})^{\gk}=(\B_{\bullet,\cp/\qp})^{\gk}\cong (\B_{\bullet,\cp/K_0}^{\nabla})^{\gk}$. In fact, the assertion follows from canonical isomorphisms $\B_{\bullet,\cp/\qp}^{\nabla}=\B_{\bullet,\cp/\qp}\to\B_{\bullet,\cp/K_0}^{\nabla}$ as $(\qp,\gk)$-rings.
\end{rem}

\begin{notation}
\begin{enumerate}
\item[(i)] For $\spadesuit\in\{\cris,\st\}$, we define the category of $\spadesuit$-representations (resp. horizontal $\spadesuit$-representations) of $\gk$ as $\rep_{\B_{\spadesuit,\cp/K_0}}\gk$ (resp. $\rep_{\B_{\spadesuit,\cp/K_0}^{\nabla}}\gk$), which is denoted by $\rep_{\spadesuit}\gk$ (resp. $\rep_{\spadesuit}^{\nabla}\gk$). We denote by $\D_{\spadesuit}$ (resp. $\D_{\spadesuit}^{\nabla}$) the corresponding functor $\D_B$ and we also denote by $\alpha_{\spadesuit,\cp/K_0}$ (resp. $\alpha_{\spadesuit,\cp/K_0}^{\nabla}$) the corresponding comparison map $\alpha_B$ (\S\S~\ref{subsec:reg}).
\item[(ii)] For $\clubsuit\in\{\dr,\HT\}$, we also define the category of $\clubsuit$-representations (resp. horizontal $\clubsuit$-representations) of $\gk$ as $\rep_{\B_{\clubsuit,\cp/K}}\gk$ (resp. $\rep_{\B_{\clubsuit,\cp/K}^{\nabla}}\gk$), which is denoted by $\rep_{\clubsuit}\gk$ (resp. $\rep_{\clubsuit}^{\nabla}\gk$). We denote by $\D_{\clubsuit}$ (resp. $\D_{\clubsuit}^{\nabla}$) the corresponding functor $\D_B$ and we also denote by $\alpha_{\clubsuit,\cp/K}$ (resp. $\alpha_{\clubsuit,\cp/K}^{\nabla}$) the corresponding comparison map $\alpha_B$ (loc. cit.).
\item[(iii)] We define rings with $\gk$-actions and automorphisms $\varphi$ by
\[
\wtil{\B}_{\rig,\cp/K_0}^{\nabla+}:=\cap_{n\in\n}\varphi^n(\B_{\cris,\cp/K_0}^{\nabla+}),\ \wtil{\B}_{\log,\cp/K_0}^{\nabla+}:=\cap_{n\in\n}\varphi^n(\B_{\st,\cp/K_0}^{\nabla+}).
\]
Note that we have $\wtil{\B}_{\spadesuit,\cp/K_0}^{\nabla+}\cong\wtil{\B}_{\spadesuit,\cp/K_0^{\pf}}^{\nabla+}$ for $\spadesuit\in\{\rig,\log\}$.
\item[(iv)] In the rest of the paper, when $k_K$ is perfect, we omit hyperscripts $\nabla$ to be consistent with the usual notation. For example, we write $\wtil{\B}_{\rig,\cp/K_0^{\pf}}^+$ instead of $\wtil{\B}_{\rig,\cp/K_0^{\pf}}^{\nabla+}$.
\end{enumerate}
\end{notation}

\begin{rem}\label{rem:indep}
As is explained in \S~\ref{subsec:cohen}, there is no canonical choice of a Cohen ring of $k_K$ nor a Frobenius lift when $k_K$ is not perfect. Since some definitions, such as the definition of cristalline representations, involve these choices. We will give remarks on an independence of definitions.
\begin{enumerate}
\item[(i)] Since we have a canonical isomorphism $\B_{\clubsuit,\cp/\qp}\cong\B_{\clubsuit,\cp/K}^{\nabla}$ for $\clubsuit\in\{\dr,\HT\}$ (Lemma~\ref{lem:isodr}), $\B_{\clubsuit,\cp/K}^{\nabla}$ depend only on $\cp$ as a abstract ring.
\item[(ii)] Since we have a canonical isomorphism $\B_{\spadesuit,\cp/\qp}^+\cong \B_{\spadesuit,\cp/K_0}^{\nabla+}$ for $\spadesuit\in\{\cris,\st\}$ (Lemma~\ref{lem:isocr}), the category $\rep_{\spadesuit}^{\nabla}\gk$ depend only on $\cp$ but not on the choice of $K_0$. It also follows that $\wtil{\B}_{\spadesuit,\cp/K_0}^{\nabla+}$ for $\spadesuit\in\{\rig,\log\}$ is independent of the choices of $K_0$ and $\varphi$ as an $\qp$-algebra with $\varphi$-action. Moreover, for a finite extension $L/K$, $\wtil{\B}_{\spadesuit,\cp/K_0}^{\nabla+}$ coincides with $\wtil{\B}_{\spadesuit,\cp/L_0}^{\nabla+}$ in $\B_{\dr,\cp/L}^{\nabla+}$.
\item[(iii)] By definition, the category $\rep_{\spadesuit}\gk$ for $\spadesuit\in\{\cris,\st\}$ may depend on the choice of $K_0$. In the case $[k_K:k_K^p]<\infty$ with $\spadesuit=\cris$, the independence is proved by Brinon \cite[Proposition~3.42]{Bri}: He proves the assertion by introducing a ring $\mathbf{A}_{\mathrm{max},K}$, which is independent of the choice of $K_0$ and is slightly bigger than $\ok\otimes_{\oo_{K_0}}\A_{\cris,\cp/K_0}$. Although a similar idea seems to work in general case, we do not treat this problem in this paper. Instead, we will state a precise version of Main Theorem later (see \S~\ref{sec:main}).
\end{enumerate}
\end{rem}

\begin{rem}[{Hilbert~90}]\label{rem:hil}
Let $V\in\rep_{\qp}\gk$. Then, $V$ is cristalline or semi-stable if and only if so is $V|_{K^{\ur}}$. In fact, we have $\B_{\cris,\cp/K_0}\cong\B_{\cris,\cp/K_0^{\ur}}$ by Remark~\ref{rem:cris}~(i), whose $G_{K^{\ur}}$-invariant is $K_0^{\ur}$ by Corollary~\ref{cor:inv}. Hence, the assertion in the cristalline case follows from Hilbert~90 and the same proof works also in the semi-stable case. We can also prove that $V$ is de Rham or Hodge-Tate if and only if so is $V|_L$ for a finite extension $L$ over the completion of an unramified extension of $K$. In fact, we reduces to the cases when $L/K$ is finite or unramified and in these cases the claim follows from Remark~\ref{rem:dr}~(ii) and Hilbert~90.
\end{rem}

Algebraic structures of rings of $p$-adic period, which are compatible with the action of $\gk$, induce additional structures on the corresponding $\D$. We do not review these structures here since we do not need all of them to prove Main Theorem. For the reader interested in these structures, see \cite[3.5]{Bri} for example. We need only the connection on $\D_{\dr}$ for the proof of Main Theorem: The finite $K$-vector space $\D_{\dr}(V)$ for $V\in\rep_{\dr}\gk$ has a connection $\nabla:\D_{\dr}(V)\to\om^1_K\otimes_K\D_{\dr}(V)$, which is compatible with the canonical derivation on $K$.

\subsection{Restriction to perfection}\label{subsec:res}
If we have $V\in\rep_{\bullet}\gk$ with $\bullet\in\{\cris,\ \st,\ \dr,\ \HT\}$, then we have $V|_{K^{\pf}}\in \rep_{\bullet}G_{K^{\pf}}$. Moreover, we have canonical isomorphisms
\[
K^{\pf}_0\otimes_{K_0}\D_{\spadesuit}(V)\to\D_{\spadesuit}(V|_{K_0^{\pf}}),\ K^{\pf}\otimes_{K}\D_{\clubsuit}(V)\to\D_{\clubsuit}(V|_{K^{\pf}}),
\]
which is induced by the canonical map $\B_{\spadesuit,\cp/K_0}\to\B_{\spadesuit,\cp/K^{\pf}_0}$ and $\B_{\clubsuit,\cp/K}\to\B_{\clubsuit,\cp/K^{\pf}}$ for $\spadesuit\in\{\cris,\st\}$ and $\clubsuit\in\{\dr,\HT\}$. We first prove the de Rham case. By applying $\B_{\dr,\cp/K^{\pf}}\otimes_{\B_{\dr,\cp/K}}$ to the comparison isomorphism $\alpha_{\dr,\cp/K}(V)$, we have a $\gk$-equivariant isomorphism
\[
\B_{\dr,\cp/K^{\pf}}\otimes_K\D_{\dr}(V)\to \B_{\dr,\cp/K^{\pf}}\otimes_{\qp}V.
\]
By taking $G_{K^{\pf}}$-invariant, we have an isomorphism $K^{\pf}\otimes_K\D_{\dr}(V)\to\D_{\dr}(V|_{K^{\pf}})$. The other cases follows similarly.

\section{Construction of $\nrig(V)$}\label{sec:const}
In this section, we will construct a $(\varphi,\gk)$-module $\nrig(V)$ over $\Brig{\cp/K_0}$ for a de Rham representation $V$ of $\gk$, possibly after Tate twist. Our $\nrig$ coincides with Colmez' $\wtil{\n}_{\rig}^+$ when the residue field $k_K$ is perfect.

We first recall Colmez' Dieudonn\'e-Manin theorem, which is a key ingredient of the construction. Let $M$ be a finite free $\B_{\dr,\cp/K}^{\nabla+}$-module of rank~$r>0$. We call $N$ a $\B_{\dr,\cp/K}^{\nabla+}$-lattice of $M$ if $N$ is a $\B_{\dr,\cp/K}^{\nabla+}$-submodule of $M$ such that $N[t^{-1}]=M[t^{-1}]$. Note that a $\B_{\dr,\cp/K}^{\nabla+}$-lattice of $M$ is finite free of rank~$r$ over $\B_{\dr,\cp/K}^{\nabla+}$ since $\B_{\dr,\cp/K}^{\nabla+}$ is a discrete valuation ring.

For $n\in\Z$, denote the composition $\Brig{\cp/K_0}\stackrel{\varphi^n}{\hookrightarrow} \Brig{\cp/K_0}\stackrel{\inc.}{\hookrightarrow}\B_{\dr,\cp/K}^{\nabla+}$ by $\varphi^n$ again. By the commutative diagram
\[\xymatrix{
\wtil{\B}_{\rig,\cp/K_0}^{\nabla+}\ar@{^(->}[r]^{\varphi^n}\ar[d]_{\can.}^{\cong}&\B_{\dr,\cp/K}^{\nabla+}\ar[d]^{\cong}_{\can.}\\
\wtil{\B}_{\rig,\cp/K^{\pf}_0}^+\ar@{^(->}[r]^{\varphi^n}&\B_{\dr,\cp/K^{\pf}}^+,
}\]
the proof of the following theorem is reduced to the perfect residue field case \cite[Proposition~0.3]{Col} (see also the remark below).

\begin{thm}[{Colmez' Dieudonn\'e-Manin classification theorem}]\label{thm:CDM}
Let $r\in\n_{>0}$ and $M$ be a $\B_{\dr,\cp/K}^{\nabla+}$-lattice of $(\B_{\dr,\cp/K}^{\nabla+})^r$. Let
\[
M_{\rig}:=\{x\in (\Brig{\cp/K_0})^r|\ \varphi^n(x)\in M\ \mathrm{for}\ \mathrm{all}\ n\in\Z\}.
\]
Then, $M_{\rig}$ is a finite free $\Brig{\cp/K_0}$-module of rank~$r$ with semi-linear $\varphi$-action and there exists a basis $e_1,\dotsc,e_r$ of $M_{\rig}$ over $\Brig{\cp/K_0}$ satisfying the followings:
\begin{enumerate}
\item[(i)] There exist $h\in\n_{>0}$ and $a_1\le \cdots \le a_r\in \n$ such that $\varphi^{h}(e_i)=p^{a_i}e_i$ for $1\le i\le r$;
\item[(ii)] $e_1,\dotsc,e_r$ is a basis of $M$ over $\B_{\dr,\cp/K}^{\nabla+}$.
\end{enumerate}
\end{thm}

\begin{rem}\label{rem:gen}
Though our condition~(ii) is weaker than that in \cite{Col}, the conclusions of the theorem are the same by the following reason: By definition, $\varphi$ acts on $M_{\rig}$. Since $\varphi^h$ is an automorphism on $M_{\rig}$ by (i), $\varphi$ is also an automorphism on $M_{\rig}$. Hence, (ii) implies that $\varphi^n(e_1),\dotsc,\varphi^n(e_r)$ is a $\Brig{\cp/K_0}$-basis of $M_{\rig}$ for all $n\in\Z$. In particular, $\varphi^n(e_1),\dotsc,\varphi^n(e_r)$ is a $\B_{\dr,\cp/K}^{\nabla+}$-basis of $M$.
\end{rem}

In the rest of this section, let $V$ be a de Rham representation of $\gk$ of dimension~$r$ such that $\D_{\dr}(V)=(\B_{\dr,\cp/K}^+\otimes_{\qp}V)^{\gk}$. Note that the last assumption is satisfied for any de Rham representation after Tate twist. Let $\ndr(V):=\B_{\dr,\cp/K}^+\otimes_K\D_{\dr}(V)$. It is a finite free $\B_{\dr,\cp/K}^+$-module of rank~$r$ with $\gk$-action and $\nabla$-action which are mutually commutative. By the comparison isomorphism $\alpha_{\dr,\cp/K}$, we have a canonical isomorphism $\ndr(V)[t^{-1}]\cong \B_{\dr,\cp/K}\otimes_{\qp}V$, in particular, we have $t^n\B_{\dr,\cp/K}^+\otimes_{\qp}V\subset\ndr(V)\subset \B_{\dr,\cp/K}^+\otimes_{\qp}V$ for sufficiently large $n\in\n$. By taking the horizontal sections, we see that $\hndr(V):=\ndr(V)^{\nabla=0}$ is a $\gk$-stable $\B_{\dr,\cp/K}^{\nabla+}$-lattice of $\B_{\dr,\cp/K}^{\nabla+}\otimes_{\qp}V$. By applying Theorem~\ref{thm:CDM} to $M=\n_{\dr}^{\nabla+}(V)$, we have the following proposition: (In the following, a $(\varphi,\gk)$-module over $\Brig{\cp/K_0}$ (of rank~$r$) means a finite free module (of rank~$r$) over $\Brig{\cp/K_0}$ with a semi-linear $\varphi$-action and a semi-linear $\gk$-action, which are mutually commutative.)

\begin{prop}\label{prop:n}
The $\Brig{\cp/K_0}$-module
\[
\wtil{\n}_{\rig}^{\nabla+}(V):=\{x\in \Brig{\cp/K_0}\otimes_{\qp}V|\ \varphi^n\otimes \mathrm{id}(x)\in \hndr(V)\ \mathrm{for}\ \mathrm{all}\ n\in\Z\}
\]
is a $(\varphi,\gk)$-module over $\Brig{\cp/K_0}$ of rank~$r$. Moreover, we have a basis $e_1,\dotsc,e_r$ of $\widetilde{\n}_{\rig}^{\nabla+}(V)$ over $\Brig{\cp/K_0}$ satisfying the followings:
\begin{enumerate}
\item[(i)] There exist $h\in\n_{>0}$ and $a_1\le \dotsb\le a_r\in\n$ such that $\varphi^h(e_i)=p^{a_i}e_i$ for $1\le i \le r$;
\item[(ii)] $e_1,\dotsc,e_r$ is a basis of $\hndr(V)$ over $\B_{\dr,\cp/K}^{\nabla+}$.
\end{enumerate}
\end{prop}

Note that $\nrig(V)$ is independent of the choice of $K_0$ by Remark~\ref{rem:indep}~(ii). We will use the following property of $\nrig(V)$ in the proof of Main Theorem.
\begin{prop}\label{prop:horizontal}
The canonical map
\[
\B_{\dr,\cp/K}^+\otimes_{\B_{\dr,\cp/K}^{\nabla+}}\hndr(V)\to \ndr(V)
\]
is an isomorphism. In particular, $\B_{\dr,\cp/K}^+\otimes_{\wtil{\B}_{\rig,\cp/K_0}^{\nabla+}}\nrig(V)$ is isomorphic to $(\B_{\dr,\cp/K}^+)^r$ as a $\B_{\dr,\cp/K}^+[\gk]$-module by Proposition~\ref{prop:n}~(ii).
\end{prop}
\begin{proof}

Since $V|_{K^{\pf}}$ is de Rham and we have the canonical isomorphism $\B_{\dr,\cp/\qp}\to\B_{\dr,\cp/K^{\pf}}$, we have the comparison isomorphism
\[
\B_{\dr,\cp/\qp}\otimes_{(\B_{\dr,\cp/\qp})^{G_{K^{\pf}}}}(\B_{\dr,\cp/\qp}\otimes_{\qp}V)^{G_{K^{\pf}}}\to\B_{\dr,\cp/\qp}\otimes_{\qp}V.
\]
By taking the base change of this isomorphism by $\B_{\dr,\cp/\qp}\to\B_{\dr,\cp/K}$, we obtain a canonical isomorphism of $\B_{\dr,\cp/K}[G_{K^{\pf}}]$-modules
\begin{equation}\label{eq:inc1}
\alpha:\B_{\dr,\cp/K}\otimes_{(\B_{\dr,\cp/\qp})^{G_{K^{\pf}}}}(\B_{\dr,\cp/\qp}\otimes_{\qp}V)^{G_{K^{\pf}}}\to\B_{\dr,\cp/K}\otimes_{\qp}V.
\end{equation}
We also have the comparison isomorphism
\[
\alpha_{\dr,\cp/K}(V):\B_{\dr,\cp/K}\otimes_K\D_{\dr}(V)\to \B_{\dr,\cp/K}\otimes_{\qp}V.
\]
Note that we have $(\B_{\dr,\cp/\qp}^+)^{G_{K^{\pf}}}=(\B_{\dr,\cp/\qp})^{G_{K^{\pf}}}$ since we have $(t^{-n}\B_{\dr,\cp/\qp}^+/t^{-n+1}\B_{\dr,\cp/\qp}^+)^{G_{K^{\pf}}}=(\cp(-n))^{G_{K^{\pf}}}=0$ for $n\in\n_{>0}$. We have only to prove that there exists an isomorphism of $\B_{\dr,\cp/K}^+$-modules
\[
(\ndr(V)=)\B_{\dr,\cp/K}^+\otimes_K\D_{\dr}(V)\cong\B_{\dr,\cp/K}^+\otimes_{(\B_{\dr,\cp/\qp}^+)^{G_{K^{\pf}}}}(\B_{\dr,\cp/\qp}\otimes_{\qp}V)^{G_{K^{\pf}}},
\]
which is compatible with the injections $\alpha_{\dr,\cp/K}(V)$ and $\alpha$. In fact, by taking the horizontal sections of both sides, we have $\hndr(V)=\B_{\dr,\cp/K}^{\nabla+}\otimes_{(\B_{\dr,\cp/\qp}^+)^{G_{K^{\pf}}}}(\B_{\dr,\cp/\qp}\otimes_{\qp}V)^{G_{K^{\pf}}}$, which implies the assertion.

We have
\[
\D_{\dr}(V)\hookrightarrow (\B_{\dr,\cp/K}\otimes_{\qp}V)^{G_{K^{\pf}}}=(\B_{\dr,\cp/K})^{G_{K^{\pf}}}\otimes_{(\B_{\dr,\cp/\qp})^{G_{K^{\pf}}}}(\B_{\dr,\cp/\qp}\otimes_{\qp}V)^{G_{K^{\pf}}},
\]
where the equality follows by taking $G_{K^{\pf}}$-invariant of (\ref{eq:inc1}). Note that we have $(\B_{\dr,\cp/K}^+)^{G_{K^{\pf}}}=(\B_{\dr,\cp/K})^{G_{K^{\pf}}}$. Indeed, if we write $x\in\text{LHS}$ as $x=t^{-n}\sum_{\bm{n}\in\n^{\oplus J_K}}a_{\bm{n}}\bm{u}^{\bm{n}}$ with $a_{\bm{n}}\in\B_{\dr,\cp/\qp}^+$, since $\{u_j\}_{j\in J_K}$ are invariant by the action of $G_{K^{\pf}}$, we have $b_{\bm{n}}:=a_{\bm{n}}/t^n\in (\B_{\dr,\cp/\qp})^{G_{K^{\pf}}}=(\B_{\dr,\cp/\qp}^+)^{G_{K^{\pf}}}$. Therefore, we have $x=\sum_{\bm{n}\in\n^{\oplus J_K}}b_{\bm{n}}\bm{u}^{\bm{n}}\in (\B_{\dr,\cp/K}^+)^{G_{K^{\pf}}}$. Hence we have a canonical map
\[
\D_{\dr}(V)\hookrightarrow\B_{\dr,\cp/K}^+\otimes_{(\B_{\dr,\cp/\qp}^+)^{G_{K^{\pf}}}}(\B_{\dr,\cp/\qp}\otimes_{\qp}V)^{G_{K^{\pf}}}.
\]
This induces a canonical homomorphism of $\B_{\dr,\cp/K}^+$-modules
\[
i:\B_{\dr,\cp/K}^+\otimes_K\D_{\dr}(V)\to\B_{\dr,\cp/K}^+\otimes_{(\B_{\dr,\cp/\qp}^+)^{G_{K^{\pf}}}}(\B_{\dr,\cp/\qp}\otimes_{\qp}V)^{G_{K^{\pf}}},
\]
which is compatible with the injections $\alpha_{\dr,\cp/K}(V)$ and $\alpha$ by construction. We have only to prove the surjectivity of $i$. By Nakayama's lemma, we have only to prove the assertion after applying $\B_{\dr,\cp/K^{\pf}}^+\otimes_{\B_{\dr,\cp/K}^+}$ (note that $\B_{\dr,\cp/K}^+\to\B_{\dr,\cp/K^{\pf}}^+$ is a surjective homomorphism of local rings). We have the commutative diagram
\[\xymatrix{
\B_{\dr,\cp/K^{\pf}}^+\otimes_K\D_{\dr}(V)\ar[r]^(.55){\alpha_{\dr,\cp/K}(V)_*}\ar[d]_{i_*}&\B_{\dr,\cp/K^{\pf}}\otimes_{\qp}V\ar@{=}[d]\\
\B_{\dr,\cp/K^{\pf}}^+\otimes_{(\B_{\dr,\cp/\qp}^+)^{G_{K^{\pf}}}}(\B_{\dr,\cp/\qp}\otimes_{\qp}V)^{G_{K^{\pf}}}\ar[r]^(.7){\alpha_*}\ar[d]^{\cong}_{\can.}&\B_{\dr,\cp/K^{\pf}}\otimes_{\qp}V\ar@{=}[d]\\
\B_{\dr,\cp/K^{\pf}}^+\otimes_{K^{\pf}}\D_{\dr}(V|_{K^{\pf}})\ar@{^{(}->}[r]^(.56){\alpha_{\dr,\cp/K^{\pf}}(V|_{K^{\pf}})}&\B_{\dr,\cp/K^{\pf}}\otimes_{\qp}V,
}\]
where the left lower arrow is induced by the $G_{K^{\pf}}$-equivariant isomorphism $\B_{\dr,\cp/\qp}\to\B_{\dr,\cp/K^{\pf}}$. Denote the composition of the left vertical arrows by $i'$. Since the canonical map $\B_{\dr,\cp/K}\to\B_{\dr,\cp/K^{\pf}}$ is $G_{K^{\pf}}$-equivariant, by the diagram, the restriction of $i'$ to $\D_{\dr}(V)$ coincides with the canonical map $\D_{\dr}(V)\to\D_{\dr}(V|_{K^{\pf}})$, which is an isomorphism after tensoring $K^{\pf}$ (see $\S\S$~\ref{subsec:res}). Therefore, $i'$ is an isomorphism and we obtain the assertion.
\end{proof}

\section{Proof of Main Theorem}\label{sec:main}
We will restate our main theorem in the point of view of Remark~\ref{rem:indep}~(iii):
\begin{mthm}
Let $V$ be a de Rham representation of $\gk$. Then, there exists a finite extension $L/K$ such that the restriction $V|_L$ is $\B_{\st,\cp/L_0}$-admissible for any choice of $L_0$.
\end{mthm}
In this section, we give a proof of Main Theorem in this form. Before the proof, we prepare technical lemmas used in the proof. The reader may go to the proof of Main Theorem and back to the lemmas if necessary.

We first recall a slightly modified version of \cite[Proposition~0.6]{Col}. In the following of this section, denote the unramified extension of $\qp$ of degree $h\in \n_{>0}$ by $\Q_{p^h}$.
\begin{prop}[cf. {\cite[Proposition~0.6]{Col}}]\label{prop:main}
Assume that $k_K$ is perfect. Let $\U'_{h,a}:=(\wtil{\B}_{\log,\cp/K_0}^+)^{\varphi^h=p^a}$ for $h$, $a\in\n$. Let $M$ be a $(\varphi,\gk)$-module over $\wtil{\B}_{\rig,\cp/K_0}^+$ of rank $r\in\n_{>0}$ with basis $e_1,\dotsc,e_r$. Assume that there exists an isomorphism of $\B_{\dr,\cp/K}^+[\gk]$-modules $\B_{\dr,\cp/K}^+\otimes_{\wtil{\B}_{\rig,\cp/K}^+}M\cong (\B_{\dr,\cp/K}^+)^r$ and that $e_1,\dotsc,e_r$ satisfies the following conditions:
\begin{enumerate}
\item[(i)] there exists $h\in\n_{>0}$ and $a_1\le\dotsb\le a_r\in\n$ such that $\varphi^h(e_i)=p^{a_i}e_i$ for $1\le i\le r$,
\item[(ii)] there exists a (unique) upper triangular matrix $c_g\in GL_r(\B_{\dr,\cp/K}^+)$ whose diagonal entries are $1$, such that $g(e_1,\dotsc,e_r)=(e_1,\dotsc,e_r)c_g$ for all $g\in\gk$.
\end{enumerate}
Then there exists a $\wtil{\B}_{\rig,\cp/K_0}^+$-basis $f_1,\dotsc,f_r$ of $\wtil{\B}_{\log,\cp/K_0}^+\otimes_{\wtil{\B}_{\rig,\cp/K_0}^+}M$ satisfying the following conditions:
\begin{enumerate}
\item[(a)] $f_i$ is fixed by $\gk$;
\item[(b)] $f_i=e_i+\sum_{1\le j\le i-1}\alpha_{ji}e_j$ with $\alpha_{ji}\in \U'_{h,a_i-a_j}$ (hence we have $\varphi^h(f_i)=p^{a_i}f_i$).
\end{enumerate}
\end{prop}
\begin{proof}
Note that we add the extra assumption~(ii) and the slightly stronger conclusion~(a) to the original proposition. Let $U$ be the subgroup of $GL_r(\B_{\dr,\cp/K}^+)$ consisting of upper triangular matrices, whose diagonal entries are $1$ and the $(i,j)$-component is in $\U'_{h,a_j-a_i}$ for $i<j$. We endow $U$ with the subspace topology of $GL_r(\B_{\dr,\cp/K}^+)$. Then, $U$ is a topological $\gk$-group and the map $g\mapsto c_g;\gk\to U$ is a continuous $1$-cocycle. By \cite[Proposition~0.6]{Col}, there exists a finite Galois extension $L/K$ such that $\mathrm{Res}^L_K([c])$ vanishes in $H^1(G_L,U)$, where $[c]$ denotes the class represented by $c$. Hence, we have only to prove that the inverse image of the neutral element by $\mathrm{Res}^L_K:H^1(\gk,U)\to H^1(G_L,U)$ consists of the neutral element.

We endow $U$ with a $\gk$-stable decreasing filtration $\{\mathcal{F}_n\}_{n\in\n}$ by $\mathcal{F}_n:=\{(x_{ij})\in U|x_{ij}=0\ \mathrm{for}\ 0<j-i\le n\}$. Then, we have $\mathcal{F}_0=U$, $\mathcal{F}_r=\{1\}$, $\mathcal{F}_{n+1}\trianglelefteq\mathcal{F}_n$ and $\mathcal{F}_n/\mathcal{F}_{n+1}$ is isomorphic to a direct sum of copies of $\U'_{h,a}$ with $a\in\n$. We have only to prove that the inverse image of the neutral element under the restriction map $\mathrm{Res}^L_K:H^1(\gk,\mathcal{F}_n)\to H^1(G_L,\mathcal{F}_n)$ for $n\in\n$ consists of the neutral element. Since there exists a $\gk$-equivariant set-theoretic section of the canonical projection $\mathcal{F}_n\to\mathcal{F}_n/\mathcal{F}_{n+1}$ (for example, we can identify $E+\sum_i{x_{i,i+n+1}E_{i,i+n+1}}\in \mathcal{F}_n$ with its image in $\mathcal{F}_n/\mathcal{F}_{n+1}$), the canonical maps $\mathcal{F}_n^{\gk}\to (\mathcal{F}_n/\mathcal{F}_{n+1})^{\gk}$ and $\mathcal{F}_n^{G_L}\to (\mathcal{F}_n/\mathcal{F}_{n+1})^{G_L}$ are surjective. By using long exact sequences, we have the commutative diagram
\[\xymatrix{
0\ar[r]&H^1(\gk,\mathcal{F}_{n+1})\ar[r]^(.53){\can.}\ar[d]^{\mathrm{Res}^L_K}&H^1(\gk,\mathcal{F}_n)\ar[r]^(.43){\can.}\ar[d]^{\mathrm{Res}^L_K}&H^1(\gk,\mathcal{F}_n/\mathcal{F}_{n+1})\ar[d]^{\mathrm{Res}^L_K}\\
0\ar[r]&H^1(G_L,\mathcal{F}_{n+1})\ar[r]^(.53){\can.}&H^1(G_L,\mathcal{F}_n)\ar[r]^(.43){\can.}&H^1(G_L,\mathcal{F}_n/\mathcal{F}_{n+1}),
}\]
whose rows are exact as pointed sets. To prove the assertion, it suffices to prove the injectivity of the restriction map $H^1(\gk,\U'_{h,a})\to H^1(G_L,\U'_{h,a})$ for $h,a\in\n$. In fact, it implies the injectivity of the right arrow in the diagram and we obtain the assertion by d\'evissage and diagram chasing. We first consider the case $a=0$, i.e., $\U'_{h,0}=\Q_{p^h}$ (Lemma~\ref{lem:eigen} below). Since $H^1(G_{L/K},\Q_{p^h}^{G_L})$ is killed by the multiplication by $[L:K]$ (using the corestriction), which induces an isomorphism on the coefficient, we have $H^1(G_{L/K},\Q_{p^h}^{G_L})=0$. By the inflation-restriction sequence, we obtain the assertion. Consider the case $a>0$. We denote by $\chi:\gk\to\zp^{\times}$ the cyclotomic character. Then, we obtain the assertion by the following commutative diagram
\[\xymatrix{
H^1(\gk,\U'_{h,a})\ar[d]^{\mathrm{Res}_K^L}\ar@{^(->}[rr]^(.32){\Pi(N^k\circ\varphi^{-n})_*}&&\prod_{n,k\in\n}H^1(\gk,\B_{\dr,\cp/K}^+)\cong \prod_{n,k\in\n}K\log{\chi}\ar@{^(->}[d]^{\Pi\mathrm{Res}_K^L}\\
H^1(G_L,\U'_{h,a})\ar@{^(->}[rr]^(.32){\Pi(N^k\circ\varphi^{-n})_*}&&\prod_{n,k\in\n}H^1(G_L,\B_{\dr,\cp/K}^+)\cong \prod_{n,k\in\n}L\log{\chi},
}\]
where two isomorphisms follow by d\'evissage and Lemma~\ref{lem:sp}, Theorem~\ref{thm:coh} (a theorem of J.~Tate) and the injectivities of the horizontal arrows follow from \cite[Proposition~0.4~(ii)]{Col}. 
\end{proof}

\begin{lem}\label{lem:eigen}
We have $(\wtil{\B}_{\rig,\cp/K_0}^{\nabla+})^{\varphi^h=p^{-a}}=(\wtil{\B}_{\log,\cp/K_0}^{\nabla+})^{\varphi^h=p^{-a}}=0$ for $a\in \n_{>0}$ and $(\wtil{\B}_{\rig,\cp/K_0}^{\nabla+})^{\varphi^h=1}=(\wtil{\B}_{\log,\cp/K_0}^{\nabla+})^{\varphi^h=1}=\Q_{p^h}$.
\end{lem}
\begin{proof}
We first prove the first assertion. Suppose that we have a non-zero element $x$ in $(\wtil{\B}_{\log,\cp/K_0}^{\nabla+})^{\varphi^h=p^{-a}}$. Since $\wtil{\B}_{\log,\cp/K_0}^{\nabla+}$ is an integral domain, we may assume that we have $x\in \A_{\st,\cp/K_0}^{\nabla}$ by multiplying a power of $p$. Since $\A_{\st,\cp/K_0}^{\nabla}$ is $\varphi$-stable, we have $\varphi^{nh}(x)=p^{-na}x\in\A_{\st,\cp/K_0}^{\nabla}$ for $n\in\n$. Since $\A_{\st,\cp/K_0}^{\nabla}$ is a lattice of $\B_{\st,\cp/K_0}^{\nabla+}$, we have $p^{-na}x\notin\A_{\st,\cp/K_0}^{\nabla}$ for sufficiently large $n$, which is a contradiction.

We prove the latter assertion. By a simple calculation, we have $(\wtil{\B}_{\log,\cp/K_0}^{\nabla+})^{\varphi^h=1}=(\wtil{\B}_{\rig,\cp/K_0}^{\nabla+})^{\varphi^h=1}$. By the canonical isomorphism $\wtil{\B}_{\rig,\cp/K_0}^{\nabla+}\cong \wtil{\B}_{\rig,\cp/\kpf_0}^+$, we may reduce to the perfect residue field case, which follows from Proposition~\cite[Proposition~8.15]{Col1}.
\end{proof}

\begin{lem}\label{lem:regp}
Let $D$ be a finite free $\B_{\dr,\cp/K}^+$-module with semi-linear $\gk$-action. Then, the canonical map $\B_{\dr,\cp/K}^+\otimes_K{D^{\gk}}\to D$ is injective. In particular, we have $\dim_KD^{\gk}\le\mathrm{rank}_{\B_{\dr,\cp/K}^+}D<\infty$.
\end{lem}
\begin{proof}
Suppose that we have linearly independent elements $f_1,\dotsc,f_n\in D^{\gk}$ over $K$, which have a non-trivial relation $\sum_i\lambda_if_i=0$ with $\lambda_i\in\B_{\dr,\cp/K}^+$. Choose the minimum $n$ among such $n$'s. Then, we have $g(\lambda_i/\lambda_1)=\lambda_i/\lambda_1$ in $\mathrm{Frac}(\B_{\dr,\cp/K})$ for $1\le i\le n$. Hence we have $\lambda_i/\lambda_1\in H^0(\gk,\mathrm{Frac}(\B_{\dr,\cp/K}))=K$ and $\sum_i(\lambda_i/\lambda_1)f_i=0$, i.e., a contradiction.
\end{proof}

\begin{lem}\label{lem:pot}
Let $W$ be an $r$-dimensional $\Q_{p^h}$-vector space with semi-linear $\gk$-action. For $0\le i< h$, we define a $\Q_{p^h}$-vector space $\varphi^i_*W$ with semi-linear $\gk$-action by $\varphi^i_*W:=W$ as $\gk$-module with scalar multiplication $\Q_{p^h}\times W\to W;(\lambda,x)\mapsto\varphi^i(\lambda) x$. If we have an isomorphism of $\B_{\dr,\cp/K}^+[\gk]$-modules $\B_{\dr,\cp/K}^+\otimes_{\Q_{p^h}}\varphi^i_*W\cong (\B_{\dr,\cp/K}^+)^r$ for $0\le i< h$, then $W$ is $\cp$-admissible as a $p$-adic representation of $\gk$.
\end{lem}
\begin{proof}
By assumption, we have isomorphisms $\B_{\dr,\cp/K}^+\otimes_{\qp}W\cong \oplus_{0\le i<h}\B_{\dr,\cp/K}^+\otimes_{\Q_{p^h}}\varphi^i_*W\cong (\B_{\dr,\cp/K}^+)^{hr}$ of $\B_{\dr,\cp/K}^+[\gk]$-modules, which implies the assertion by tensoring $\cp$ over $\B_{\dr,\cp/K}^+$.
\end{proof}

\begin{lem}\label{lem:pfpoincare}
Assume that $e_{K/K_{\can}}=1$. Then, the complex
\[\xymatrix{
K\otimes_{K_0}(\B_{\cris,\cp/K_0}^+)^{G_{K^{\pf}}}\ar[r]^{\nabla}&\Omega^1_K\hat{\otimes}_{K_0}(\B_{\cris,\cp/K_0}^+)^{G_{K^{\pf}}}\ar[r]^{\nabla_1}&\Omega^2_K\hat{\otimes}_{K_0}(\B_{\cris,\cp/K_0}^+)^{G_{K^{\pf}}},
}\]
which is induced by the inclusion $K\otimes_{K_0}\B_{\cris,\cp/K_0}^+\to\B_{\dr,\cp/K}^+$ (Proposition~\ref{prop:inj}) and Lemma~\ref{lem:poincare}, is exact. Here, we endow $(\B_{\cris,\cp/K_0}^+)^{G_{K^{\pf}}}$ with the $p$-adic topology induced by the $p$-adic semi-valuation $v_{\cris,\cp/K}$.
\end{lem}
\begin{proof}
Note that the connections are $K_{\can}$-linear by Proposition~\ref{prop:dlog}. By Remark~\ref{rem:can}~(ii) and Lemma~\ref{lem:dif}~(iii), we may reduces to the case $K=K_0$. Let $\omega\in\ker{\nabla_1}$. We can write $\omega=\sum_{j\in J_K}dt_j\otimes \lambda_j$ with $\lambda_j\in\B_{\cris,\cp/K}^+$ such that $\{v_{\cris,\cp/K}(\lambda_j)\}_{j\in J_K}\to \infty$. We can also write $\lambda_j=\sum_{\bm{n}\in\n^{\oplus J_K}}\lambda_{j,\bm{n}}\bm{u}^{[\bm{n}]}$ with $\lambda_{j,\bm{n}}\in \B_{\cris,\cp/\qp}^+$ such that $\{v_{\cris,\cp/\qp}(\lambda_{j,\bm{n}})\}_{\bm{n}\in\n^{\oplus J_K}}\to\infty$. Since $u_j$ is invariant under the action of $G_{K^{\pf}}$, we have $\lambda_{j,\bm{n}}\in (\B_{\cris,\cp/\qp}^+)^{G_{K^{\pf}}}$. Recall the proof of Lemma~\ref{lem:poincare}: We define $a_{\mathbf{0}}=0$ and $a_{\bm{n}}=\lambda_{j,\bm{n}-\mathbf{e}_j}$ if $n_j\neq 0$. Then, we have $x=\sum_{\bm{n}\in\n^{\oplus J_K}}a_{\bm{n}}\bm{u}^{[\bm{n}]}\in \B_{\dr,\cp/K}$ and $\nabla(x)=\omega$. Note that we have $x\in (\B_{\dr,\cp/K}^+)^{G_{K^{\pf}}}$. Hence, we have only to prove $x\in \B_{\cris,\cp/K}^+$. Fix $N\in\n$ and we will prove that we have $v_{\cris,\cp/K}(a_{\bm{n}})\ge N$ for all but finitely many $\bm{n}\in\n^{\oplus J_K}$. Choose a finite subset $J$ of $J_K$ such that we have $v_{\cris,\cp/K}(\lambda_j)\ge N$ for $j\in J_K\setminus J$. We also choose $n\in\n$ such that we have $v_{\cris,\cp/\qp}(\lambda_{j,\bm{n}})\ge N$ for $j\in J$ and $|\bm{n}|\ge n$. Let $\bm{n}\in\n^{\oplus J_K}\setminus \n^J$. Then, we have $v_{\cris,\cp/\qp}(a_{\bm{n}})=v_{\cris,\cp/\qp}(\lambda_{j,\bm{n}-\bm{e}_j})\ge N$ for some $j\in J_K\setminus J$. Let $\bm{n}\in\n^J$ with $|\bm{n}|> n$. Then, we have $v_{\cris,\cp/\qp}(a_{\bm{n}})=v_{\cris,\cp/\qp}(\lambda_{j,\bm{n}-\bm{e}_j})\ge v_{\cris,\cp/K}(\lambda_j)\ge N$ for some $j\in J$, where the first inequality follows from Proposition~\ref{lem:cris}. Since the set $\{\bm{n}\in\n^J||\bm{n}|\le n\}$ is finite, these inequalities imply the assertion.
\end{proof}

\begin{proof}[Proof of Main Theorem]
Obviously, we may assume $r:=\dim_{\qp}V>0$. By Hilbert~90, we may replace $K$ by $K^{\ur}$. Hence, we may assume that $k_K$ is separably closed. By Tate twist, we may also assume that $V$ satisfies the assumption of $\S~5$, i.e., we have $\D_{\dr}(V)=(\B_{\dr,\cp/K}^+\otimes_{\qp}V)^{\gk}$.

We divide the rest of the proof into two steps: We will construct a finite extension $L/K$ in Step~1 and after replacing $L$ by $K$, we will prove the semi-stability of $V$ in Step~2. Note that only Step~2 involves the choice of $K_0$.

Step~1: Let $\mathcal{M}:=\nrig(V)$ and $e_1,\dotsc,e_r$ be as in Proposition~\ref{prop:n}. Let $\{a'_1<\dotsb <a'_{r'}\}$ be the set of distinct elements of the multiset $\{a_1,\dotsc,a_r\}$ and $m_i$ be the multiplicity of $a'_i$ in the multiset for $1\le i\le r'$. Let $\{e^{(i)}_1,\dotsc,e^{(i)}_{m_i}\}$ be the subset of $\{e_1,\dotsc,e_r\}$ satisfying $\varphi^h(e_l)=p^{a'_i}e_l$. We define an exhaustive and separated increasing filtration of $\mathcal{M}$ by
\[
\mathcal{M}_n:=
\begin{cases}
0&n\le 0\\
\oplus_{1\le i\le n}(\Brig{\cp/K_0}e^{(i)}_1\oplus\dotsb\Brig{\cp/K_0}e^{(i)}_{m_i})&1\le n<r'\\
\mathcal{M}&\text{otherwise}.
\end{cases}
\]
The filtration is stable under $\varphi$ and $\gk$-actions. In fact, for $1\le i\le n<r'$ and $g\in\gk$, we have
\[
\varphi(e^{(i)}_1),\dotsc,\varphi(e^{(i)}_{m_i}),g(e^{(i)}_1),\dotsc,g(e^{(i)}_{m_i})\in\mathcal{M}^{\varphi^h=p^{a'_i}}\subset\mathcal{M}_n,
\]
where the last inclusion follows from Lemma~\ref{lem:eigen}. We also define $W_n:=(\mathcal{M}_n/\mathcal{M}_{n-1})^{\varphi^h=p^{a'_n}}$ for $1\le n\le r'$. Since we have $W_n=\Q_{p^h}\bar{e}^{(n)}_1\oplus\dotsb\oplus\Q_{p^h}\bar{e}^{(n)}_{m_n}$ by Lemma~\ref{lem:eigen}, where $\bar{e}^{(n)}_i$ denotes the image of $e^{(n)}_i$ in $\mathcal{M}_n/\mathcal{M}_{n-1}$, $W_n$ is an $m_n$-dimensional $\Q_{p^h}$-vector space with continuous semi-linear $\gk$-action. Let $D_n:=\B_{\dr,\cp/K}^+\otimes_{\Brig{\cp/K_0}}\mathcal{M}_n$. Then, we have the left exact sequence of finite $K$-vector spaces
\begin{equation}\label{eq:lex}
\xymatrix{
0\ar[r]&D_{n-1}^{\gk}\ar[r]^{\inc.}&D_n^{\gk}\ar[r]^(.35){\mathrm{pr}.}&(D_n/D_{n-1})^{\gk}.
}
\end{equation}
Hence, we have an inequalities
\[
\dim_K{D_n^{\gk}}\le\dim_K{D_{n-1}^{\gk}}+\dim_K{(D_n/D_{n-1})^{\gk}}\le\dim_K{D_{n-1}^{\gk}}+m_n
\]
for $n\in\Z$ by Lemma~\ref{lem:regp}. By Proposition~\ref{prop:horizontal}, we have an isomorphism of $\B_{\dr}^+[\gk]$-modules
\begin{equation}\label{eq:isom}
\B_{\dr,\cp/K}^+\otimes_{\Brig{\cp/K_0}}\mathcal{M}\cong(\B_{\dr,\cp/K}^+)^r,
\end{equation}
which implies $\dim_K{D_n^{\gk}}=r$ for $n\ge r'$. Hence, the summation of the above inequalities are equalities. Therefore, the above inequalities are equalities, in particular, the map $\mathrm{pr}.:D_n^{\gk}\to (D_n/D_{n-1})^{\gk}$ in (\ref{eq:lex}) is surjective. Thus, we have the commutative diagram
\[\xymatrix{
0\ar[r]&\B_{\dr,\cp/K}^+\otimes_K{D_{n-1}^{\gk}}\ar[r]\ar@{^(->}[d]&\B_{\dr,\cp/K}^+\otimes_K{D_n^{\gk}}\ar[r]\ar@{^(->}[d]&\B_{\dr,\cp/K}^+\otimes_K{(D_n/D_{n-1})^{\gk}}\ar[r]\ar@{^(->}[d]&0\\
0\ar[r]&D_{n-1}\ar[r]&D_n\ar[r]&D_n/D_{n-1}\ar[r]&0
}\]
with exact rows and injective vertical arrows by Lemma~\ref{lem:regp}. Since the middle vertical arrow is an isomorphism for $n\ge r'$ by (\ref{eq:isom}), all vertical arrows are isomorphisms. In particular, for $1\le n\le r'$, we have isomorphisms of $\B_{\dr,\cp/K}^+[\gk]$-modules $\B_{\dr,\cp/K}^+\otimes_{\Q_{p^h}}W_n\cong D_n/D_{n-1}\cong (\B_{\dr,\cp/K}^+)^{m_n}$. Since $W_n$ is stable under the action of $\varphi$, the map $W_n\to\varphi^i_*W_n;x\mapsto \varphi^i(x)$ is an isomorphism of $\Q_{p^h}[\gk]$-modules. In particular, we have isomorphisms of $\B_{\dr,\cp/K}^+[\gk]$-modules $\B_{\dr,\cp/K}^+\otimes_{\Q_{p^h}}\varphi^i_*W_n\cong\B_{\dr,\cp/K}^+\otimes_{\Q_{p^h}}W_n\cong (\B_{\dr,\cp/K}^+)^{m_n}$ for $1\le n\le r'$ and $0\le i<h$, which implies the $\cp$-admissibility of $W_n$ by Lemma~\ref{lem:pot}. Hence, $\gk$ acts on $W_n$ factoring through a finite quotient by Theorem~\ref{thm:Sen}. We choose a finite extension $L/K$ such that $G_L$ acts on $W_n$ trivially for all $1\le n\le r'$ and that $L$ satisfies Condition~(H).

Step~2: By replacing $V$ by $V|_L$, we will prove that $V$ is semi-stable by calculating Galois cohomology associated to $\nrig(V)$. In the following, we fix $K_0$ and a lift $\{t_j\}_{j\in J_K}$ of a $p$-basis of $k_K$ in $K_0$. We regard $\{t_j\}_{j\in J_K}$ as a lift of a $p$-basis of $k_K$ in $K$. We also fix notation: For a commutative ring $R$, let $U_r(R)\subset GL_r(R)$ be the set of upper triangular matrices, whose diagonals are equal to $1$. Let $N_r(R)\subset M_r(R)$ be the Lie algebra of $U_r(R)$, i.e., the set of upper triangular matrices, whose diagonals are equal to $0$. We denote $U_{r,\dr}^+:=U_r(\B_{\dr,\cp/K}^+)$, $U_{r,\dr}^{\nabla+}=U_r(\B_{\dr,\cp/K}^{\nabla+})$ for simplicity.

By assumption, we have $g(e_1,\dotsc,e_r)=(e_1,\dotsc,e_r)c_g$ with $1$-cocycle $c:\gk\to U_r(\Brig{\cp/K_0})$. Since we have $\nrig(V)\subset\Brig{\cp/K_0}\otimes_{\qp}V$ and $(K\otimes_{K_0}\B_{\st,\cp/K_0}^+\otimes_{\qp}V)^{\gk}=K\otimes_{K_0}(\B_{\st,\cp/K_0}^+\otimes_{\qp}V)^{\gk}$, we have only to prove that $c$ is a $1$-coboundary in $U_r(K\otimes_{K_0}\B_{\st,\cp/K_0}^+)$. We have the exact sequence of pointed sets
\begin{equation}\label{eq:ptd}
\xymatrix{
(U_{r,\dr}^+/U_{r,\dr}^{\nabla+})^{\gk}\ar[r]^{\delta}&H^1(\gk,U_{r,\dr}^{\nabla+})\ar[r]^{\inc._*}&H^1(\gk,U_{r,\dr}^+),
}
\end{equation}
where $U_{r,\dr}^+/U_{r,\dr}^{\nabla+}$ denotes the left coset of $U_{r,\dr}^+$ by $U_{r,\dr}^{\nabla+}$, i.e., $X\sim Y$ if $X^{-1}Y\in U_{r,\dr}^{\nabla+}$. The class $[c]\in H^1(\gk,U_{r,\dr}^{\nabla+})$ represented by $c$ vanishes in $H^1(\gk,U_{r,\dr}^+)$. In fact, since we have $\bar{e}_i^{(n)}\in (D_n/D_{n-1})^{\gk}$ for $1\le n\le r'$ and $1\le i\le m_n$ by assumption, there exists a unique element $\wtil{e}^{(n)}_i\in D_n^{\gk}$ such that $\wtil{e}_i^{(n)}-e_i^{(n)}\in D_{n-1}^{\gk}$ by the exactness of (\ref{eq:lex}). Then, $(\wtil{e}_1^{(1)},\dotsc,\wtil{e}_{m_1}^{(1)},\dotsc,\wtil{e}_1^{(n)},\dotsc,\wtil{e}_{m_n}^{(n)})$ is a $\B_{\dr,\cp/K}^+$-basis of $D_n$ for $1\le n\le r'$ and we have a unique matrix $U\in U_{r,\dr}^+$ such that
\[
(e_1^{(1)},\dotsc,e_{m_1}^{(1)},\dotsc,e_1^{(r')},\dotsc,e_{m_{r'}}^{(r')})=(\wtil{e}_1^{(1)},\dotsc,\wtil{e}_{m_1}^{(1)},\dotsc,\wtil{e}_1^{(r')},\dotsc,\wtil{e}_{m_{r'}}^{(r')})U.
\]
By a simple calculation, we have $c_g=U^{-1}g(U)$ for all $g\in \gk$. Hence, the class $[c]$ is represented by an element of the image of $(U_{r,\dr}^+/U_{r,\dr}^{\nabla+})^{\gk}$ under $\delta$ by the exact sequence (\ref{eq:ptd}). We regard $K\otimes_{K_0}\B_{\cris,\cp/K_0}^+$ as a subring of $\B_{\dr,\cp/K}^+$ by Proposition~\ref{prop:inj}. Then, we have the following lemma:

\begin{lem}\label{lem:key}
Every element of $(U_{r,\dr}^+/U_{r,\dr}^{\nabla+})^{\gk}$ is represented by an element in $U_r(K\otimes_{K_0}(\B_{\cris,\cp/K_0}^+)^{G_{K^{\pf}}})$.
\end{lem}
We leave the proof of Lemma~\ref{lem:key} to the end of the proof. Thanks to the lemma, there exist $X_1\in U_r(K\otimes_{K_0}(\B_{\cris,\cp/K_0})^{G_{K^{\pf}}})$ and $X_2\in U_{r,\dr}^{\nabla+}$ such that
\begin{equation}\label{eq:pf}
c_g=X_2^{-1}X_1^{-1}g(X_1)g(X_2)
\end{equation}
for all $g\in\gk$.

Since the canonical isomorphism $i:\wtil{\B}_{\rig,\cp/K_0}^{\nabla+}\to\wtil{\B}_{\rig,\cp/K_0^{\pf}}^+$ is compatible with the actions of $\varphi$ and $G_{K^{\pf}}$, we may regard $M:=i^*\mathcal{M}$ as a $(\varphi,G_{K^{\pf}})$-module over $\wtil{\B}_{\rig,\cp/K_0^{\pf}}^+$. Then, the triple $(M,\{e_1,\dotsc,e_r\},i^*c)$ satisfies the assumption of Proposition~\ref{prop:main}. In fact, the assumption~(i) follows from Proposition~\ref{prop:n}, Proposition~\ref{prop:horizontal} and the functoriality. The image of $c$ is in $U_r(\wtil{\B}_{\rig,\cp/K_0}^{\nabla+})$, which implies the assumption~(ii). By applying Proposition~\ref{prop:main} to the above triple, we have $X'_3\in U_r(\B_{\st,\cp/K^{\pf}_0}^+)$ such that $i(c_g)=(X'_3)^{-1}g(X'_3)$. Hence, $X_3:=i^{-1}(X'_3)\in U_r(\B_{\st,\cp/K_0}^{\nabla+})$ satisfies $c_g=X_3^{-1}g(X_3)$ for $g\in G_{K^{\pf}}$. Since we have $c_g=X_2^{-1}g(X_2)$ for $g\in G_{K^{\pf}}$ by (\ref{eq:pf}), we have $X_2X_3^{-1}\in (U_{r,\dr}^{\nabla+})^{G_{K^{\pf}}}=U_r((\B_{\dr,\cp/K}^{\nabla+})^{G_{K^{\pf}}})$. Note that the canonical map $K_{\can}\otimes_{K_{\can,0}}(\B_{\cris,\cp/K_0}^{\nabla+})^{G_{K^{\pf}}}\to (\B_{\dr,\cp/K}^{\nabla+})^{G_{K^{\pf}}}$ is an isomorphism. In fact, by using the canonical isomorphisms $\B_{\cris,\cp/K_0}^{\nabla+}\to\B_{\cris,\cp/K_0^{\pf}}^+$ and $\B_{\dr,\cp/K}^{\nabla+}\to\B_{\dr,\cp/K^{\pf}}^+$, it follows from the isomorphisms
\[
K_{\can}\otimes_{K_{\can,0}}K_0^{\pf}\cong K\otimes_{K_0}K_0^{\pf}\cong K^{\pf},
\]
where the first isomorphism easily follows from Remark~\ref{rem:can}~(ii) and the second one is trivial. Thus, we have
\[
c_g=(X_1\cdot X_2X_3^{-1}\cdot X_3)^{-1}g(X_1\cdot X_2X_3^{-1}\cdot X_3)
\]
for all $g\in \gk$ with $X_1$, $X_2X_3^{-1}$, $X_3\in U_r(K\otimes_{K_0}\B_{\st,\cp/K_0}^+)$, which implies the assertion.

Now, we return to the proof of Lemma~\ref{lem:key}. We endow $M_r(\B_{\dr,\cp/K}^+)\cong (\B_{\dr,\cp/K}^+)^{r^2}$ with the product topology. Let 
\begin{align*}
d&:M_r(\B_{\dr,\cp/K}^+)\to\om_K^1\hat{\otimes}_KM_r(\B_{\dr,\cp/K}^+);(x_{ij})\mapsto (\nabla(x_{ij})),\\
d_1&:\om_K^1\hat{\otimes}_KM_r(\B_{\dr,\cp/K}^+)\to \om_K^2\hat{\otimes}_KM_r(\B_{\dr,\cp/K}^+);(\omega_{ij})\mapsto (\nabla_1(\omega_{ij}))
\end{align*}
be derivations. We define a left (resp. right) action of $M_r(\B_{\dr,\cp/K}^+)$ on $\om^i_K\hat{\otimes}_KM_r(\B_{\dr,\cp/K}^+)$ for $i=1,2$ induced by the left (resp. right) multiplication on $M_r(\B_{\dr,\cp/K}^+)$. We also define a wedge product
\begin{gather*}
\wedge:\om^1_K\hat{\otimes}_KN_r(K)\times\om^1_K\hat{\otimes}_KN_r(\B_{\dr,\cp/K}^+)\to\om^2_K\hat{\otimes}_KN_r(\B_{\dr,\cp/K}^+)\\
(\omega_{ij})\times (\omega'_{ij})\mapsto \left(\sum_{1\le k\le r}\omega_{ik}\wedge \omega'_{kj}\right).
\end{gather*}
Then, we have the following formulas
\begin{gather*}
d_1\circ d=0,\\
d(XX')=dX\cdot X'+X\cdot dX',\ d_1(\omega\cdot X)=d_1\omega\cdot X-\omega\wedge dX,\\
(\omega\wedge\omega')\cdot X=\omega\wedge (\omega'\cdot X)
\end{gather*}
for $X$, $X'\in \om^1_K\hat{\otimes}_KN_r(\B_{\dr,\cp/K}^+)$, $\omega\in\om^1_K\hat{\otimes}_KN_r(K)$, $\omega'\in\om^1_K\hat{\otimes}_KN_r(\B_{\dr,\cp/K}^+)$. We define a log differential
\[
\mathrm{dlog}:U_{r,\dr}^+\to\om^1_K\hat{\otimes}_KN_r(\B_{\dr,\cp/K}^+);X\mapsto dX\cdot X^{-1},
\]
which is $\gk$-equivariant. (Note that it does not preserve the group laws in general.) Since we have $\mathrm{dlog}(XA)=\mathrm{dlog}(X)$ for $A\in U_{r,\dr}^{\nabla+}$ and $X\in U_{r,\dr}^+$ by the above formulas, $\mathrm{dlog}$ induces a morphism of $\gk$-sets $\mathrm{dlog}_*:U_{r,\dr}^+/U_{r,\dr}^{\nabla+}\to\om_K^1\hat{\otimes}_KN_r(\B_{\dr,\cp/K}^+)$. Moreover, $\mathrm{dlog}_*$ is injective. In fact, let $X,Y\in U_{r,\dr}^+$ such that $\mathrm{dlog}X=\mathrm{dlog}Y$. By $d\mathbf{E}=d(Y^{-1}Y)=0$ and the above formulas, we have $d(Y^{-1})=-Y^{-1}dY\cdot Y^{-1}$. Hence, we have
\[
\mathrm{dlog}(Y^{-1}X)=(d(Y^{-1})\cdot X+Y^{-1}dX)\cdot X^{-1}Y=-Y^{-1}(dY\cdot Y^{-1}-dX\cdot X^{-1})\cdot Y=0.
\]
Since the inverse image of $\{0\}$ by $\mathrm{dlog}$ is $U_{r,\dr}^{\nabla+}$, we have $X\sim Y$. By taking $H^0(\gk,-)$ of $\mathrm{dlog}_*$, we have an injection of sets
\[
\mathrm{dlog}_*:(U_{r,\dr}^+/U_{r,\dr}^{\nabla+})^{\gk}\hookrightarrow \om_K^1\hat{\otimes}_KN_r(K).
\]

We define a decreasing filtration on $N_r(\B_{\dr,\cp/K}^+)$ by
\[
\mathrm{Fil}^nN_r(\B_{\dr,\cp/K}^+):=\{(a_{ij})\in N_r(\B_{\dr,\cp/K}^+)|a_{ij}=0\ \mathrm{if}\ j-i\le n\}.
\]
Then, we have $\mathrm{Fil}^0N_r(\B_{\dr,\cp/K}^+)=N_r(\B_{\dr,\cp/K}^+)$ and $\mathrm{Fil}^rN_r(\B_{\dr,\cp/K}^+)=0$. Let $X\in U_{r,\dr}^+$ such that we have $[X]\in (U_{r,\dr}^+/U_{r,\dr}^{\nabla+})^{\gk}$. Let $\omega:=\mathrm{dlog}(X)\in \om^1_K\hat{\otimes}_KN_r(K)$. We will construct $X^{(n)}\in U_r(K\otimes_{K_0}(\B_{\cris,\cp/K_0})^{G_{K^{\pf}}})$ for $n\in\n$ satisfying $\omega\cdot X^{(n)}\equiv dX^{(n)}\mod{\om^1_K\hat{\otimes}_K\mathrm{Fil}^nN_r(\B_{\dr,\cp/K}^+)}$. Set $X^{(0)}:=1$. Suppose that we have constructed $X^{(n)}$. Since we have $\omega\cdot X=dX$, we have $d_1\omega\cdot X=\omega\wedge dX$ by taking $d_1$. Hence, we have $d_1\omega=(\omega\wedge dX)\cdot X^{-1}=\omega\wedge\omega$. Let $\omega'=(\omega'_{ij}):=\omega\cdot X^{(n)}-dX^{(n)}\in \om^1_K\hat{\otimes}_K\mathrm{Fil}^nN_r(\B_{\dr,\cp/K}^+)$. Then, by a simple calculation using the above formulas, we have
\[
d_1\omega'=\omega\wedge (\omega\cdot X^{(n)}-dX^{(n)})=\omega\wedge\omega'\equiv 0\mod{\om^2_K\hat{\otimes}_K\mathrm{Fil}^{n+1}N_r(\B_{\dr,\cp/K}^+)},
\]
which implies $\nabla_1(\omega'_{i,i+n+1})=0$. Since we have $\omega'_{ij}\in \om_K^1\hat{\otimes}_K(K\otimes_{K_0}(\B_{\cris,\cp/K_0}^+)^{G_{K^{\pf}}})$, by Lemma~\ref{lem:pfpoincare}, there exists $x'_{i,i+n+1}\in K\otimes_{K_0} (\B_{\cris,\cp/K_0}^+)^{G_{K^{\pf}}}$ such that $\nabla(x'_{i,i+n+1})=\omega'_{i,i+n+1}$. Let $X^{(n+1)}:=X^{(n)}+\sum_ix'_{i,i+n+1}E_{i,i+n+1}\in U_r(K\otimes_{K_0}(\B_{\cris,\cp/K_0})^{G_{K^{\pf}}})$. Then, by a simple calculation, we have
\begin{align*}
\omega\cdot X^{(n+1)}-dX^{(n+1)}&\equiv \omega\cdot X^{(n)}-dX^{(n)}-d\left(\sum_ix'_{i,i+n+1}E_{i,i+n+1}\right)\\
&\equiv\omega'-\sum_i\nabla(x'_{i,i+n+1})E_{i,i+n+1}\equiv 0\mod{\om_K\hat{\otimes}_K\mathrm{Fil}^{n+1}N_r(\B_{\dr,\cp/K}^+)}.
\end{align*}
Hence, we have $\mathrm{dlog}(X^{(r)})=\omega$, which implies the assertion.
\end{proof}

\section{Applications}\label{sec:app}
We will give applications of Main Theorem. In \S\S~\ref{subsec:additional}, we will recall linear algebraic structures, which appear in the following. In \S\S~\ref{subsec:horizontal}, we will prove a horizontal analogue of the $p$-adic monodromy theorem. The results of the next two subsections are applications of this theorem. In \S\S~\ref{subsec:Gal}, we will prove an equivalence between the category of horizontal de Rham representations of $\gk$ and the category of de Rham representation of $G_{K_{\can}}$. In \S\S~\ref{subsec:Hyo}, we will prove a generalization of Hyodo's theorem~\ref{thm:coh2}.

In this section, unless particular mention is stated, we will denote $\B_{\spadesuit,\cp/K_0}^{\nabla}$ (resp. $\B_{\clubsuit,\cp/K}^{\nabla}$) by $\B_{\spadesuit}^{\nabla}$ (resp. $\B_{\clubsuit}^{\nabla}$) for $\spadesuit\in\{\cris,\st\}$ (resp. $\clubsuit\in\{\dr,\HT\}$): This notation is justified by the facts that $\B_{\spadesuit,\cp/K_0}^{\nabla}$ and $\B_{\clubsuit,\cp/K}^{\nabla}$ are isomorphic to $\B_{\spadesuit,\cp/\qp}$ and $\B_{\clubsuit,\cp/\qp}$ as $(\qp,\gk)$-rings respectively.

\subsection{Additional structures}\label{subsec:additional}
In the following, let $V\in \rep_{\qp}\gk$. For $\bullet\in\{\cris,\st,\dr,\HT\}$, the vector space $\D_{\bullet}^{\nabla}(V)$ has an additional structure, which we will recall following \cite{Fon2}.

$\bullet$ The Hodge-Tate case

We define a graded $K$-vector space as a finite dimensional $K$-vector space $D$ endowed with a decomposition $D=\oplus_{n\in\Z}D_n$. Denote by $MG_K$ the category of graded $K$-vector spaces. The graded ring structure on $\B^{\nabla}_{\HT}$ induces a graded $K$-vector space structure on $\D_{\HT}^{\nabla}(V)$. Hence, we have a $\otimes$-functor
\[
\D_{\HT}^{\nabla}:\rep_{\HT}^{\nabla}\gk\to MG_K.
\]
Assume that we have $V\in\rep_{\HT}^{\nabla}\gk$. We define the Hodge-Tate weights of $V$ as the multiset consisting of $n\in\Z$ with multiplicity $m_n:=\dim_K{(\cp(-n)\otimes_{\qp}V)^{\gk}}$. Since the comparison isomorphism $\alpha_{\HT}^{\nabla}$ is compatible with $\gk$-actions and gradings, by taking the degree zero part, we have an isomorphism of $\cp[\gk]$-modules
\[
\cp\otimes_{\qp}V\cong \oplus_{n\in\Z}\cp(n)^{m_n},
\]
which is refered as the Hodge-Tate decomposition of $V$. Note that if $V\in\rep_{\qp}\gk$ admits such a decomposition, then it is horizontal Hodge-Tate.

$\bullet$ The de Rham case

We define a filtered $K_{\can}$-module as a finite dimensional $K_{\can}$-vector space endowed with a decreasing filtration $\{\mathrm{Fil}^n D\}_{n\in \Z}$ of $K_{\can}$-subspaces such that $\mathrm{Fil}^nD=D$ for $n\ll 0$ and $\mathrm{Fil}^nD=0$ for $n\gg 0$. Denote by $MF_{K_{\can}}$ the category of filtered $K_{\can}$-modules. The filtration $\mathrm{Fil}^n\B_{\dr}^{\nabla}=t^n\B_{\dr}^{\nabla+}$ on $\B_{\dr}^{\nabla}$ induces a filtered $K_{\can}$-module structure on $\D_{\dr}^{\nabla}(V)$. Hence, we have a $\otimes$-functor
\[
\D_{\dr}^{\nabla}:\rep_{\dr}^{\nabla}\gk\to MF_{K_{\can}}.
\]

$\bullet$ The cristalline and semi-stable cases

We first define filtered $(\varphi,N,G_{L/K})$-modules for our later use.

\begin{dfn}\label{dfn:fil}
\begin{enumerate}
\item[(i)] Let $L/K$ be a finite Galois extension. A filtered $(\varphi,N,G_{L/K})$-module is a finite dimensional $L_{\can,0}$-vector space $D$ endowed with
\begin{enumerate}
\item[$\bullet$] the Frobenius endomorphism: a bijective $\varphi$-semi-linear map $\varphi:D\to D$,
\item[$\bullet$] the monodromy operator: an $L_{\can,0}$-linear map $N:D\to D$ such that $N\varphi=p\varphi N$,
\item[$\bullet$] the Galois action: an $L_{\can,0}$-semi-linear action of $G_{L/K}$, which is commutative with $\varphi$ and $N$,
\item[$\bullet$] the filtration: a decreasing filtration $\{\mathrm{Fil}^nD_{L_{\can}}\}_{n\in\Z}$ of $G_{L/K}$-stable $L_{\can}$-subspaces of $D_{L_{\can}}:=L_{\can}\otimes_{L_{\can,0}}D$ satisfying
\[
\mathrm{Fil}^nD_{L_{\can}}=D_{L_{\can}}\ \mathrm{for}\ n\ll 0\ \text{and}\ \mathrm{Fil}^nD_{L_{\can}}=0\ \text{for}\ n\gg 0.
\]
\end{enumerate}
If $L=K$, then we call $D$ a filtered $(\varphi,N)$-module relative to $K_{\can}$. Moreover, if $N=0$, then we call $D$ a filtered $\varphi$-module relative to $K_{\can}$.

A morphism $D_1\to D_2$ of filtered $(\varphi,N,G_{L/K})$-modules is an $L_{\can,0}$-linear map $f:D_1\to D_2$ such that $f$ commutes with $\varphi$ and $N$, $G_{L/K}$-actions and we have $f(\mathrm{Fil}^nD_{1,L_{\can}})\subset \mathrm{Fil}^nD_{2,{L_{\can}}}$ for all $n\in\Z$.

Denote by $MF(\varphi,N,G_{L/K})$ (resp. $MF_{K_{\can}}(\varphi,N)$, $MF_{K_{\can}}(\varphi)$) the category of filtered $(\varphi,N,G_{L/K})$-modules (resp. filtered $(\varphi,N)$-modules relative to $K_{\can}$, filtered $\varphi$-modules relative to $K_{\can}$).
\item[(ii)] Let $D\in MF_{K_{\can}}(\varphi,N)$ and $r:=\dim_{K_{\can,0}}D$. We define $t_N(D)$ and $t_H(D)$ in the following way: First, we consider the case $r=1$. If we have $v\in D\setminus\{0\}$ and $\varphi(v)=\lambda v$, then $v_p(\lambda)\in\Z$ is independent of choices of $v$. We denote it by $t_N(D)$. We denote by $t_H(D)$ the maximum number $n\in\Z$ such that $\mathrm{Fil}^nD_{K_{\can}}\neq 0$. In general case, we define
\[
t_N(D):=t_N(\wedge^rD),\ t_H(D):=t_H(\wedge^rD).
\]

We say that $D$ is weakly admissible if we have $t_N(D)=t_H(D)$ and $t_N(D')\ge t_H(D')$ for any $K_{\can,0}$-subspace $D'$ of $D$ which is stable by $\varphi$ and $N$, with $D'_{K_{\can}}$ endowed with the induced filtration of $D_{K_{\can}}$.

Denote by $MF^{\mathrm{wa}}(\varphi,N,G_{L/K})$ the full subcategory of $MF(\varphi,N,G_{L/K})$, whose objects are weakly admissible as object of $MF_{L_{\can}}(\varphi,N)$. We also define $MF^{\mathrm{wa}}_{K_{\can}}(\varphi,N)$ and $MF_{K_{\can}}^{\mathrm{wa}}(\varphi)$ similarly.
\end{enumerate}
\end{dfn}

Let $\spadesuit\in\{\cris,\st\}$. By Proposition~\ref{prop:inj}, we have a $K_{\can}$-linear injection $K_{\can}\otimes_{K_{\can,0}}\D^{\nabla}_{\spadesuit}(V)\to\D_{\dr}^{\nabla}(V)$. We endow $K_{\can}\otimes_{K_{\can,0}}\D^{\nabla}_{\spadesuit}(V)$ with the induced filtration of $\D_{\dr}^{\nabla}(V)$. Together with the Frobenius endomorphism $\varphi$ and the monodromy operator $N$ on $\B_{\st}^{\nabla}$, these data induce a structure of a filtered $(\varphi,N)$-module over $K_{\can}$ relative to $K_{\can,0}$ on $\D_{\spadesuit}^{\nabla}(V)$. Since we have $\D_{\cris}^{\nabla}(V)=(\D_{\cris}^{\nabla}(V))^{N=0}$, $\D_{\cris}^{\nabla}(V)$ has a structure of a filtered $\varphi$-module over $K_{\can}$ relative to $K_{\can,0}$. Therefore, we have $\otimes$-functors
\[
\D_{\cris}^{\nabla}:\rep_{\cris}^{\nabla}\gk\to MF_{K_{\can}}(\varphi),\ \D_{\st}^{\nabla}:\rep_{\st}^{\nabla}\gk\to MF_{K_{\can}}(\varphi,N).
\]
For $D\in MF_{K_{\can}}(\varphi,N)$, we define
\[
\V_{\st}(D):=(\B_{\st}^{\nabla}\otimes_{K_{\can,0}}D)^{N=0,\varphi=1}\cap \mathrm{Fil}^0(\B_{\dr}^{\nabla}\otimes_{K_{\can}}D_{K_{\can}}).
\]
For $D\in MF_{K_{\can}}(\varphi)$, we define $\V_{\cris}(D):=\V_{\st}(D)$. These are (possibly infinite dimensional) $\qp$-vector spaces with $\gk$-action.

\begin{rem}\label{rem:hie}
Note that we have the following hierarchy of full subcategories of $\rep_{\qp}\gk$
\[
\rep_{\cris}^{\nabla}\gk\subset\rep_{\st}^{\nabla}\gk\subset\rep_{\dr}^{\nabla}\gk\subset\rep_{\HT}^{\nabla}\gk.
\]
In fact, if we have $V\in\rep_{\cris}^{\nabla}\gk$, then we have $\dim_{\qp}V=\dim_{K_{\can,0}}\D_{\cris}^{\nabla}(V)\le\dim_{K_{\can,0}}\D_{\st}^{\nabla}(V)$, which implies that $V$ is horizontal semi-stable by Lemma~\ref{lem:reg}. In the case, the canonical injection $\D_{\cris}^{\nabla}(V)\hookrightarrow\D_{\st}^{\nabla}(V)$ is an isomorphism as filtered $(\varphi,N)$-modules relative to $K_{\can}$. The inclusion $\rep_{\st}^{\nabla}\gk\subset\rep_{\dr}^{\nabla}\gk$ follows from Lemma~\ref{lem:rel} and Proposition~\ref{prop:inj}, Corollary~\ref{cor:inv}. Moreover, if we have $V\in\rep_{\st}^{\nabla}\gk$, then the canonical map $K_{\can}\otimes_{K_{\can,0}}\D_{\st}^{\nabla}(V)\to\D_{\dr}^{\nabla}(V)$ is an isomorphism of filtered $K_{\can}$-modules. Finally, let $V\in\rep_{\dr}^{\nabla}\gk$. We choose a lift of a $K_{\can}$-basis of $\mathrm{gr}^n\D_{\dr}^{\nabla}(V)$ in $\mathrm{Fil}^n\D_{\dr}^{\nabla}(V)$ for all $n\in\Z$. We denote these lifts by $\{e_i\}$ and let $n_i\in\Z$ such that $e_i\in\mathrm{Fil}^{n_i}\D_{\dr}^{\nabla}(V)\setminus\mathrm{Fil}^{n_i+1}\D_{\dr}^{\nabla}(V)$. Then, $\{e_i\}$ forms a $K_{\can}$-basis of $\D_{\dr}^{\nabla}(V)$. Consider the comparison isomorphism
\[
\B_{\dr}^{\nabla}\otimes_{K_{\can}}\D_{\dr}^{\nabla}(V)\to\B_{\dr}^{\nabla}\otimes_{\qp}V,
\]
which is compatible with the filtrations. By taking $\mathrm{Fil}^n$ of both sides, we have
\[
\sum_{i}t^{n-n_i}\B_{\dr}^{\nabla+}e_i=t^n\B_{\dr}^{\nabla+}\otimes_{\qp}V
\]
By taking $\mathrm{gr}^n$ of both side and taking $H^0(\gk,-)$, we have
\[
K\otimes_{K_{\can}}\mathrm{gr}^n\D^{\nabla}_{\dr}(V)\cong\bigoplus_{i:n_i=n}Ke_i\cong (\cp(n)\otimes_{\qp}V)^{\gk}
\]
by Theorem~\ref{thm:coh}. Hence, we have an isomorphism $K\otimes_{K_{\can}}\mathrm{gr}\D_{\dr}^{\nabla}(V)\cong \D_{\HT}^{\nabla}(V)$ of filtered $K$-vector spaces, which implies $V\in\rep_{\HT}^{\nabla}\gk$ by Lemma~\ref{lem:reg}. In particular, the multiset of Hodge-Tate weights of $V$ coincides with the multiset consisting of $n\in\Z$ with multiplicity $\dim_{K_{\can}}\mathrm{Fil}^{-n}\D_{\dr}^{\nabla}(V)/\mathrm{Fil}^{-n+1}\D_{\dr}^{\nabla}(V)$.
\end{rem}

\begin{prop}\label{prop:lin}
The functors $\D_{\cris}^{\nabla}$ and $\D_{\st}^{\nabla}$ induce the functors
\[
\D_{\cris}^{\nabla}:\rep_{\cris}^{\nabla}\gk\to MF_{K_{\can}}^{\mathrm{wa}}(\varphi),\ \D_{\st}^{\nabla}:\rep_{\st}^{\nabla}\gk\to MF_{K_{\can}}^{\mathrm{wa}}(\varphi,N).
\]
Moreover, these functors are fully faithful.
\end{prop}
\begin{proof}
We first prove the weak admissibilities of the images. As noted in Remark~\ref{rem:hie}, if $V$ is horizontal cristalline, then $\D_{\cris}^{\nabla}(V)$ coincides with $\D_{\st}^{\nabla}(V)$ as a filtered $(\varphi,N)$-module relative to $K_{\can}$. Therefore, we may reduces to the case that $V$ is horizontal semi-stable.

For a filtered $(\varphi,N)$-module $D$ relative to $K_{\can}$, we endow the finite $\kpf_0$-vector space $D_{\kpf_0}$ with a structure of a filtered $(\varphi,N)$-module relative to $K^{\pf}$ as follows. We extend the Frobenius $\varphi$ on $D$ to $D_{\kpf_0}$ semi-linearly and extend the monodromy operator $N$ on $D$ to $D_{\kpf_0}$ linearly. We also define a filtration of $D_{\kpf}$ as $\mathrm{Fil}^{\bullet}D_{\kpf}:=\kpf\otimes_{K_{\can}}\mathrm{Fil}^{\bullet}D_{K_{\can}}$. Moreover, the scalar extension
\[
K^{\pf}_0\otimes_{K_{\can,0}}(-): MF_{K_{\can}}(\varphi,N)\to MF_{\kpf}(\varphi,N)
\]
induces a functor. We have only to prove that the following diagram is commutative:
\[\xymatrix{
\rep_{\st}^{\nabla}\gk\ar[rr]^{\D_{\st}^{\nabla}}\ar[d]^{\mathrm{Res}^{K^{\pf}}_K}&&MF_{K_{\can}}(\varphi,N)\ar[d]^{K^{\pf}_0\otimes_{K_{\can,0}}(-)}\\
\rep_{\st}G_{\kpf}\ar[rr]^{\D_{\st}}&&MF_{\kpf}(\varphi,N).
}\]
In fact, since $\D_{\st}(V|_{K^{\pf}})=K^{\pf}_0\otimes_{K_{\can,0}}\D_{\st}^{\nabla}(V)$ is weakly admissible by \cite[Proposition~5.4.2~(i)]{Fon2}, $\D_{\st}^{\nabla}(V)$ is weakly admissible by definition.

By functoriality, the canonical map $i:K_0^{\pf}\otimes_{K_{\can,0}}\D_{\st}^{\nabla}(V)\to \D_{\st}(V|_{K^{\pf}})$ is a morphism of filtered $(\varphi,N)$-modules relative to $K^{\pf}_0$. Consider the associated graded homomorphism after applying $K^{\pf}\otimes_{K^{\pf}_0}$. The resulting homomorphism coincides with the canonical map $K^{\pf}\otimes_K\D_{\HT}^{\nabla}(V)\to\D_{\HT}(V|_{K^{\pf}})$. Since we have $V\in\rep_{\HT}^{\nabla}\gk$ by Remark~\ref{rem:hie}, a Hodge-Tate decomposition $\cp\otimes_{\qp}V\cong \oplus_{n\in\n}\cp(n)^{m_n}$ of $V$ induces a Hodge-Tate decomposition of $V|_{K^{\pf}}$. In particular, $V|_{K^{\pf}}$ is also Hodge-Tate and the above canonical map is an isomorphism. Since the filtrations of $\D_{\st}^{\nabla}(V)$ and $\D_{\st}(V|_{K^{\pf}})$ are separated and exhaustive, $i$ is an isomorphism as filtered $(\varphi,N)$-modules relative to $K^{\pf}_0$.

We prove the full faithfulness. We have the fundamental exact sequence
\[\xymatrix{
0\ar[r]&\qp\ar[r]^(.35){\inc.}&(\B_{\cris}^{\nabla})^{\varphi=1}\ar[r]^{\can.}&\B_{\dr}^{\nabla}/\B_{\dr}^{\nabla+}\ar[r]&0.
}\]
In fact, by identifying $\B_{\cris}^{\nabla}$ (resp. $\B_{\dr}^{\nabla+},\B_{\dr}^{\nabla}$) as $\B_{\cris,\cp/K_0^{\pf}}$ (resp. $\B_{\dr,\cp/K^{\pf}}^+,\B_{\dr,\cp/K^{\pf}}$), the exactness is reduced to \cite[Proposition~9.25]{CF}. By the fundamental exact sequence, we have $\V_{\st}\circ \D_{\st}^{\nabla}(V)=V$ (resp. $\V_{\cris}\circ \D_{\cris}^{\nabla}(V)=V$) for $V\in \rep_{\st}^{\nabla}{\gk}$ (resp. $V\in \rep_{\cris}^{\nabla}{\gk}$). This implies the full faithfulness.
\end{proof}

In Proposition~\ref{prop:CF}~(ii), we will prove that the above functors induce equivalences of categories, i.e, are essentially surjective.

\subsection{A horizontal analogue of the $p$-adic monodromy theorem}\label{subsec:horizontal}
The following is a horizontal analogue of the $p$-adic monodromy theorem. Note that the converse is true by Hilbert~90 and Corollary~\ref{cor:inv}.
\begin{thm}\label{thm:horizontal}
Let $V\in\rep_{\dr}^{\nabla}\gk$. Then, there exists a finite extension $K'/K_{\can}$ such that $V|_{KK'}$ is horizontal semi-stable.
\end{thm}
\begin{proof}
The comparison isomorphism $\alpha_{\dr,\cp/K}^{\nabla}$ induces an isomorphism of $\B_{\dr,\cp/K}[\gk]$-modules
\[
\B_{\dr,\cp/K}\otimes_{K_{\can}}\D_{\dr}^{\nabla}(V)\to\B_{\dr,\cp/K}\otimes_{\qp}V.
\]
By taking $H^0(\gk,-)$, we have $\dim_K\D_{\dr}(V)=\dim_{\qp}V$ by Corollary~\ref{cor:inv}, which implies $V\in\rep_{\dr}\gk$ by Lemma~\ref{lem:reg}. Hence, there exists a finite extension $L/K$ such that $V|_L$ is semi-stable by Main Theorem. We may assume that $L/K$ is a finite Galois extension satisfying Condition~(H) by the proof of Main Theorem (Step~1) and Epp's theorem~\ref{thm:Epp}. The extension $L_{\can}/K_\can$ is finite Galois by Lemma~\ref{lem:can}~(ii). We will prove the assertion for $K'=L_{\can}$.

We have canonical isomorphisms
\[
L_{\can}\otimes_{L_{\can,0}}\D_{\st}(V|_L)\cong L\otimes_{L_0}\D_{\st}(V|_L)\cong\D_{\dr}(V|_L),
\]
where the first one is induced by a canonical isomorphism $L_{\can}\otimes_{L_{\can,0}}L_0\to L$ (Remark~\ref{rem:can}~(ii)), the second one follows by using Lemma~\ref{lem:rel} and Proposition~\ref{prop:inj}. Moreover, these maps are commute with the residual $G_{L/K}$-actions and the $\nabla$-actions. By taking the horizontal sections, we have
\begin{align*}
\D_{\dr}^{\nabla}(V|_L)&=\D_{\dr}(V|_L)^{\nabla=0}=(L_{\can}\otimes_{L_{\can,0}}\D_{\st}(V|_L))^{\nabla=0}\\
&=L_{\can}\otimes_{L_{\can,0}}\D_{\st}(V|_L)^{\nabla=0}=L_{\can}\otimes_{L_{\can,0}}\D_{\st}^{\nabla}(V|_L),
\end{align*}
where the third equality follows from the fact $\nabla|_{L_{\can}}=0$. By taking $G_{L/K\cdot L_{\can}}$-invariants, we have $\D_{\dr}^{\nabla}(V|_{K\cdot L_{\can}})=L_{\can}\otimes_{L_{\can,0}}\D_{\st}^{\nabla}(V|_{K\cdot L_{\can}})$. Since $V|_{K\cdot L_{\can}}$ is horizontal de Rham by Remark~\ref{rem:res} and we have $(K\cdot L_{\can})_{\can}=L_{\can}$ (Lemma~\ref{lem:can}~(iv)), we have
\[
\dim_{L_{\can}}\D_{\dr}^{\nabla}(V|_{K\cdot L_{\can}})=\dim_{\qp}V=\dim_{L_{\can,0}}\D_{\st}^{\nabla}(V|_{K\cdot L_{\can}}),
\]
which implies that $V|_{K\cdot L_{\can}}$ is horizontal semi-stable.
\end{proof}

\subsection{Equivalences of categories}\label{subsec:Gal}
The surjection of profinite groups $\imath^*:G_K\to G_{K_{\can}}$ induces a $\otimes$-functor of Tannakian categories
\[
\imath^*:\rep_{\qp}G_{K_{\can}}\to \rep_{\qp}G_K
\]
Obviously, the functor $\imath^*$ is fully faithful. Denote by $\bm{C}_p$ the $p$-adic completion of the algebraic closure of $K_{\can}$ in $\kbar$. For $\bullet\in\{\cris,\st,\dr\}$, we have a Galois equivariant canonical injection $\B_{\bullet,\bm{C}_p/K_{\can}}\to \B^{\nabla}_{\bullet,\cp/K}$ by functoriality and we have $(\B_{\bullet,\bm{C}_p/K_{\can}})^{G_{K_{\can}}}\cong (\B^{\nabla}_{\bullet,\cp/K})^{\gk}$ ($=K_{\can}$ if $\bullet=\dr$, $K_{\can,0}$ otherwise) by Proposition~\ref{prop:inj}. Hence, if we have $V\in \rep_{\bullet}G_{K_{\can}}$, then we have $\imath^*V\in\rep_{\bullet}^{\nabla}\gk$. In fact, we have a canonical injection $\D_{\bullet}(V)\subset\D^{\nabla}_{\bullet}(\imath^*V)$ of $(\B_{\bullet,\bm{C}_p/K_{\can}})^{G_{K_{\can}}}$-vector spaces, which implies the $\B_{\bullet,\cp/K}^{\nabla}$-admissibility of $\imath^*V\in\rep_{\qp}\gk$ by Lemma~\ref{lem:reg}. Hence, $\imath^*$ induces a fully faithful $\otimes$-functor $\imath^*_{\bullet}:\rep_{\bullet}G_{K_{\can}}\to \rep^{\nabla}_{\bullet}G_K$.

The following proposition is a direct consequence of Colmez-Fontaine theorem.

\begin{prop}(A horizontal analogue of Colmez-Fontaine theorem (cf. \cite[Th\'eor\`eme~4.3]{CF}))\label{prop:CF}
\begin{enumerate}
\item[(i)] The functors $\imath^*_{\cris}$ and $\imath^*_{\st}$ are essentially surjective. In particular, $\imath^*_{\cris}$ and $\imath^*_{\st}$ induce equivalences of Tannakian categories.
\item[(ii)] The functors
\[
\D_{\cris}^{\nabla}:\rep_{\cris}^{\nabla}\gk\to MF^{\mathrm{wa}}_{K_{\can}}(\varphi),\ \D_{\st}^{\nabla}:\rep_{\st}^{\nabla}\gk\to MF^{\mathrm{wa}}_{K_{\can}}(\varphi,N)
\]
induce equivalences of categories with quasi-inverses $\V_{\cris}$, $\V_{\st}$.
\end{enumerate}
\end{prop}
\begin{proof}
We first prove the assertion in the semi-stable case. Together with the full faithfulness of $\D_{\st}^{\nabla}$, we have only to prove the commutativity of the diagram
\[\xymatrix{
\rep_{\st}G_{K_{\can}}\ar[r]^{\imath^*_{\st}}\ar[d]^{\D_{\st}}_{\cong}&\rep_{\st}^{\nabla}\gk\ar[d]^{\D_{\st}^{\nabla}}\\
MF^{\mathrm{wa}}_{K_{\can}}(\varphi,N)\ar[r]^{\mathrm{id}}&MF^{\mathrm{wa}}_{K_{\can}}(\varphi,N),
}\]
where $\D_{\st}$ is an equivalence of categories by Colmez-Fontaine theorem (\cite[Th\'eor\`eme~4.3]{CF}). As we mentioned above, the canonical map $\D_{\st}(V)\to\D_{\st}^{\nabla}(\imath^*V)$, which commutes with $\varphi$ and $N$-actions, is an isomorphism of $K_{\can,0}$-vector spaces. We have only to prove that the map also preserves the filtrations. Obviously, we have $\mathrm{Fil}^{\bullet}\D_{\st}(V)\subset \mathrm{Fil}^{\bullet}\D_{\st}^{\nabla}(\imath^*V)$. To prove the converse, it suffices to prove that the associated graded modules of both side have the same dimension since the filtrations are exhaustive and separated. Let $\bm{C}_p\otimes_{\qp}V\cong\oplus_{n\in\Z}\bm{C}_p(n)^{m_n}$ be the Hodge-Tate decomposition of $V$. Then, it induces the Hodge-Tate decomposition of $\imath^*V$, i.e., $\cp\otimes_{\qp}\imath^*V\cong \oplus_{n\in\Z}\cp(n)^{m_n}$, which implies the assertion.

In the horizontal cristalline case, a similar proof works by replacing $MF^{\mathrm{wa}}_{K_{\can}}(\varphi,N)$ and $*_{\st}$ by $MF_{K_{\can}}^{\mathrm{wa}}(\varphi)$ and $*_{\cris}$.
\end{proof}

\begin{thm}\label{thm:eqdR}
The functor $\imath^*_{\dr}$ is essentially surjective. In particular, $\imath^*_{\dr}$ induces an equivalence of Tannakian categories.
\end{thm}
\begin{proof}
For a finite Galois extension $L/K$ such that $K\cdot L_{\can}=L$, let $\mathcal{C}_{L/K}$ be the full subcategory of $\rep^{\nabla}_{\dr}G_K$, whose objects consist of $V\in\rep^{\nabla}_{\dr}G_K$ such that $V|_L$ is horizontal semi-stable. Recall the notation in Definition~\ref{dfn:fil}. Then, we have an equivalence of categories
\[
\D_{\st,L}^{\nabla}:\mathcal{C}_{L/K}\to MF^{\mathrm{wa}}(\varphi,N,G_{L/K});V\mapsto\D^{\nabla}_{\st}(V|_L).
\]
In fact, we have the following quasi-inverse $\V_{\st,L}$: For $D\in MF^{\mathrm{wa}}(\varphi,N,G_{L/K})$, we regard $D$ as an object of $MF^{\mathrm{wa}}_{L_{\can}}(\varphi,N)$ and let $\V_{\st,L}(D):=\V_{\st}(D)$. We have $\V_{\st,L}(D)\in\rep_{\st}^{\nabla}G_L$ by Proposition~\ref{prop:CF}~(ii) and $\V_{\st,L}(D)$ has a canonical $\gk$-action, which is an extension of the action of $G_L$, induced by the $G_{L/K}$-action on $D$. We have $D\in\mathcal{C}_{L/K}$ by Remark~\ref{rem:hil} and Remark~\ref{rem:hie}. We have $\V_{\st,L}\circ \D_{\st,L}^{\nabla}\cong \mathrm{id}_{\mathcal{C}_{L/K}}$ and $\D_{\st,L}^{\nabla}\circ\V_{\st,L}\cong\mathrm{id}_{MF^{\mathrm{wa}}(\varphi,N,G_{L/K})}$ by Proposition~\ref{prop:CF}~(ii).

The restriction map $\mathrm{Res}^L_{L_{\can}}:G_{L/K}\stackrel{\cong}{\to} G_{L_{\can}/K_{\can}}$ induces the equivalence of categories
\[\xymatrix{
(\mathrm{Res}^L_{L_{\can}})^*:MF^{\mathrm{wa}}(\varphi,N,G_{L_{\can}/K_{\can}})\ar[r]^(.6){\cong}&MF^{\mathrm{wa}}(\varphi,N,G_{L/K}).
}\]
We will prove that the following diagram
\[\xymatrix{
MF^{\mathrm{wa}}(\varphi,N,G_{L_{\can}/K_{\can}})\ar[rr]^{(\mathrm{Res}^L_{L_{\can}})^*}_{\cong}\ar[d]^{\V_{\st,L_{\can}}}_{\cong}&&MF^{\mathrm{wa}}(\varphi,N,G_{L/K})\ar[d]^{\V_{\st,L}^{\nabla}}_{\cong}\\
\mathcal{C}_{L_{\can}/K_{\can}}\ar[rr]^{\imath^*_{\dr}}&&\mathcal{C}_{L/K}
}\]
is commutative, where the bottom horizontal arrow is induced by $\imath^*_{\dr}:\rep_{\dr}G_{K_{\can}}\to \rep^{\nabla}_{\dr}G_K$. In fact, we have the $\gk$-equivariant inclusion
\[
\imath^*_{\dr}\circ\V_{\st,L_{\can}}(D)\subset \V_{\st,L}^{\nabla}\circ (\mathrm{Res}^L_{L_{\can}})^*(D)
\]
for $D\in MF^{\mathrm{wa}}(\varphi,N,G_{L_{\can}/K_{\can}})$ by construction. Since both sides have the same dimension over $\qp$, this inclusion is an equality. By the commutative diagram, the functor $\imath^*_{\dr}:\mathcal{C}_{L_{\can}/K_{\can}}\to\mathcal{C}_{L/K}$ is essentially surjective.

Let $V\in \rep_{\dr}^{\nabla}G_K$. By Theorem~\ref{thm:horizontal}, we have a finite Galois extension $K'/K_{\can}$ such that $V|_{G_{KK'}}$ is horizontal semi-stable. Let $L:=KK'$. By Lemma~\ref{lem:can}~(iv), we have $L_{\can}=K'$, i.e., $L/K$ satisfies the above assumption. Since we have $V\in \mathcal{C}_{L/K}$, the assertion follows from the essentially surjectivity of $\imath^*_{\dr}:\mathcal{C}_{L_{\can}/K_{\can}}\to\mathcal{C}_{L/K}$.
\end{proof}

The above equivalence induces a $\qp$-linear isomorphism of $\mathrm{Ext}^1$ on $\rep_{\dr}G_{K_{\can}}$ and $\rep_{\dr}^{\nabla}\gk$. Note that we may regard $\mathrm{Ext}^1_{\rep_{\dr}G_{K_{\can}}}(\qp,V)$ and $\mathrm{Ext}^1_{\rep_{\dr}^{\nabla}\gk}(\qp,\imath^*V)$ for $V\in\rep_{\dr}G_{K_{\can}}$ as
\begin{align*}
H^1_g(G_{K_{\can}},V)&:=\ker{(H^1(\gk,V)\stackrel{(1\otimes\mathrm{id})_*}{\to} H^1(\gk,\B_{\dr,\cp/K_{\can}}\otimes_{\qp}V))},\\
H^{1,\nabla}_g(\gk,\imath^*V)&:=\ker{(H^1(\gk,\imath^*V)\stackrel{(1\otimes\mathrm{id})_*}{\to} H^1(\gk,\B_{\dr,\cp/K}^{\nabla}\otimes_{\qp}\imath^*V))}
\end{align*}
respectively. In particular, we have

\begin{cor}\label{cor:ext}
For $V\in \rep_{\dr}G_{K_{\can}}$, the inflation map $\mathrm{Inf}:H^1(G_{K_{\can}},V)\to H^1(\gk,\imath^*V)$ induces the isomorphism
\[
\mathrm{Inf}: H^1_g(G_{K_{\can}},V)\cong H^{1,\nabla}_g(G_K,\imath^*V).
\]
\end{cor}

\subsection{A comparison theorem on $H^1$}\label{subsec:Hyo}
Notation is as in the previous subsection.

\begin{thm}[A generalization of Theorem~\ref{thm:coh2}]\label{thm:gen}
Let $V\in \rep_{\qp}G_{K_{\can}}$ be a de Rham representation, whose Hodge-Tate weights are greater than or equal to $1$. Then, we have the following exact sequence
\begin{equation}\label{eq:hyo1}
\xymatrix{
0\ar[r]&H^1(G_{K_{\can}},V)\ar[r]^{\mathrm{Inf}}&H^1(G_K,\imath^*V)\ar[r]^(.44){(1\otimes \mathrm{id})_*}&H^1(G_K,\cp\otimes_{\qp}\imath^*V)
}
\end{equation}
and a canonical isomorphism
\begin{equation}\label{eq:hyo2}
(\bm{C}_p\otimes_{\qp} V(-1))^{G_{K_{\can}}}\otimes_{K_{\can}}\hat{\Omega}^1_K\cong H^1(G_K,\cp\otimes_{\qp}\imath^*V).
\end{equation}
Moreover, if the Hodge-Tate weights of $V$ are greater than or equal to $2$, then $H^1(G_K,\cp\otimes_{\qp}\imath^*V)$ vanishes, in particular, the inflation map
\[
\mathrm{Inf}:H^1(G_{K_{\can}},V)\to H^1(G_K,\imath^*V)
\]
is an isomorphism.
\end{thm}
\begin{proof}
We first prove the exactness of (\ref{eq:hyo1}). Note that the injectivity of the inflation map follows by definition. We have the commutative diagram
\[\xymatrix{
H^1(G_{K_{\can}},V)\ar[d]^{(1\otimes \mathrm{id})_*}\ar[rd]^{(1\otimes \mathrm{id})_*\circ \mathrm{Inf}}&\\
H^1(G_{K_{\can}},\bm{C}_p\otimes_{\qp}V)\ar[r]^{\mathrm{Inf}}&H^1(G_K,\cp\otimes_{\qp}\imath^*V).
}\]
Since we have a Hodge-Tate decomposition $\bm{C}_p\otimes_{\qp}V\cong\oplus_{n\in\n_{\ge 1}}\bm{C}_p(n)^{m_n}$, we have $H^1(G_{K_{\can}},\bm{C}_p\otimes_{\qp}V)=0$ by Theorem~\ref{thm:coh}, which implies $(1\otimes \mathrm{id})_* \circ \mathrm{Inf}=0$.

Let $\mathcal{H}:=\ker{\{(1\otimes \mathrm{id})_*:H^1(G_K,\imath^*V)\to H^1(G_K,\cp\otimes_{\qp}\imath^*V)\}}$. We have only to prove that $\mathcal{H}$ is contained in the image of $\mathrm{Inf}:H^1(G_{K_{\can}},V)\to H^1(\gk,\imath^*V)$. Consider the exact sequence
\[\xymatrix{
0\ar[r]&t\B^{\nabla+}_{\dr,\cp/K}\ar[r]^{\mathrm{inc}.}&\B^{\nabla+}_{\dr,\cp/K}\ar[r]^(.6){\theta}&\cp\ar[r]&0
}\]
with $\theta:=\theta_{\cp/K}$. By applying $\otimes_{\qp}\imath^*V$ and taking $H^{\bullet}(G_K,-)$, we have the commutative diagram with exact row
\[\xymatrix{
&&H^1(G_K,\imath^*V)\ar[d]_{(1\otimes \mathrm{id})_*}\ar[dr]^{(1\otimes \mathrm{id})_*}&\\
H^1(G_K,t\B_{\dr,\cp/K}^{\nabla+}\otimes_{\qp}\imath^*V)\ar[rr]^{(\inc.\otimes \mathrm{id})_*}&&H^1(G_K,\B_{\dr,\cp/K}^{\nabla+}\otimes_{\qp}\imath^*V)\ar[r]^(.55){(\theta\otimes\mathrm{id})_*}&H^1(G_K,\cp\otimes_{\qp}\imath^*V).
}\]
Since $V(1)$ is de Rham with Hodge-Tate weights $\ge 2$, we have $H^1(G_K,t\B_{\dr,\cp/K}^{\nabla+}\otimes_{\qp}\imath^*V)=0$ by Theorem~\ref{thm:coh} and Lemma~\ref{lem:sp} and d\'evissage. Hence, the canonical map $(1\otimes \mathrm{id})_*:\mathcal{H}\to H^1(G_K,\B_{\dr,\cp/K}^{\nabla+}\otimes_{\qp}\imath^*V)$ vanishes by the above exact sequence. In particular, we have $\mathcal{H}\subset H^{1,\nabla}_g(G_K,\imath^*V)$. By Corollary~\ref{cor:ext}, we have $\mathrm{Inf}:H^1_g(G_{K_{\can}},V)\cong H^{1,\nabla}_g(G_K,\imath^*V)$, which implies the assertion~(\ref{eq:hyo2}).

Then, we will prove the existence of a canonical isomorphism. By the inclusion $(\bm{C}_p\otimes_{\qp}V(-1))^{G_{K_{\can}}}\subset (\cp\otimes_{\qp}\imath^*V(-1))^{\gk}$ and the canonical isomorphism $\hat{\Omega}^1_K\to H^1(\gk,\cp(1))$ in Theorem~\ref{thm:coh}, we can define our canonical map $f$ as the composition of
\[
(\bm{C}_p\otimes_{\qp}V(-1))^{G_{K_{\can}}}\otimes_{K_{\can}}\hat{\Omega}^1_K\stackrel{\inc.\otimes\can.}{\to}(\cp\otimes_{\qp}\imath^*V(-1))^{\gk}\otimes_K H^1(\gk,\cp(1))\stackrel{\mathrm{cup}.}{\to}H^1(\gk,\cp\otimes_{\qp}\imath^*V).
\]
We will prove that $f$ is an isomorphism. A Hodge-Tate decomposition of $V$ induces a Hodge-Tate decomposition $\cp\otimes_{\qp}\imath^*V\cong\oplus_{n\in\n_{\ge 1}}\cp(n)^{m_n}$ of $\imath^*V$. By replacing $\bm{C}_p\otimes_{\qp}V$ and $\cp\otimes_{\qp}\imath^*V$ by their Hodge-Tate decompositions, we may reduce to the case $V=\qp(n)$ with $n\in\n_{\ge 1}$ since the cup product commutes with the direct sum. Then, the assertion follows from Theorem~\ref{thm:coh}.

We will prove the last assertion. The assumption implies that we have $m_1=0$ in the above notation, hence, we have $H^1(\gk,\cp\otimes_{\qp}\imath^*V)=0$ by the Hodge-Tate decomposition of $\imath^*V$ and Theorem~\ref{thm:coh}.
\end{proof}

\begin{rem}
\begin{enumerate}
\item[(i)]Originally, Theorem~\ref{thm:coh2}~(i) and (ii) are proved separately by using ramification theory in some sense.
\item[(ii)] (Finiteness) Suppose that we have $[K_{\can}:\qp]<\infty$. For example, consider the case that $K$ has a structure of a higher dimensional local field (Example~\ref{ex:high}). Let $V\in \rep_{\qp}\gk$ be horizontal de Rham of Hodge-Tate weights greater than or equal to $2$. Then we have
\[
\dim_{\qp}H^1(\gk,V)=[K_{\can}:\qp]\dim_{\qp}V<\infty.
\]
In fact, by Theorem~\ref{thm:eqdR} and ~\ref{thm:gen}, we may reduce to the case $K=K_{\can}$. By a Hodge-Tate decomposition $\cp\otimes_{\qp}V\cong\oplus_{n\in\n_{\ge 2}}\cp(n)^{m_n}$ with $m_n\in\n$, we have $H^0(\gk,V)\subset H^0(\gk,\cp\otimes_{\qp}V)=0$ and $H^2(\gk,V)\cong H^0(\gk,V\spcheck (1))\subset H^0(\gk, \cp\otimes_{\qp}V\spcheck (1))=0$ by the local Tate duality (\cite[Th\'eor\`eme in Introduction]{Her}), where ${}\spcheck$ denotes the dual. Then, the assertion follows from Euler-Poincar\'e characteristic formula (loc. cit).

Note that $H^1(\gk,V)$ is not finite over $\qp$ without the condition on Hodge-Tate weights: For example, $H^1(\gk,\qp(1))\cong \qp\otimes_{\zp}\varprojlim_nK^{\times}/(K^{\times})^{p^n}$ contains $\qp\otimes_{\zp}\ok$, which is infinite dimensional over $\qp$ if $k_K$ is imperfect, via the map $\ok\hookrightarrow U_K^{(1)};x\mapsto \log{(1+2px)}$.
\end{enumerate}
\end{rem}

\noindent Department of Mathematical Sciences, University of Tokyo, Tokyo 153-8914, Japan

\noindent shuno@ms.u-tokyo.ac.jp
\end{document}